\documentclass[11pt,leqno,a4paper]{amsart}
\usepackage{amssymb,amscd,amsfonts,amsbsy,enumerate}
\usepackage{latexsym}
\usepackage{exscale}
\usepackage{amsmath,amsthm,amsfonts}
\usepackage{mathrsfs}

\usepackage[colorlinks=true, pdfstartview=FitV, linkcolor=blue, citecolor=red, urlcolor=blue]{hyperref}
\numberwithin{equation}{section}

\def\Xint#1{\mathchoice
   {\XXint\displaystyle\textstyle{#1}}%
   {\XXint\textstyle\scriptstyle{#1}}%
   {\XXint\scriptstyle\scriptscriptstyle{#1}}%
   {\XXint\scriptscriptstyle\scriptscriptstyle{#1}}%
   \!\int}
\def\XXint#1#2#3{{\setbox0=\hbox{$#1{#2#3}{\int}$}
     \vcenter{\hbox{$#2#3$}}\kern-.5\wd0}}
\def\aver#1{\Xint-_{#1}}

\allowdisplaybreaks

\newtheorem{theorem}{Theorem}[section]
\newtheorem{lemma}[theorem]{Lemma}
\newtheorem{corollary}[theorem]{Corollary}
\newtheorem{proposition}[theorem]{Proposition}
\newtheorem{definition}[theorem]{Definition}

\theoremstyle{remark}
\newtheorem{remark}[theorem]{Remark}

\DeclareMathOperator*{\essinf}{ess\,inf}
\DeclareMathOperator*{\esssup}{ess\,sup}

\begin{document}

\title[The Dirichlet problem with data in K\"othe function spaces]
{The Dirichlet problem for elliptic systems with \\[4pt]
data in K\"othe function spaces}


\author[J.M. Martell, D. Mitrea, I. Mitrea, and M. Mitrea]{Jos\'e Mar{\'\i}a Martell, Dorina Mitrea, Irina Mitrea, and Marius Mitrea}


\address{{\sc Jos\'e Mar{\'\i}a Martell}
\\
Instituto de Ciencias Matem\'aticas CSIC-UAM-UC3M-UCM
\\
Consejo Superior de Investigaciones Cient{\'\i}ficas
\\
C/ Nicol\'as Cabrera, 13-15
\\
E-28049 Madrid, Spain} 
\email{chema.martell@icmat.es}

\address{{\sc Dorina Mitrea}
\\
Department of Mathematics
\\
University of Missouri
\\
Columbia, MO 65211, USA} 
\email{mitread@missouri.edu}

\address{{\sc Irina Mitrea}
\\
Department of Mathematics
\\
Temple University\!
\\
1805\,N.\,Broad\,Street
\\
Philadelphia, PA 19122, USA}
\email{imitrea@temple.edu}

\address{{\sc Marius Mitrea}
\\
Department of Mathematics
\\
University of Missouri
\\
Columbia, MO 65211, USA}
\email{mitream@missouri.edu}





%
%

\begin{abstract}
We show that the boundedness of the Hardy-Littlewood maximal operator on a K\"othe
function space ${\mathbb{X}}$ and on its K\"othe dual ${\mathbb{X}}'$ is equivalent to
the well-posedness of the $\mathbb{X}$-Dirichlet and $\mathbb{X}'$-Dirichlet problems in
$\mathbb{R}^{n}_{+}$ in the class of all second-order, homogeneous, elliptic systems,
with constant complex coefficients. As a consequence, we obtain that the Dirichlet problem
for such systems is well-posed for boundary data in Lebesgue spaces, variable
exponent Lebesgue spaces, Lorentz spaces, Zygmund spaces, as well as their weighted versions.
We also discuss a version of the aforementioned result which contains, as a particular case,
the Dirichlet problem for elliptic systems with data in the classical Hardy space $H^1$, and
the Beurling-Hardy space ${\rm HA}^p$ for $p\in(1,\infty)$.
Based on the well-posedness of the $L^p$-Dirichlet problem we then prove the uniqueness
of the Poisson kernel associated with such systems, as well as the fact that they
generate a strongly continuous semigroup in natural settings. Finally, we establish
a general Fatou type theorem guaranteeing the existence of the pointwise nontangential
boundary trace for null-solutions of such systems.
\end{abstract}

%
%

\allowdisplaybreaks

\maketitle


\section{Introduction, Statement of Main Results, and Examples}
\setcounter{equation}{0}
\label{S-1}

Let $M\in{\mathbb{N}}$ be fixed and consider the second-order, homogeneous, $M\times M$ system,
with constant complex coefficients, written (with the usual convention of summation over
repeated indices in place) as
\begin{equation}\label{L-def}
Lu:=\Bigl(\partial_r(a^{\alpha\beta}_{rs}\partial_s u_\beta)
\Bigr)_{1\leq\alpha\leq M},
\end{equation}
when acting on a ${\mathscr{C}}^2$ vector-valued function $u=(u_\beta)_{1\leq\beta\leq M}$
defined in the upper-half space ${\mathbb{R}}^n_{+}:=\{(x',x_n)\in{\mathbb{R}}^{n-1}
\times{\mathbb{R}}:\,x_n>0\}$, $n\geq 2$. A standing assumption in this paper is
that $L$ is {\tt elliptic} in the sense that there exists a real number $\kappa_o>0$ such
that the following Legendre-Hadamard condition is satisfied:
\begin{equation}\label{L-ell.X}
\begin{array}{c}
{\rm Re}\,\bigl[a^{\alpha\beta}_{rs}\xi_r\xi_s\overline{\eta_\alpha}
\eta_\beta\,\bigr]\geq\kappa_o|\xi|^2|\eta|^2\,\,\mbox{ for every}
\\[8pt]
\xi=(\xi_r)_{1\leq r\leq n}\in{\mathbb{R}}^n\,\,\mbox{ and }\,\,
\eta=(\eta_\alpha)_{1\leq\alpha\leq M}\in{\mathbb{C}}^M.
\end{array}
\end{equation}
Two basic examples to keep in mind are the Laplacian $L:=\Delta$ in ${\mathbb{R}}^n$,
and the Lam\'e system
\begin{align}\label{TYd-YG-76g}
Lu:=\mu\Delta u+(\lambda+\mu)\nabla{\rm div}\,u,\qquad u=(u_1,\dots,u_n)\in{\mathscr{C}}^2,
\end{align}
where the constants $\lambda,\mu\in{\mathbb{R}}$ (typically called Lam\'e moduli)
are assumed to satisfy
\begin{align}\label{Yfhv-8yg}
\mu>0\quad\mbox{ and }\quad 2\mu+\lambda>0,
\end{align}
a condition actually equivalent to the demand that the Lam\'e system
\eqref{TYd-YG-76g} satisfies the Legendre-Hadamard ellipticity condition \eqref{L-ell.X}.

As is known from the seminal work of S.\,Agmon, A.\,Douglis, and L.\,Nirenberg in
\cite{ADNI} and \cite{ADNII}, every operator $L$ as in \eqref{L-def}-\eqref{L-ell.X} has
a Poisson kernel, denoted by $P^L$, an object whose properties mirror the most
basic characteristics of the classical harmonic Poisson kernel
\begin{equation}\label{Uah-TTT}
P^{\Delta}(x'):=\frac{2}{\omega_{n-1}}\frac{1}{\big(1+|x'|^2\big)^{\frac{n}{2}}},
\qquad\forall\,x'\in{\mathbb{R}}^{n-1},
\end{equation}
where $\omega_{n-1}$ is the area of the unit sphere $S^{n-1}$ in ${\mathbb{R}}^n$.
For details, see Theorem~\ref{ya-T4-fav} below. Here we only wish to note that, using the
notation $P_t(x'):=t^{1-n}P(x'/t)$ for each $t\in(0,\infty)$ and $x'\in{\mathbb{R}}^{n-1}$,
where $P$ is a generic function defined in ${\mathbb{R}}^{n-1}$, it follows that
there exists some $C\in(0,\infty)$ such that
\begin{equation}\label{tghn-jan-1}
|P^L_t(x')|\leq C\frac{t}{(t^2+|x'|^2)^{\frac{n}{2}}},
\qquad\forall\,x'\in{\mathbb{R}}^{n-1},\quad\forall\,t\in(0,\infty).
\end{equation}

The main goal of this paper is to establish well-posedness results for the Dirichlet problem
for a system $L$, as above, in ${\mathbb{R}}^n_{+}$ formulated in terms of certain types of
function spaces (made precise below).

Prior to formulating the most general result in this paper, some comments on the notation used
are in order. The symbol ${\mathcal{M}}$ is reserved for the Hardy-Littlewood maximal operator
in ${\mathbb{R}}^{n-1}$; see \eqref{MMax}.
Also, given a function $u$ defined in ${\mathbb{R}}^n_{+}$, by ${\mathcal{N}}u$ we shall denote
the nontangential maximal function of $u$; see \eqref{NT-Fct} for a precise definition.
Next, by $u\big|^{{}^{\rm n.t.}}_{\partial{\mathbb{R}}^n_{+}}$ we denote the nontangential
limit of the given function $u$ on the boundary of the upper half-space
(canonically identified with ${\mathbb{R}}^{n-1}$), as defined in \eqref{nkc-EE-2}.
Going further, denote by $\mathbb{M}$ the collection of all (equivalence classes of)
Lebesgue measurable functions $f:\mathbb{R}^{n-1}\rightarrow [-\infty,\infty]$ such that
$|f|<\infty$ a.e.~in ${\mathbb{R}}^{n-1}$.
Also, call a subset ${\mathbb{Y}}$ of $\mathbb{M}$ a {\tt function} {\tt lattice} if
the following properties hold:

\begin{enumerate}
\item[(i)] whenever $f,g\in{\mathbb{M}}$ satisfy $0\leq f\leq g$ a.e.~in
$\mathbb{R}^{n-1}$ and $g\in{\mathbb{Y}}$ then necessarily $f\in{\mathbb{Y}}$;
\item[(ii)] $0\leq f\in{\mathbb{Y}}$ implies $\lambda f\in{\mathbb{Y}}$ for every $\lambda\in(0,\infty)$;
\item[(iii)] $0\leq f,g\in{\mathbb{Y}}$ implies $\max\{f,g\}\in{\mathbb{Y}}$.
\end{enumerate}

In passing, note that, granted $({\rm i})$, one may replace ${\rm (ii)}$-${\rm (iii)}$
above by the condition: $0\leq f,g\in{\mathbb{Y}}$ implies $f+g\in{\mathbb{Y}}$.
As usual, we set $\log_{+}t:=\max\big\{0\,,\,\ln t\big\}$ for each $t\in(0,\infty)$.
Finally, we alert the reader that the notation employed does not always distinguish
between vector and scalar valued functions (which should be clear from context).

\begin{theorem}\label{Them-General}
Let $L$ be a system as in \eqref{L-def}-\eqref{L-ell.X}, and assume
that $\mathbb{X},\mathbb{Y}$ are arbitrary collections of measurable functions satisfying 
\begin{align}\label{Fi-AN.1}
& {\mathbb{X}}\subset
L^1\Big({\mathbb{R}}^{n-1}\,,\,\frac{1}{1+|x'|^n}\,dx'\Big),\qquad
{\mathbb{Y}}\subset
L^1\Big({\mathbb{R}}^{n-1}\,,\,\frac{1+\log_{+}|x'|}{1+|x'|^{n-1}}\,dx'\Big),
\\[6pt]
& \hskip 0.80in
\mbox{$\mathbb{Y}$ is a function lattice},\qquad
{\mathcal{M}}\mathbb{X}\subset\mathbb{Y}.
\label{Fi-AN.2}
\end{align}

Then the $(\mathbb{X},\mathbb{Y})$-Dirichlet boundary value problem for
$L$ in $\mathbb{R}^{n}_{+}$,
\begin{equation}\label{Dir-BVP-XY}
\left\{
\begin{array}{l}
u\in{\mathscr{C}}^\infty({\mathbb{R}}^n_{+}),
\\[4pt]
Lu=0\,\,\mbox{ in }\,\,\mathbb{R}^{n}_{+},
\\[4pt]
\mathcal{N}u\in\mathbb{Y},
\\[6pt]
u\bigl|_{\partial\mathbb{R}^{n}_{+}}^{{}^{\rm n.t.}}=f\in\mathbb{X},
\end{array}
\right.
\end{equation}
has a unique solution. Moreover, the solution $u$ of \eqref{Dir-BVP-XY}
is given by
\begin{equation}\label{eq:FDww}
u(x',t)=(P^L_t\ast f)(x')\,\,\,\mbox{ for all }\,\,\,(x',t)\in{\mathbb{R}}^n_{+},
\end{equation}
where $P^L$ is the Poisson kernel for $L$ in $\mathbb{R}^{n}_{+}$, and satisfies
\begin{equation}\label{eJHBb}
\mathcal{N} u(x')\leq C\,\mathcal{M} f(x'),
\qquad\forall\,x'\in\mathbb{R}^{n-1},
\end{equation}
for some constant $C\in[1,\infty)$ that depends only on $L$ and $n$.
\end{theorem}

Regarding the formulation of Theorem~\ref{Them-General}, we wish to note that
the first condition in \eqref{Fi-AN.1} is actually redundant, and we have only included
it for its pedagogical value (as it makes the proof of the existence of a solution for
\eqref{Dir-BVP-XY} most natural). Indeed, a more general result of this flavor holds, namely:
\begin{equation}\label{eq:jlk}
{\mathbb{X}}\subset{\mathbb{M}}\, \mbox{ and }\, 
{\mathcal{M}}f  \not\equiv\infty \mbox{ for each }f\in{\mathbb{X}}\,\Longrightarrow\,{\mathbb{X}}\subset
L^1\Big({\mathbb{R}}^{n-1} ,\,\frac{1}{1+|x'|^n}\,dx'\Big).
\end{equation}
Granted this, it is clear that the first inclusion in \eqref{Fi-AN.1} is implied
by the last condition in \eqref{Fi-AN.2} and the second condition in \eqref{Fi-AN.1}.
As regards the justification 
of \eqref{eq:jlk}, let $f\in{\mathbb{X}}$ be arbitrary.
Then the hypotheses in \eqref{eq:jlk} imply that there exists some
$x'_0\in{\mathbb{R}}^{n-1}$ such that ${\mathcal{M}}f(x'_0)<\infty$
in which case, for some finite constant $C=C(n,x'_0)>0$, we may estimate
\begin{equation}\label{eq:ERbv}
\int_{\mathbb{R}^{n-1}}|f(x')|\frac{1}{1+|x'|^n}\,dx'
\leq C\int_{\mathbb{R}^{n-1}}\frac{|f(x')|}{1+|x'-x'_0|^n}\,dx'
\leq C{\mathcal{M}}f(x'_0)<\infty,
\end{equation}
where the next-to-last inequality follows from a familiar dyadic annular decomposition argument
(in the spirit of \eqref{Drsgy-1jab}). Thus,  \eqref{eq:jlk} is true.

The particular case ${\mathbb{X}}={\mathbb{Y}}$ holds a special significance (in this vein,
see Theorem~\ref{Theorem-Nice} below). Incidentally, in this scenario the first condition
in \eqref{Fi-AN.1} is simply implied by the second condition in \eqref{Fi-AN.1} alone.
This being said, the case ${\mathbb{X}}\not={\mathbb{Y}}$ is natural to consider, as it arises
commonly in practice. For example, the Dirichlet problem \eqref{Dir-BVP-XY} is well-posed for any
system $L$ as in \eqref{L-def}-\eqref{L-ell.X} provided, for a given $p\in(1,\infty)$,
\begin{equation}\label{eq:BVCc}
{\mathbb{X}}:=L^1({\mathbb{R}}^{n-1})\cap L^p({\mathbb{R}}^{n-1})
\,\,\,\mbox{ and }\,\,\,{\mathbb{Y}}:=L^{1,\infty}({\mathbb{R}}^{n-1})\cap L^p({\mathbb{R}}^{n-1}),
\end{equation}
since conditions \eqref{Fi-AN.1}-\eqref{Fi-AN.2} are easily verified in this case.
We stress that in the formulation of Theorem~\ref{Them-General} the set ${\mathbb{X}}$
is not required to be a function lattice, and this is a relevant observation for
the $(H^1,L^1)$-Dirichlet problem discussed below in Corollary~\ref{tfav.tRR}
(cf. also Corollary~\ref{tfav.tRR.BEUR} for a similar phenomenon).

The proof of Theorem~\ref{Them-General} in \S\ref{S-4} makes strong use of the
results established in \S\ref{S-3}. More specifically, the second inclusion in \eqref{Fi-AN.1}
ensures (keeping in mind the function lattice property for ${\mathbb{Y}}$)
the applicability of Theorem~\ref{thm:uniqueness}, which yields uniqueness.
The first inclusion in \eqref{Fi-AN.1} guarantees the applicability of Theorem~\ref{thm:existence},
which shows that the function $u$ as in \eqref{eq:FDww} belongs to 
${\mathscr{C}}^\infty({\mathbb{R}}^n_{+})$ and satisfies $Lu=0$ in ${\mathbb{R}}^n_{+}$
as well as $u\bigl|_{\partial\mathbb{R}^{n}_{+}}^{{}^{\rm n.t.}}=f$ and \eqref{eJHBb}. 
Granted the latter property (and bearing in mind the function lattice property for
${\mathbb{Y}}$), the last condition in \eqref{Fi-AN.2} then guarantees that
\begin{equation}\label{eq:BV333}
\mbox{$f\in\mathbb{X}$ and $u$ as in \eqref{eq:FDww}}\,\Longrightarrow\,\mathcal{N}u\in\mathbb{Y}.
\end{equation}
This proves existence in Theorem~\ref{Them-General}.
It is worth noting that ${\mathcal{M}}{\mathbb{X}}\subset{\mathbb{Y}}$ may be replaced
in the formulation of Theorem~\ref{Them-General} (without affecting the conclusions)
by the weaker condition \eqref{eq:BV333}. This is significant, because the latter holds
even though the former fails in the important case of the Dirichlet problem with data from the
Hardy space, when
\begin{equation}\label{eq:BVCc.2}
{\mathbb{X}}:=H^1({\mathbb{R}}^{n-1})\,\,\,\mbox{ and }\,\,\,{\mathbb{Y}}:=L^1({\mathbb{R}}^{n-1}).
\end{equation}
This permits us to prove (see \S\ref{S-4} for details) the following well-posedness result.

\begin{corollary}\label{tfav.tRR}
The $(H^1,L^1)$-Dirichlet boundary value problem
in $\mathbb{R}^{n}_{+}$ is well-posed for each system $L$ as in \eqref{L-def}-\eqref{L-ell.X}.
\end{corollary}

In fact, the weaker condition in the left-hand side of
\eqref{eq:BV333} is also relevant in other scenarios such as
the Dirichlet problem with data from the Beurling-Hardy space, when
\begin{equation}\label{eq:BVCc.2BEUR}
{\mathbb{X}}:={\rm HA}^p({\mathbb{R}}^{n-1})\,\,\mbox{ and }\,\,{\mathbb{Y}}:={\rm A}^p({\mathbb{R}}^{n-1})
\,\,\mbox{ for some }\,\,p\in(1,\infty).
\end{equation}
Above, ${\rm A}^p({\mathbb{R}}^{n-1})$ is the classical (convolution) algebra introduced by A.\,Beurling
in \cite{Beu}, while ${\rm HA}^p({\mathbb{R}}^{n-1})$ is the Hardy space associated with the Beurling
algebra ${\rm A}^p({\mathbb{R}}^{n-1})$ as in \cite{GC} (following work in the complex plane in \cite{CL}).
For concrete definitions the reader is referred to \S\ref{S-4}, where the proof of the following
well-posedness result may also be found.

\begin{corollary}\label{tfav.tRR.BEUR}
For each $p\in(1,\infty)$, the $({\rm HA}^p,{\rm A}^p)$-Dirichlet boundary value problem
in $\mathbb{R}^{n}_{+}$ is well-posed whenever $L$ is a system as in \eqref{L-def}-\eqref{L-ell.X}.
\end{corollary}

As is apparent from the statement of Theorem~\ref{Them-General}, devising practical ways
for checking the validity of the inclusions in \eqref{Fi-AN.1} becomes a significant issue
that deserves further attention. One natural, and also general, setting where the named
inclusions may be equivalently rephrased as the membership of the intervening weight
functions to dual spaces is that of K\"othe function spaces. Since the latter class
of function spaces plays a significant role for us here, we proceed to summarize their
definition and basic properties (more details may be found in Bennett and
Sharpley~\cite{bennett-sharpley88} where the terminology employed is that of
Banach function spaces; cf.~also \cite{CMP}, \cite{LT}, \cite{Zaa}). Specifically,
call a mapping $\|\cdot\|:\mathbb{M}\rightarrow[0,\infty]$ a {\tt function} {\tt norm}
provided the following properties are satisfied for all $f,\,g\in\mathbb{M}$:

\begin{list}{$(\theenumi)$}{\usecounter{enumi}\leftmargin=1cm
\labelwidth=1cm\itemsep=0.2cm\topsep=.2cm
\renewcommand{\theenumi}{\alph{enumi}}}

\item[(1)] $\|f\|=\big\||f|\big\|$, and $\|f\|=0$
if and only if $f=0$ a.e.~in $\mathbb{R}^{n-1}$;

\item[(2)] $\|f+g\|\leq\|f\|+\|g\|$, and $\|\lambda f\|=|\lambda|\,\|f\|$ for each $\lambda\in\mathbb{R}$;

\item[(3)] if  $|f|\leq |g|$ a.e.~in $\mathbb{R}^{n-1}$ then $\|f\|\leq\|g\|$;

\item[(4)] if $\{f_k\}_{k\in{\mathbb{N}}}\subset\mathbb{M}$ is a sequence such that $|f_k|$
increases to $|f|$ pointwise a.e.~in $\mathbb{R}^{n-1}$ as $k\to\infty$, then $\|f_k\|$
increases to $\|f\|$ as $k\to\infty$;

\item[(5)] if $E\subset\mathbb{R}^{n-1}$ is a measurable set of finite measure
then its characteristic function ${\bf 1}_E$ satisfies $\|{\bf 1}_E\|<\infty$,
and $\int_E |f(x')|\,dx'\leq C_E\|f\|$ where $C_E<\infty$ depends on $E$, but not on $f$.
\end{list}

Given a function norm $\|\cdot\|$, the set
\begin{equation}\label{Tfaca}
\mathbb{X}:=\big\{f\in\mathbb{M}:\,\|f\|<\infty\big\}
\end{equation}
is referred to as a {\tt K\"othe} {\tt function} {\tt space} on $({\mathbb{R}}^{n-1},dx')$.
In such a scenario, we shall write
$\|\cdot\|_\mathbb{X}$ in place of $\|\cdot\|$ in order to emphasize the connection between
the function norm $\|\cdot\|$ and its associated K\"othe function space $\mathbb{X}$.
Then $\big(\mathbb{X},\|\cdot\|_\mathbb{X}\big)$ is a complete normed vector subspace of $\mathbb{M}$,
hence a Banach space. It is apparent from the above definitions that many of the classical
function spaces in analysis are actually K\"othe function spaces. This includes ordinary 
Lebesgue spaces, variable exponent Lebesgue spaces, Orlicz spaces, Lorentz spaces, mixed-normed spaces, 
Marcinkiewicz spaces, Morrey spaces, etc. Typically, function spaces whose definitions take into 
account cancellation or differentiability properties of the functions, such as Hardy spaces, BMO, 
Sobolev spaces, etc., fail to be K\"othe function spaces.

\smallskip

Starting with a K\"othe function space $\mathbb{X}$, we can define its K\"othe dual (also known as
its associate space in the terminology of \cite{bennett-sharpley88}) according to
\begin{equation}\label{Ytr44}
\begin{array}{c}
\mbox{$\mathbb{X}':=\big\{f\in{\mathbb{M}}:\,\|f\|_{{\mathbb{X}}'}<\infty\big\}$}
\,\,\mbox{ where, for each }\,f\in\mathbb{M},
\\[8pt]
\displaystyle
\|f\|_{\mathbb{X}'}:=\sup\left\{\int_{\mathbb{R}^{n-1}} |f(x')\,g(x')|\,dx':\,
g\in\mathbb{X},\,\,\|g\|_\mathbb{X}\leq 1\right\}.
\end{array}
\end{equation}
One can check that $\|\cdot\|_{\mathbb{X}'}$ is indeed a function norm, hence
$\mathbb{X}'$ is itself a K\"othe function space.

An immediate consequence of the above definitions is the generalized
H\"older's inequality:
\begin{equation}\label{Holder-X}
\int_{\mathbb{R}^{n-1}} |f(x')\,g(x')|\,dx'
\leq\|f\|_\mathbb{X}\,\|g\|_{\mathbb{X}'},\quad\mbox{ for all }\,\,f\in\mathbb{X},\,\,g\in\mathbb{X}'.
\end{equation}
In this regard, let us also record here the following characterization of the K\"othe dual
given in \cite[Lemma~2.6, p.\,10]{bennett-sharpley88}:
\begin{equation}\label{YtrBS}
\mathbb{X}'=\left\{g\in{\mathbb{M}}:\,\int_{\mathbb{R}^{n-1}}|f(x')\,g(x')|\,dx'<\infty
\,\,\mbox{ for each }\,\,f\in\mathbb{X}\right\}.
\end{equation}
Moreover,
\begin{equation}\label{dual-X2}
(\mathbb{X}')'=\mathbb{X},
\end{equation}
i.e., the K\"othe dual space of $\mathbb{X}'$ is again
$\mathbb{X}$. As a consequence, the function norm on $\mathbb{X}$ may be expressed in terms of the
function norm on $\mathbb{X}'$ according to
\begin{equation}\label{dual-X}
\|f\|_\mathbb{X}=\sup\left\{\int_{\mathbb{R}^{n-1}}|f(x')\,g(x')|\,dx':\,
g\in\mathbb{X}',\,\|g\|_{\mathbb{X}'}\leq 1\right\},\qquad\forall\,f\in{\mathbb{X}}.
\end{equation}
For further reference it will be of interest to note that
\begin{equation}\label{Lab-1}
{\bf 1}_E\in\mathbb{X}\cap\mathbb{X}'\,\,\mbox{ if
$E\subset\mathbb{R}^{n-1}$ is a measurable set of finite measure},
\end{equation}
and
\begin{equation}\label{Lab-2}
\mathbb{X}\subset L^1_{\rm loc}({\mathbb{R}}^{n-1}),\qquad
\mathbb{X}'\subset L^1_{\rm loc}({\mathbb{R}}^{n-1}).
\end{equation}

The key observation is that whenever $\mathbb{X},\mathbb{Y}$ are K\"othe function spaces
then $\mathbb{Y}$ is a function lattice by design and, thanks to \eqref{YtrBS},
the inclusions in \eqref{Fi-AN.1} are equivalent to the memberships
\begin{equation}\label{Ygag-1}
\frac{1}{1+|x'|^n}\in{\mathbb{X}}'\quad\mbox{ and }\quad
\frac{1+\log_{+}|x'|}{1+|x'|^{n-1}}\in{\mathbb{Y}}'.
\end{equation}
Furthermore, if the last condition in \eqref{Fi-AN.2} is strengthened to
\begin{equation}\label{Ygag-2}
{\mathcal{M}}:\mathbb{X}\longrightarrow\mathbb{Y}\,\,\mbox{ boundedly,}
\end{equation}
then by \eqref{eJHBb} and the monotonicity of the function norm in ${\mathbb{Y}}$ it follows that
there exists a constant $C=C(n,L,\mathbb{X},\mathbb{Y})\in(0,\infty)$ with the property that
the solution $u$ of \eqref{Dir-BVP-XY} satisfies
\begin{equation}\label{Dir-BVP-X:bounds}
\|\mathcal{N} u\|_{\mathbb{Y}}\leq C\,\|f\|_{\mathbb{X}}.
\end{equation}
One convenient practical way of ensuring that \eqref{Ygag-1} holds is to check that
${\mathcal{M}}$ is bounded on $\mathbb{X}'$ and $\mathbb{Y}'$. This is a consequence
of \eqref{Lab-1} and Lemma~\ref{lemma:M-ball}, in the body of the paper.

In the important special case of K\"othe function spaces satisfying $\mathbb{X}=\mathbb{Y}$,
the first condition in \eqref{Ygag-1} becomes redundant (as it is implied by the second).
In this scenario, if
\begin{equation}\label{Ygag-1bis}
\frac{1+\log_{+}|x'|}{1+|x'|^{n-1}}\in{\mathbb{X}}'
\,\,\,\mbox{ and }\,\,\,\mathcal{M}\mathbb{X}\subset\mathbb{X}
\end{equation}
then the $\mathbb{X}$-Dirichlet boundary value problem for $L$ in
$\mathbb{R}^{n}_{+}$, formulated as in \eqref{Dir-BVP-XY} with $\mathbb{Y}=\mathbb{X}$,
is well-posed. Moreover,
\begin{equation}\label{Ygag-2bis}
\mbox{${\mathcal{M}}$ bounded on $\mathbb{X}$ implies
$\|\mathcal{N} u\|_{\mathbb{X}}\leq C\,\|f\|_{\mathbb{X}}$}.
\end{equation}
Let us also note here that, as seen from \eqref{Ygag-1bis} and Lemma~\ref{lemma:M-ball},
the first condition in \eqref{Ygag-1bis} may also be expressed in terms of the Hardy-Littlewood
maximal operator as
\begin{equation}\label{eq:12N34}
{\mathcal{M}}^{(2)}\big({\bf 1}_{B_{n-1}(0',1)}\big)\in{\mathbb{X}}',
\end{equation}
where ${\mathcal{M}}^{(2)}$ is the two-fold composition of ${\mathcal{M}}$ with itself,
and where $B_{n-1}(0',1)$ denotes the $(n-1)$-dimensional Euclidean ball of radius $1$
centered at the origin $0'=(0,\dots,0)\in{\mathbb{R}}^{n-1}$. In particular,
\begin{equation}\label{eq:12N35}
\mbox{if ${\mathcal{M}}$ is bounded on ${\mathbb{X}}'$ then
the first condition in \eqref{Ygag-1bis} holds.}
\end{equation}

As a consequence of the above considerations, we have the following notable result
showing that the boundedness of the Hardy-Littlewood maximal operator on ${\mathbb{X}}$ and
${\mathbb{X}}'$ is equivalent to the well-posedness of the $\mathbb{X}$-Dirichlet
and $\mathbb{X}'$-Dirichlet boundary value problems in $\mathbb{R}^{n}_{+}$ for the
class of all second-order, homogeneous, elliptic systems, with constant complex coefficients. 

\begin{theorem}\label{Theorem-Nice}
Assume that $L$ is a system as in \eqref{L-def}-\eqref{L-ell.X}, and
suppose $\mathbb{X}$ is a K\"othe function space such that
\begin{equation}\label{Nabb}
\mbox{${\mathcal{M}}$ is bounded both on $\mathbb{X}$ and $\mathbb{X}'$.}
\end{equation}

Then the $\mathbb{X}$-Dirichlet boundary value problem for $L$ in $\mathbb{R}^{n}_{+}$,
\begin{equation}\label{Dir-BVP-Xalone}
\left\{
\begin{array}{l}
u\in{\mathscr{C}}^\infty(\mathbb{R}^{n}_{+}),
\\[4pt]
Lu=0\,\,\mbox{ in }\,\,\mathbb{R}^{n}_{+},
\\[4pt]
\mathcal{N}u\in\mathbb{X},
\\[6pt]
u\big|_{\partial\mathbb{R}^{n}_{+}}^{{}^{\rm n.t.}}=f\in\mathbb{X},
\end{array}
\right.
\end{equation}
is well-posed. In addition, the solution $u$ of \eqref{Dir-BVP-Xalone}
is given by $u(x',t)=(P^L_t\ast f)(x')$ for all $(x',t)\in{\mathbb{R}}^n_{+}$,
where $P^L$ is the Poisson kernel for $L$ in $\mathbb{R}^{n}_{+}$. Also,
\begin{equation}\label{Dir-BVP-X:b}
\|\mathcal{N} u\|_{\mathbb{X}}\approx\|f\|_{\mathbb{X}},
\end{equation}
where the constants involved depend only on $\mathbb{X}$, $n$, and $L$.

Moreover, the $\mathbb{X}'$-Dirichlet boundary value problem for $L$ in $\mathbb{R}^{n}_{+}$,
formulated analogously to \eqref{Dir-BVP-Xalone} {\rm (}with $\mathbb{X}'$ replacing $\mathbb{X}${\rm )}
is also well-posed, and the solution enjoys the same type of properties as above.

Finally, the above result is sharp in the sense that 
the solvability of both the $\mathbb{X}$-Dirichlet and the $\mathbb{X}'$-Dirichlet boundary 
value problems for the class of all second-order, homogeneous, elliptic systems, with constant 
complex coefficients {\rm (}in the form of convolution with the Poisson kernel{\rm )} with naturally 
accompanying bounds {\rm (}as in \eqref{Dir-BVP-X:b}{\rm )} is equivalent to the boundedness of 
the Hardy-Littlewood maximal operator ${\mathcal{M}}$ both on $\mathbb{X}$ and on $\mathbb{X}'$.
\end{theorem}

Assuming Theorem~\ref{Them-General}, the proof of Theorem~\ref{Theorem-Nice} is rather short.
Indeed, the discussion preceding its statement gives the well-posedness of
the $\mathbb{X}$-Dirichlet boundary value problem. Furthermore, since \eqref{dual-X2} entails
that the hypothesis \eqref{Nabb} is stable under replacing $\mathbb{X}$ by $\mathbb{X}'$,
the well-posedness of the $\mathbb{X}'$-Dirichlet boundary value problem follows as well.

As regards the sharpness claim from the last part of the statement, first assume
the solvability of the $\mathbb{X}$-Dirichlet boundary value problem
for the Laplacian in $\mathbb{R}^{n}_{+}$ (in the form of convolution with the
Poisson kernel) with naturally accompanying bounds. Note that, with
$P^{\Delta}$ as in \eqref{Uah-TTT}, whenever $0\leq f\in\mathbb{M}$ we may estimate,
for each $(x',t)\in{\mathbb{R}}^n_{+}$,
\begin{align}\label{eq:TRFcc}
u(x',t) &=(P^\Delta_t\ast f)(x')=\frac{2}{\omega_{n-1}}
\int_{\mathbb{R}^{n-1}}\frac{t}{\big(t^2+|x'-y'|^2\big)^{\frac{n}{2}}}\,f(y')\,dy'
\\[4pt]
& \geq C_n\aver{B_{n-1}(x',t)}f(y')\,dy'.\nonumber
\end{align}
Hence,
\begin{equation}\label{eq:Rdac}
({\mathcal{N}}u)(x')\geq\sup_{t>0}u(x',t)
\geq C_n({\mathcal{M}}f)(x'),\quad\mbox{for each }\,\,x'\in{\mathbb{R}}^{n-1},
\end{equation}
which, together with the upper estimate in \eqref{Dir-BVP-X:b}, implies
the boundedness of ${\mathcal{M}}$ on ${\mathbb{X}}$. Likewise, the
solvability of the $\mathbb{X}'$-Dirichlet boundary value problem
for the Laplacian in $\mathbb{R}^{n}_{+}$ yields the boundedness of
${\mathcal{M}}$ on ${\mathbb{X}}'$. This finishes the proof of Theorem~\ref{Theorem-Nice}.

\medskip

As is apparent from the above proof, the solvability of both the $\mathbb{X}$-Dirichlet and
the $\mathbb{X}'$-Dirichlet boundary value problems for the Laplacian in $\mathbb{R}^{n}_{+}$
{\rm (}in the form of convolution with the Poisson kernel{\rm )} with naturally accompanying bounds
implies the boundedness of ${\mathcal{M}}$ both on $\mathbb{X}$ and $\mathbb{X}'$.
As a consequence, the solvability of the $\mathbb{X}$-Dirichlet and
the $\mathbb{X}'$-Dirichlet boundary value problems for the Laplacian in $\mathbb{R}^{n}_{+}$
{\rm (}in the manner described above{\rm )} is equivalent to the
solvability of the $\mathbb{X}$-Dirichlet and the $\mathbb{X}'$-Dirichlet
boundary value problems in $\mathbb{R}^{n}_{+}$ for all systems $L$
as in \eqref{L-def}-\eqref{L-ell.X}.

Here we also wish to remark that, under the background assumptions in Theorem~\ref{Theorem-Nice},
the first three conditions in \eqref{Dir-BVP-Xalone} imply that the nontangential pointwise trace
$u\big|_{\partial\mathbb{R}^{n}_{+}}^{{}^{\rm n.t.}}$ exists a.e.~in ${\mathbb{R}}^{n-1}$.
Indeed, this becomes a consequence of a general Fatou type result established in
Theorem~\ref{tuFatou}, upon observing that 
\begin{equation}
{\mathbb{X}}\subset
L^1\Big({\mathbb{R}}^{n-1}\,,\,\frac{1+\log_{+}|x'|}{1+|x'|^{n-1}}\,dx'\Big)
\label{eq:sgetgt}.
\end{equation}
In turn, thanks to \eqref{YtrBS}, the latter condition is equivalent to
$\frac{1+\log_{+}|x'|}{1+|x'|^{n-1}}\in{\mathbb{X}}'$ which is further implied
by Lemma~\ref{lemma:M-ball} and \eqref{Lab-1}.

\medskip

At this stage we find it instructive to illustrate the scope of
Theorems~\ref{Them-General}-\ref{Theorem-Nice} by providing two examples of interest.

\medskip

\noindent\textbf{Example~1: Ordinary Lebesgue spaces.}
For $p\in(1,\infty)$, $\mathbb{X}:=L^p({\mathbb{R}}^{n-1})$ is a K\"othe function space,
with K\"othe dual $\mathbb{X}'=L^{p'}({\mathbb{R}}^{n-1})$ with $1/p+1/p'=1$. Hence,
in this case \eqref{Nabb} holds. As such, Theorem~\ref{Theorem-Nice} shows that
the $L^p$-Dirichlet boundary value problem in ${\mathbb{R}}^n_{+}$,
\begin{equation}\label{Dir-BVP-Lpw-intro}
\left\{
\begin{array}{l}
u\in{\mathscr{C}}^\infty(\mathbb{R}^{n}_{+}),
\\[4pt]
Lu=0\,\,\mbox{ in }\,\,\mathbb{R}^{n}_{+},
\\[6pt]
\mathcal{N}u\in L^p({\mathbb{R}}^{n-1}),
\\[4pt]
u\big|_{\partial\mathbb{R}^{n}_{+}}^{{}^{\rm n.t.}}=f\in L^p({\mathbb{R}}^{n-1}),
\end{array}
\right.
\end{equation}
is well-posed for any system $L$ as in \eqref{L-def}-\eqref{L-ell.X}. Moreover,
the solution is given by \eqref{eq:FDww} and satisfies naturally accompanying bounds.
Of course, one can also arrive at the same conclusion using Theorem~\ref{Them-General}
instead, since \eqref{Fi-AN.1}-\eqref{Fi-AN.2} are readily checked for
$\mathbb{X}=\mathbb{Y}:=L^p({\mathbb{R}}^{n-1})$ with $p\in(1,\infty)$.

In \S\ref{sect:Poisson}, the well-posedness of the $L^p$-Dirichlet problem \eqref{Dir-BVP-Lpw-intro}
is then used as a tool for establishing the uniqueness of the (Agmon-Douglis-Nirenberg) Poisson kernel
for the system $L$ (from Theorem~\ref{ya-T4-fav}), and to show that the said kernel
satisfies the semigroup property (cf. Theorem~\ref{taf87h6g}).

\medskip

\noindent\textbf{Example~2: Weighted Lebesgue spaces.}
Given $p\in(1,\infty)$, along with an a.e.~positive and finite measurable function $w$
defined on $\mathbb{R}^{n-1}$, let $L^p({\mathbb{R}}^{n-1},\,w)$ denote the Lebesgue space of $p$-th
power integrable functions in the measure space $\big(\mathbb{R}^{n-1},w(x')\,dx'\big)$. 
For a system $L$ as in \eqref{L-def}-\eqref{L-ell.X} the 
$L^p({\mathbb{R}}^{n-1},\,w)$-Dirichlet problem then reads:
\begin{equation}\label{Dir-BVP-Lpw-REal}
\left\{
\begin{array}{l}
u\in{\mathscr{C}}^\infty(\mathbb{R}^{n}_{+}),
\\[4pt]
Lu=0\,\,\mbox{ in }\,\,\mathbb{R}^{n}_{+},
\\[6pt]
\mathcal{N}u\in L^p({\mathbb{R}}^{n-1},\,w),
\\[4pt]
u\big|_{\partial\mathbb{R}^{n}_{+}}^{{}^{\rm n.t.}}=f\in L^p({\mathbb{R}}^{n-1},\,w).
\end{array}
\right.
\end{equation}
Theorem~\ref{Them-General} may then be invoked in order to show that
(with $A_p({\mathbb{R}}^{n-1})$ denoting the class of Muckenhoupt weights,
as defined in \eqref{Ap-CaD1})
\begin{equation}\label{LPDIRww}
\begin{array}{l}
\mbox{if $L$ is a system as in \eqref{L-def}-\eqref{L-ell.X}, $1<p<\infty$,
and $w\in A_p({\mathbb{R}}^{n-1})$ then}
\\[4pt]
\mbox{the $L^p({\mathbb{R}}^{n-1},\,w)$-Dirichlet problem 
\eqref{Dir-BVP-Lpw-REal} is well-posed, the solution}
\\[4pt]
\mbox{$u$ is given by \eqref{eq:FDww}, and satisfies }\,\,
\|\mathcal{N} u\|_{L^p({\mathbb{R}}^{n-1},\,w)}\leq C\|f\|_{L^p({\mathbb{R}}^{n-1},\,w)}.
\end{array}
\end{equation}

To see that this is the case, note that ${\mathbb{X}}={\mathbb{Y}}=L^p({\mathbb{R}}^{n-1},\,w)$
satisfy \eqref{Fi-AN.2} (taking into account Muckenhoupt's classical result), whereas the
second embedding in \eqref{Fi-AN.1} is checked by estimating for every $h\in L^p({\mathbb{R}}^{n-1},\,w)$
\begin{align}\label{Lpw-hyp}
&\int_{\mathbb{R}^{n-1}}\frac{1+\log_{+}|x'|}{1+|x'|^{n-1}}|h(x')|\,dx'
\\
&\qquad\qquad\qquad\le
C\,\int_{\mathbb{R}^{n-1}} |h(x')|\,w(x')^{\frac1p}\,
\mathcal{M}^{(2)}\big({\bf 1}_{B_{n-1}(0',1)}\big)(x')\,w(x')^{-\frac1p}\,dx'
\nonumber\\[4pt]
&\qquad\qquad\qquad\le
C\,\|h\|_{L^p({\mathbb{R}}^{n-1},\,w)}\,
\big\|\mathcal{M}^{(2)}\big({\bf 1}_{B_{n-1}(0',1)}\big)\big\|_{L^{p'}({\mathbb{R}}^{n-1},\,w^{1-p'})}
\nonumber\\[4pt]
&\qquad\qquad\qquad\le
C\,\|h\|_{L^p({\mathbb{R}}^{n-1},\,w)}\,
\big\|{\bf 1}_{B_{n-1}(0',1)}\big\|_{L^{p'}({\mathbb{R}}^{n-1},\,w^{1-p'})}\,
\nonumber\\[4pt]
&\qquad\qquad\qquad\le
C\,\|h\|_{L^p({\mathbb{R}}^{n-1},\,w)}\,w\big(B_{n-1}(0',1)\big)^{-\frac1{p}},\nonumber
\end{align}
where we have used Lemma~\ref{lemma:M-ball} for the first inequality,
H\"older's inequality for the second, that $\mathcal{M}$ is bounded on
$L^{p'}({\mathbb{R}}^{n-1},\,w^{1-p'})$ since $w\in A_p({\mathbb{R}}^{n-1})$
if and only if $w^{1-p'}\in A_{p'}({\mathbb{R}}^{n-1})$ in the third and, lastly,
that $w\in A_p({\mathbb{R}}^{n-1})$.
This takes care of the well-posedness, while the corresponding bound follows from
\eqref{eJHBb} and the boundedness of $\mathcal{M}$ on
$L^{p}({\mathbb{R}}^{n-1},\,w)$.

In this vein, it is worth noting that, as \eqref{eq:Rdac} shows, the bound in \eqref{LPDIRww}
in the case when $L=\Delta$ necessarily places the weight function $w$ in the Muckenhoupt class
$A_p({\mathbb{R}}^{n-1})$.

One may well wonder whether Theorem~\ref{Theorem-Nice} is also effective in the current setting.
However, this is not the case. To illustrate the root of the problem note that,
technically speaking, $L^p({\mathbb{R}}^{n-1},\,w)$ is not
a K\"othe function space on $({\mathbb{R}}^{n-1},dx')$
according to the terminology used earlier. Altering the definition so that
$L^p({\mathbb{R}}^{n-1},\,w)$ would be a K\"othe function space requires working with
$\big({\mathbb{R}}^{n-1},\,w(x')\,dx'\big)$ as the underlying measure space, and such a change
affects the manner in which the K\"othe dual is computed. Indeed, the K\"othe dual
of $L^p({\mathbb{R}}^{n-1},\,w)$ \big(which now has to be taken with respect to the measure space
$\big({\mathbb{R}}^{n-1},\,w(x')\,dx'\big)$\big) is $L^{p'}({\mathbb{R}}^{n-1},\,w)$. However, ${\mathcal{M}}$
is not necessarily bounded on this space, so \eqref{Nabb} cannot be ensured.

\medskip

So far we have seen that the Dirichlet problem with data in ordinary $L^p$ spaces can be treated
by Theorem~\ref{Theorem-Nice}, though this theorem ceases to be effective in the case of weighted $L^p$ spaces.
The question now becomes:
\begin{align}\label{NB-g43ka333}
\parbox{10.20cm}
{Is there a suitable version of Theorem~\ref{Theorem-Nice} targeted to more specialized
K\"othe function spaces, such as rearrangement invariant spaces, devised for the purpose of
treating not just $L^p({\mathbb{R}}^{n-1},\,w)$, but a variety of other weighted K\"othe spaces?}
\end{align}
Recall that a K\"othe function space $\big(\mathbb{X},\|\cdot\|_{\mathbb{X}}\big)$
is said to be {\tt rearrangement} {\tt inva\-riant} provided the function norm $\|f\|_{\mathbb{X}}$
of any $f\in{\mathbb{X}}$ may be expressed in terms of the measure of the level sets of that function.
The reader is referred to \S\ref{S-4} for a more detailed discussion, which also elaborates on the
notion of lower and upper Boyd indices, denoted by $p_{\mathbb{X}}$ and $q_{\mathbb{X}}$
(our definition ensures that $p_\mathbb{X}=q_\mathbb{X}=p$ if $\mathbb{X}=L^p({\mathbb{R}}^{n-1})$).
Given a weight $w$ on ${\mathbb{R}}^{n-1}$, if $f_w^\ast$ denotes the decreasing
rearrangement of $f$ with respect to the measure $w(x')\,dx'$, the weighted version
$\mathbb{X}(w)$ of the K\"othe function space $\mathbb{X}$ is defined as
\begin{equation}\label{eq:aaa.3sa}
\mathbb{X}(w):=\big\{f\in\mathbb{M}:\,\|f^\ast_w\|_{\overline{\mathbb{X}}}<\infty\big\},\qquad
\|f\|_{\mathbb{X}(w)}:=\|f_w^\ast\|_{\overline{\mathbb{X}}},
\end{equation}
where $\overline{\mathbb{X}}$ is the rearrangement invariant function space on $[0,\infty)$
associated with the original ${\mathbb{X}}$ as in Luxemburg's representation theorem.
One can check that if $\mathbb{X}:=L^p({\mathbb{R}}^{n-1})$, $p\in(1,\infty)$,
then $\mathbb{X}(w)=L^p({\mathbb{R}}^{n-1},\,w)$.

The theorem answering the question posed in \eqref{NB-g43ka333} is as follows.

\begin{theorem}\label{corol:Xw-Dir}
Let $L$ be a system as in \eqref{L-def}-\eqref{L-ell.X}, and let $\mathbb{X}$ be a
rearrangement invariant space whose lower and upper Boyd indices satisfy
\begin{equation}\label{eq:RRF}
1<p_\mathbb{X}\leq q_\mathbb{X}<\infty.
\end{equation}

Then for every Muckenhoupt weight $w\in A_{p_\mathbb{X}}({\mathbb{R}}^{n-1})$,
the $\mathbb{X}(w)$-Dirichlet boundary value problem for $L$ in $\mathbb{R}^{n}_{+}$,
\begin{equation}\label{Dir-BVP-Xw}
\left\{
\begin{array}{l}
u\in{\mathscr{C}}^\infty(\mathbb{R}^{n}_{+}),
\\[4pt]
Lu=0\,\,\mbox{ in }\,\,\mathbb{R}^{n}_{+},
\\[4pt]
\mathcal{N}u\in\mathbb{X}(w),
\\[6pt]
u\bigl|_{\partial\mathbb{R}^{n}_{+}}^{{}^{\rm n.t.}}=f\in\mathbb{X}(w),
\end{array}
\right.
\end{equation}
has a unique solution. Furthermore, the solution $u$ of \eqref{Dir-BVP-Xw}
is given by $u(x',t)=(P^L_t\ast f)(x')$ for all $(x',t)\in{\mathbb{R}}^n_{+}$,
where $P^L$ is the Poisson kernel for $L$ in $\mathbb{R}^{n}_{+}$, and there exists a
constant $C=C(n,L,\mathbb{X},w)\in(0,\infty)$ with the property that
\begin{equation}\label{Dir-BVP-Xw:bounds}
\|\mathcal{N} u\|_{\mathbb{X}(w)}
\leq C\,\|f\|_{\mathbb{X}(w)}.
\end{equation}
\end{theorem}

\vskip 0.08in

As a consequence of the classical result of Lorentz-Shimogaki,
given a rearrangement invariant space $\mathbb{X}$, condition \eqref{eq:RRF} is equivalent
to \eqref{Nabb}, i.e., to the fact that $\mathcal{M}$ is bounded on both $\mathbb{X}$ and $\mathbb{X}'$.
Thus, in the class of rearrangement invariant spaces, Theorem~\ref{corol:Xw-Dir} may be viewed as
a weighted version of Theorem~\ref{Theorem-Nice} (to which the latter reduces when the weight is a constant).
As was the case with Theorem~\ref{Theorem-Nice}, we also have that Theorem~\ref{corol:Xw-Dir} is sharp;
its proof is presented in \S\ref{S-4}, and the strategy relies on Theorem~\ref{Them-General}.
This requires verifying the embedding
\begin{equation}\label{Fi-AN.1BBB}
\mathbb{X}(w)\subset L^1\Big({\mathbb{R}}^{n-1}\,,\,\frac{1+\log_{+}|x'|}{1+|x'|^{n-1}}\,dx'\Big).
\end{equation}
A direct approach based on duality, along the lines of \eqref{Lpw-hyp}, quickly runs
into difficulties (due to the general nature of ${\mathbb{X}}(w)$, in contrast to the
particular case of $L^p({\mathbb{R}}^{n-1},\,w)$ considered in \eqref{Lpw-hyp}).
This being said, the fact that \eqref{Lpw-hyp} can be carried out for all
weights $w\in A_p({\mathbb{R}}^{n-1})$ eventually allows us to use Rubio de Francia's
extrapolation in the context of rearrangement invariant spaces (cf.~\cite{CMP}) in order to derive
a similar estimate in ${\mathbb{X}}(w)$ (cf.~Lemma~\ref{lemma:Xw-decay} for actual details).

In spite of its elegance and sharpness, Theorem~\ref{corol:Xw-Dir} is confined to the class
of rearrangement invariant spaces. An example of interest, lying outside the latter class, is
that of variable exponent Lebesgue spaces. As discussed below, in this setting it is
Theorem~\ref{Theorem-Nice} which may be employed in order to treat the corresponding
Dirichlet problem.

\medskip

\noindent\textbf{Example~3: Variable exponent Lebesgue spaces.}
Given a (Lebesgue) measurable function $p(\cdot):\mathbb{R}^{n-1}\rightarrow(1,\infty)$,
the variable Lebesgue space $L^{p(\cdot)}(\mathbb{R}^{n-1})$ is defined as the collection of all
measurable functions $f$ such that, for some $\lambda>0$,
\begin{equation}
\int_{\mathbb{R}^{n-1} }
\biggl(\frac{|f(x')|}{\lambda}\biggr)^{p(x')}\,dx'<\infty.
\end{equation}
Here and elsewhere, we follow the custom of writing $p(\cdot)$ instead of $p$
in order to emphasize that the exponent is a function and not necessarily a constant.
The set $L^{p(\cdot)}(\mathbb{R}^{n-1})$ becomes a K\"othe function space when equipped with
the function norm
\begin{equation}
\|f\|_{L^{p(\cdot)}(\mathbb{R}^{n-1})}:=
\inf\biggl\{\lambda>0:\int_{\mathbb{R}^{n-1} }
\biggl(\frac{|f(x')|}{\lambda}\biggr)^{p(x')}\,dx'\leq 1\biggr\}.
\end{equation}
This family of spaces generalizes the scale of ordinary Lebesgue spaces.
Indeed, if $p(x')\equiv p_0$, then $L^{p(\cdot)}(\mathbb{R}^{n-1})$ equals $L^{p_0}(\mathbb{R}^{n-1})$.
The K\"othe dual space of $L^{p(\cdot)}(\mathbb{R}^{n-1})$ is $L^{p'(\cdot)}(\mathbb{R}^{n-1})$,
where the conjugate exponent function $p'(\cdot)$ is uniquely defined by the demand that
\begin{equation}
\frac{1}{p(x')}+\frac{1}{p'(x')} = 1,\qquad\forall\,x'\in{\mathbb{R}}^{n-1}.
\end{equation}

Associated to $p(\cdot)$ we introduce the following natural parameters:
\begin{equation}
\label{eq:p-:p+}
p_{-}:=\essinf_{\mathbb{R}^{n-1}}p(\cdot)
\qquad\mbox{and}\qquad
p_{+}:=\esssup_{\mathbb{R}^{n-1}}p(\cdot).
\end{equation}

To apply Theorem~\ref{Theorem-Nice} to $\mathbb{X}:=L^{p(\cdot)}(\mathbb{R}^{n-1})$, we need
$\mathcal{M}$ to be bounded on both $L^{p(\cdot)}(\mathbb{R}^{n-1})$ and
$L^{p'(\cdot)}(\mathbb{R}^{n-1})$. This, in turn, is known to imply that $1<p_{-}\leq p_{+}<\infty$;
see \cite{cruz-uribe-fiorenza-neugebauer03}. Assuming $1<p_{-}\leq p_{+}<\infty$, it has been shown in
\cite{diening05} that $\mathcal{M}$ is bounded on $L^{p(\cdot)}(\mathbb{R}^{n-1})$ if and only if
$\mathcal{M}$ is bounded on $L^{p'(\cdot)}(\mathbb{R}^{n-1})$. Therefore, Theorem~\ref{Theorem-Nice}
gives the following result:
\begin{equation}\label{LPDIR}
\begin{array}{c}
\mbox{whenever $L$ is a second-order system as in \eqref{L-def}-\eqref{L-ell.X},}
\\[4pt]
\mbox{$1<p_{-}\leq p_{+}<\infty$, and if 
$\mathcal{M}$ is bounded on $L^{p(\cdot)}(\mathbb{R}^{n-1})$,}
\\[4pt]
\mbox{the $L^{p(\cdot)}(\mathbb{R}^{n-1})$-Dirichlet problem for $L$ 
in ${\mathbb{R}}^n_{+}$ is well-posed.}
\end{array}
\end{equation}
Moreover, the sharpness of Theorem~\ref{Theorem-Nice} yields a characterization of the
boundedness of the Hardy-Littlewood maximal operator on $L^{p(\cdot)}(\mathbb{R}^{n-1})$
and on $L^{p'(\cdot)}(\mathbb{R}^{n-1})$ in terms of the well-posedness of
the $L^{p(\cdot)}(\mathbb{R}^{n-1})$-Dirichlet and 
$L^{p'(\cdot)}(\mathbb{R}^{n-1})$-Dirich\-let problems in ${\mathbb{R}}^n_{+}$.

Let us further augment the above discussion by noting that, as proved in
\cite{cruz-uribe-fiorenza-neugebauer03} and \cite{nekvinda04}, the operator $\mathcal{M}$
is bounded on $L^{p(\cdot)}(\mathbb{R}^{n-1})$ if $p(\cdot)$ satisfies the following log-H\"older
continuity conditions: there exist constants $C\in[0,\infty)$ and $p_\infty\in[0,\infty)$ such that for
each $x',y'\in\mathbb{R}^{n-1}$,
\begin{equation}\label{logcond-1}
|p(x')-p(y')|\leq\frac{C}{-\log|x'-y'|}\quad\mbox{whenever }\,\,0<|x'-y'|\leq 1/2,
\end{equation}
and
\begin{equation}\label{logcond-2}
|p(x')-p_\infty|\leq\frac{C}{\log(e+|x'|)}.
\end{equation}
We refer the reader to \cite{CU-F} and \cite{DHHR} for full details and complete references.

\medskip

Moving on, we discuss two more classes of spaces for which Theorem~\ref{corol:Xw-Dir} applies.

\medskip

\noindent\textbf{Example~4: Weighted Lorentz spaces.}
Let $f^\ast$ denote the decreasing rearrangement of a function $f\in{\mathbb{M}}$
(cf.~\eqref{Ygav-iyg}). For $0<p,\,q<\infty$, define
\begin{equation}\label{eq:TRaa.1}
\|f\|_{L^{p,q}({\mathbb{R}}^{n-1})}:=\left(\int_0^\infty f^\ast(s)^q s^{q/p-1}\,ds\right)^{1/q},
\end{equation}
and, corresponding to $q=\infty$,
\begin{equation}\label{eq:TRaa.2}
\|f\|_{L^{p,\infty}({\mathbb{R}}^{n-1})}:=\sup_{0<s<\infty}\big[f^\ast(s)s^{1/p}\big].
\end{equation}
Then set
\begin{equation}\label{eq:MVc}
L^{p,q}({\mathbb{R}}^{n-1}):=\big\{f\in{\mathbb{M}}:\,\|f\|_{L^{p,q}({\mathbb{R}}^{n-1})}<\infty\big\}.
\end{equation}
For $0<p<\infty$ and $0<q\leq\infty$, the Lorentz spaces just defined
are only quasi-normed spaces, but when $1<p<\infty$ and $1\leq q\leq\infty$,
or when $p=1$ and $1\leq q<\infty$, they are equivalent to normed spaces.
Also,
\begin{equation}\label{eq:Asx}
\begin{array}{l}
\mbox{if $1<p<\infty$ and $1\leq q\leq\infty$,
or $p=1$ and $1\leq q<\infty$, then}
\\[4pt]
\mbox{$\mathbb{X}:=L^{p,q}({\mathbb{R}}^{n-1})$ is a rearrangement invariant function space}
\\[4pt]
\mbox{with lower and upper Boyd indices given by $p_\mathbb{X}=q_\mathbb{X}=p$.}
\end{array}
\end{equation}
The spaces $\mathbb{X}(w)$ are the weighted Lorentz spaces
$L^{p,q}\big({\mathbb{R}}^{n-1},w(x')dx'\big)$ obtained by replacing
$f^\ast$ with $f^\ast_w$ in \eqref{eq:TRaa.1}-\eqref{eq:MVc}.
Granted \eqref{eq:Asx}, Theorem~\ref{corol:Xw-Dir} applies and yields the well-posedness of
the Dirichlet problem in ${\mathbb{R}}^n_{+}$ for a system $L$ as in \eqref{L-def}-\eqref{L-ell.X}
with data in $L^{p,q}\big(\mathbb{R}^{n-1},\,w(x')dx'\big)$ provided $1<p<\infty$, $1\leq q\leq\infty$,
and $w\in A_p(\mathbb{R}^{n-1})$. In particular, this well-posedness result holds for data in
the standard Lorentz spaces $L^{p,q}(\mathbb{R}^{n-1})$ with $1<p<\infty$ and $1\leq q\leq\infty$.

\medskip

\noindent\textbf{Example~5: Weighted Orlicz spaces.}
Given a Young function $\Phi$, define the Orlicz space $L^\Phi(\mathbb{R}^{n-1})$ to
be the function space associated with the Luxemburg norm
\begin{equation}\label{eq:ORL}
\|f\|_{L^\Phi(\mathbb{R}^{n-1})}:=\inf\left\{\lambda>0:\,
\int_{\mathbb{R}^{n-1}}\Phi\left(\frac{|f(x')|}{\lambda}\right)\,dx'\leq 1\right\}.
\end{equation}
Then ${\mathbb{X}}:=L^\Phi(\mathbb{R}^{n-1})$ is a rearrangement invariant function space.
It turns out that its weighted version $\mathbb{X}(w)$, originally defined as in \eqref{eq:aaa.3sa},
may be described as above with the Lebesgue measure replaced by $w(x')\,dx'$.
Clearly the Lebesgue spaces are Orlicz spaces with $\Phi(t):=t^p$.
Other examples include the Zygmund spaces $L^p(\log L)^\alpha$,
$1<p<\infty$, $\alpha\in\mathbb{R}$, which are defined using
$\Phi(t):=t^p\log(e+t)^\alpha$. In this case, $p_\mathbb{X}=q_\mathbb{X}=p$,
so Theorem~\ref{corol:Xw-Dir} applies and yields the well-posedness of
the Dirichlet problem in ${\mathbb{R}}^n_{+}$ for a system $L$ as in \eqref{L-def}-\eqref{L-ell.X}
with data in the weighted Zygmund spaces
$L^p(\log L)^\alpha({\mathbb{R}}^{n-1},\,w(x')dx')$, $1<p<\infty$, $\alpha\in\mathbb{R}$, and
$w\in A_p({\mathbb{R}}^{n-1})$.

The spaces $L^p+L^q$ and $L^p\cap L^q$ can also be treated as Orlicz spaces, with
$\Phi(t)\approx\max\{t^p,t^q\}$ and $\Phi(t)\approx\min\{t^p,t^q\}$, respectively.
In both cases, $p_\mathbb{X}=\min\{p,q\}$ and $q_\mathbb{X}=\max\{p,q\}$.
Hence, if $1<\min\{p,q\}$ and $\max\{p,q\}<\infty$ then Theorem~\ref{corol:Xw-Dir} applies.
Note that for these and other Orlicz spaces, the Boyd indices can be computed directly
from the function $\Phi$ (see \cite[Chapter~4]{CMP}).

\vskip 0.08in
\begin{remark}\label{Rfav57}
As the alert reader has perhaps noted, in the applications of Theorem~\ref{Them-General}
(such as those discussed in \eqref{eq:BVCc}, \eqref{eq:BVCc.2}, Theorem~\ref{Theorem-Nice},
Theorem~\ref{corol:Xw-Dir}, as well as in Examples~1-5) we have
taken the set ${\mathbb{X}}$ to actually be a linear subspace of ${\mathbb{M}}$. This is no accident
since, in general, starting with $\mathbb{X},\mathbb{Y}$ merely satisfying \eqref{Fi-AN.1}-\eqref{Fi-AN.2},
if $\widehat{\mathbb{X}}$ is the linear span of $\mathbb{X}$ in $\mathbb{M}$, then
the pair $\widehat{\mathbb{X}},\mathbb{Y}$ continue to satisfy \eqref{Fi-AN.1}-\eqref{Fi-AN.2}.
Indeed, this is readily seen from the sublinearity of ${\mathcal{M}}$ and the fact that ${\mathbb{Y}}$
is a function lattice. In particular, for any system $L$ as in \eqref{L-def}-\eqref{L-ell.X}, the
$(\widehat{\mathbb{X}},\mathbb{Y})$-Dirichlet boundary value problem for $L$ in $\mathbb{R}^{n}_{+}$
is uniquely solvable in the same manner as before.
\end{remark}

We conclude our list of examples by discussing another significant case 
when Theorem~\ref{Them-General} applies. 

\medskip

\noindent\textbf{Example~6: Morrey spaces.}
Recall that the Morrey scale in ${\mathbb{R}}^{n-1}$ consists of spaces 
$\mathfrak{L}^{p,\lambda}({\mathbb{R}}^{n-1})$ defined for each $p\in(1,\infty)$ 
and $\lambda\in(0,n-1)$ according to 
\begin{equation}\label{MOR.1}
\mathfrak{L}^{p,\lambda}({\mathbb{R}}^{n-1}):=\Big\{f\in L^p_{\rm loc}({\mathbb{R}}^{n-1}):\,
\|f\|_{\mathfrak{L}^{p,\lambda}({\mathbb{R}}^{n-1})}<\infty\Big\}
\end{equation}
where
\begin{equation}\label{MOR.2}
\|f\|_{\mathfrak{L}^{p,\lambda}({\mathbb{R}}^{n-1})}:=
\sup_{x'\in{\mathbb{R}}^{n-1},\,r>0}\Big(r^{-\lambda}\int_{B_{n-1}(x',r)}|f(y')|^p\,dy'\Big)^{1/p}.
\end{equation}
Given a function $f\in L^p_{\rm loc}({\mathbb{R}}^{n-1})$, break up 
\begin{align}\label{MOR.3}
\int_{{\mathbb{R}}^{n-1}}|f(x')|\frac{1+\log_{+}|x'|}{1+|x'|^{n-1}}\,dx'
=\sum_{j=0}^\infty I_j
\end{align}
where 
\begin{align}\label{MOR.4}
I_0:=\int_{B_{n-1}(0',2)}|f(x')|\frac{1+\log_{+}|x'|}{1+|x'|^{n-1}}\,dx'
\end{align}
and, for each $j\in{\mathbb{N}}$, 
\begin{align}\label{MOR.5}
I_j:=\int_{B_{n-1}(0',2^{j+1})\setminus B_{n-1}(0',2^{j})}|f(x')|\frac{1+\log_{+}|x'|}{1+|x'|^{n-1}}\,dx'.
\end{align}
Use H\"older's inequality and \eqref{MOR.2} to estimate
\begin{align}\label{MOR.6}
I_j &\leq\frac{1+(j+1)\ln 2}{2^{j(n-1)}}\int_{B_{n-1}(0',2^{j+1})}|f(x')|\,dx'
\\[4pt]
&\leq 2^{n-1}\Big(\frac{\omega_{n-2}}{n-1}\Big)^{\frac{p-1}{p}}
\frac{1+(j+1)\ln 2}{2^{(j+1)(n-1)/p}}\Big(\int_{B_{n-1}(0',2^{j+1})}|f(x')|^p\,dx'\Big)^{1/p}
\nonumber\\[4pt]
&\leq 2^{n-1}\Big(\frac{\omega_{n-2}}{n-1}\Big)^{\frac{p-1}{p}}
\frac{1+(j+1)\ln 2}{2^{(j+1)(n-1-\lambda)/p}}
\|f\|_{\mathfrak{L}^{p,\lambda}({\mathbb{R}}^{n-1})},\qquad\forall\,j\in{\mathbb{N}},\nonumber
\end{align}
and, likewise, 
\begin{align}\label{MOR.7}
I_0 \leq 2^{n-1}\Big(\frac{\omega_{n-2}}{n-1}\Big)^{\frac{p-1}{p}}
\frac{1+\ln 2}{2^{(n-1-\lambda)/p}}
\|f\|_{\mathfrak{L}^{p,\lambda}({\mathbb{R}}^{n-1})}.
\end{align}
Bearing in mind that 
\begin{equation}\label{jDEgsgw}
\lambda<n-1\,\Longrightarrow\,
\sum_{j=0}^\infty \frac{1+(j+1)\ln 2}{2^{(j+1)(n-1-\lambda)/p}}<\infty,
\end{equation}
then yields 
\begin{align}\label{MOR.8}
\int_{{\mathbb{R}}^{n-1}}|f(x')|\frac{1+\log_{+}|x'|}{1+|x'|^{n-1}}\,dx'
\leq C_{n,p,\lambda}\|f\|_{\mathfrak{L}^{p,\lambda}({\mathbb{R}}^{n-1})}
\end{align}
for some finite constant $C_{n,p,\lambda}>0$ independent of $f$. This proves that if $p\in(1,\infty)$ and $\lambda\in(0,n-1)$ then
\begin{align}\label{MOR.9}
\mathfrak{L}^{p,\lambda}({\mathbb{R}}^{n-1})\subset
L^1\Big({\mathbb{R}}^{n-1}\,,\,\frac{1+\log_{+}|x'|}{1+|x'|^{n-1}}\,dx'\Big).
\end{align}
In addition, it is clear from \eqref{MOR.1}-\eqref{MOR.2} that
\begin{align}\label{MOR.10}
\text{$\mathfrak{L}^{p,\lambda}({\mathbb{R}}^{n-1})$ is a function lattice if 
$1<p<\infty$ and $0<\lambda<n-1$,}
\end{align}
and it has been proved by F. Chiarenza and M. Frasca in \cite{CF} that
\begin{align}\label{MOR.11}
\parbox{11.0cm}{the Hardy-Littlewood maximal operator ${\mathcal{M}}$ is bounded on 
the Morrey space $\mathfrak{L}^{p,\lambda}({\mathbb{R}}^{n-1})$ whenever 
$1<p<\infty$ and $0<\lambda<n-1$.}
\end{align}
Granted \eqref{MOR.9}-\eqref{MOR.11}, Theorem~\ref{Them-General} applies and gives that for 
any system $L$ as in \eqref{L-def}-\eqref{L-ell.X} the 
$\mathfrak{L}^{p,\lambda}({\mathbb{R}}^{n-1})$-Dirichlet problem 
\begin{equation}\label{Dir-BVP-Morrey}
\left\{
\begin{array}{l}
u\in{\mathscr{C}}^\infty(\mathbb{R}^{n}_{+}),
\\[4pt]
Lu=0\,\,\mbox{ in }\,\,\mathbb{R}^{n}_{+},
\\[6pt]
\mathcal{N}u\in\mathfrak{L}^{p,\lambda}({\mathbb{R}}^{n-1}),
\\[4pt]
u\big|_{\partial\mathbb{R}^{n}_{+}}^{{}^{\rm n.t.}}=f\in\mathfrak{L}^{p,\lambda}({\mathbb{R}}^{n-1}),
\end{array}
\right.
\end{equation}
is well-posed for arbitrary $p\in(1,\infty)$ and $\lambda\in(0,n-1)$. 

\medskip

In the last portion of this section we briefly comment on the literature dealing with 
Dirichlet boundary value problems for elliptic operators in the upper-half space. 
From the outset it is important to recognize that the nature 
of these problems is affected not only by the choice of the function space from which the
boundary datum $f$ is selected but also by the means through which the size of the solution $u$ 
is measured and the very manner in which its boundary trace is considered. For example, 
there is an enormous amount of work devoted to the case when the solution $u$ is sought in 
various Sobolev spaces in ${\mathbb{R}}^n_{+}$, the boundary datum $f$ is assumed to belong 
to suitable Besov spaces on ${\mathbb{R}}^{n-1}$, and the boundary trace of $u$ is taken in 
the generalized sense of Sobolev space theory. Classical references in this regard include 
\cite{ADNI}, \cite{ADNII}, \cite{LionsMagenes}, \cite{MazShap}, \cite{Taylor}, and
the reader is also invited to consult the literature cited therein.      

In this paper we are interested in the case when the size of $u$ is measured in terms of the 
nontangential maximal function and the trace of $u$ on the boundary of ${\mathbb{R}}^n_{+}$ 
is taken in a nontangential pointwise sense (cf. \eqref{nkc-EE-2}). 
In the particular case when $L=\Delta$, 
the Laplacian in ${\mathbb{R}}^n$, the boundary value problem \eqref{Dir-BVP-Lpw-intro} has been treated 
at length in a number of monographs, including \cite{ABR}, \cite{GCRF85}, \cite{St70}, \cite{Stein93}, 
and \cite{StWe71}. In all these works, the existence part makes use of the explicit form of the harmonic 
Poisson kernel from \eqref{Uah-TTT}, while the uniqueness relies on either the Maximum Principle, or 
the Schwarz reflection principle for harmonic functions. Neither of the latter techniques may be adapted
successfully to prove uniqueness in the case of general systems treated here, so we develop a new approach 
based on the properties of the Green function for an elliptic system in the upper half-space 
(reviewed in the appendix). While arguments involving Green functions have been successfully used 
in the past to prove uniqueness, the novelty here is that we succeed in constructing a Green function 
whose basic properties are compatible with the very formulation of the original boundary value problem. 
In our case, a key aspect is the specific manner in which the nontangential maximal function of 
the derivatives of the said Green function are controlled; cf. \eqref{bouMNN}, and other pertinent 
features from Theorem~\ref{ta.av-GGG.2A}. It is remarkable that such a detailed analysis may be 
carried out for the entire class of elliptic systems $L$ as in \eqref{L-def}-\eqref{L-ell.X}.  

There is also a sizable amount of work devoted to studying the (classical) Dirichlet problem 
for the Laplacian in the upper-half space with a continuous boundary datum $f$. In such 
a scenario, one seeks a harmonic function 
$u\in{\mathscr{C}}^\infty({\mathbb{R}}^n_{+})\cap{\mathscr{C}}^0(\overline{{\mathbb{R}}^n_{+}})$
satisfying $u|_{\partial{\mathbb{R}}^n_{+}}=f$. As noted by Helms in \cite[p.\,42 and p.\,158]{He}, 
even in the case when the boundary datum $f$ is a bounded continuous
function in ${\mathbb{R}}^{n-1}$ the solution $u$ of this classical Dirichlet problem is not
unique. To ensure uniqueness in such a setting one typically specifies the behavior of $u(x',t)$ 
as $t\to\infty$. A case in point is \cite{ST96}, where uniqueness is established in the class of harmonic 
functions $u\in{\mathscr{C}}^\infty({\mathbb{R}}^n_{+})\cap{\mathscr{C}}^0(\overline{{\mathbb{R}}^n_{+}})$ 
satisfying $u(x)=o(|x|\sec^\gamma\theta)$ as $|x|\to\infty$ (where $\theta:=\arccos(x_n/|x|)$ and 
$\gamma\in{\mathbb{R}}$ is arbitrary), by proving a Phragm\'en-Lindel\"of principle under the
latter growth condition. This builds on the work of \cite{Si88}, \cite{Wolf41}, and others.
In this regard, see also \cite{Yo96}. All these works rely on positivity and specialized 
properties of the Laplace operator, so the techniques employed do not extend to the considerably more 
general class of elliptic systems considered in the present paper. 

Much attention has also been paid to the case of the $L^p$-Dirichlet problem in the upper-half space 
for variable coefficient scalar elliptic operators in divergence form, $L={\rm div}A\nabla$, under 
various assumptions on the coefficient matrix $A=A(x',t)$ for $(x',t)\in{\mathbb{R}}^n_{+}$. 
For this topic, the interested reader is referred to the excellent exposition in Kenig's monograph 
\cite{Ke94}, as well as the more recent work in \cite{AAAHK}, \cite{HofMitMor} 
and in the references cited there. This body of work crucially relies on the De Giorgi-Nash-Moser 
theory, an ingredient not available for the type of systems considered in the
present paper. 

Finally, we wish to mention that in \cite{Shen} Shen has considered the well-posedness of the
Dirichlet problem for elliptic systems $L$ as in \eqref{L-def}-\eqref{L-ell.X} in a Lipschitz
domain $\Omega$ with boundary data from Morrey spaces on $\partial\Omega$. For this
more general class of domains he proved the well-posedness of a boundary value problem 
formulated as in \eqref{Dir-BVP-Morrey} with the upper-half space ${\mathbb{R}}^n_{+}$ replaced by 
a Lipschitz domain $\Omega$ but only when $p=2$. In relation to this, the novelty in our paper
is the consideration of the full range $p\in(1,\infty)$.

\section{Preliminary Matters}
\setcounter{equation}{0}
\label{S-2}

Throughout the paper, we let ${\mathbb{N}}$ stand for the collection of all strictly positive
integers, and set ${\mathbb{N}}_0:={\mathbb{N}}\cup\{0\}$. In this way $\mathbb{N}_0^k$,
where $k\in\mathbb{N}$, stands for the set of multi-indices $\alpha=(\alpha_1,\dots,\alpha_k)$ with
$\alpha_j\in\mathbb{N}_0$ for $1\leq j\leq k$. Also, fix $n\in{\mathbb{N}}$ with $n\geq 2$.
We shall work in the upper-half space
\begin{equation}\label{RRR-UpHs}
{\mathbb{R}}^{n}_{+}:=\big\{x=(x',x_n)\in
{\mathbb{R}}^{n}={\mathbb{R}}^{n-1}\times{\mathbb{R}}:\,x_n>0\big\},
\end{equation}
whose topological boundary $\partial{\mathbb{R}}^{n}_{+}={\mathbb{R}}^{n-1}\times\{0\}$
will be frequently identified with the horizontal hyperplane ${\mathbb{R}}^{n-1}$
via $(x',0)\equiv x'$. The origin in ${\mathbb{R}}^{n-1}$ is denoted by $0'$
and we let $B_{n-1}(x',r)$ stand for the $(n-1)$-dimensional Euclidean ball of radius $r$
centered at $x'\in{\mathbb{R}}^{n-1}$. Fix a number $\kappa>0$ and for each boundary point
$x'\in\partial{\mathbb{R}}^{n}_{+}$ introduce the conical nontangential approach region
with vertex at $x'$ as
\begin{equation}\label{NT-1}
\Gamma(x'):=\Gamma_\kappa(x'):=\big\{y=(y',t)\in{\mathbb{R}}^{n}_{+}:\,
|x'-y'|<\kappa\,t\big\}.
\end{equation}
Given a vector-valued function $u:{\mathbb{R}}^{n}_{+}\to{\mathbb{C}}^M$,
define the nontangential maximal function of $u$ by
\begin{equation}\label{NT-Fct}
\big({\mathcal{N}}u\big)(x'):=\big({\mathcal{N}}_\kappa u\big)(x')
:=\sup\big\{|u(y)|:\,y\in\Gamma_\kappa(x')\big\},\qquad
x'\in{\mathbb{R}}^{n-1}.
\end{equation}
It is well-known that the aperture of the cones used to define the nontangential maximal
operator plays only a secondary role; see Proposition~\ref{prop:cones-Lpw} for a concrete
result of this flavor. Whenever meaningful, we also define
\begin{equation}\label{nkc-EE-2}
u\Big|^{{}^{\rm n.t.}}_{\partial{\mathbb{R}}^{n}_{+}}(x')
:=\lim_{\Gamma_{\kappa}(x')\ni y\to (x',0)}u(y)
\quad\mbox{for }\,x'\in{\mathbb{R}}^{n-1}.
\end{equation}

In the sequel, we shall need to consider a localized version of the
nontangential maximal operator. Specifically, given any
$E\subset{\mathbb{R}}^n_{+}$, for each $u:E\to{\mathbb{C}}^M$
we set
\begin{equation}\label{NT-Fct.23}
\big({\mathcal{N}}^E u\big)(x'):=\big({\mathcal{N}}^E_\kappa u\big)(x')
:=\sup\big\{|u(y)|:\,y\in\Gamma_\kappa(x')\cap E\big\},
\quad x'\in{\mathbb{R}}^{n-1}.
\end{equation}
Hence, ${\mathcal{N}}^E_\kappa u={\mathcal{N}}_\kappa\widetilde{u}$
where $\widetilde{u}$ is the extension of $u$ to ${\mathbb{R}}^n_{+}$
by zero outside $E$. In the scenario when $u$ is originally defined in
the entire upper-half space ${\mathbb{R}}^n_{+}$ we may
therefore write
\begin{equation}\label{NT-Fct.23PPPP}
{\mathcal{N}}^E_\kappa u={\mathcal{N}}_\kappa({\bf 1}_E u),
\end{equation}
where ${\bf 1}_E $ denotes the characteristic function of $E$.
Corresponding to the special case when $E=\big\{(x',x_n)\in
{\mathbb{R}}^{n}_{+}:\,x_n<\varepsilon\big\}$, we simply write
${\mathcal{N}}^{(\varepsilon)}_\kappa$ in place of ${\mathcal{N}}^E_\kappa$.
That is,
\begin{equation}\label{Gruah.4}
{\mathcal{N}}^{(\varepsilon)}_\kappa u(x')
:=\sup_{\substack{y=(y',y_n)\in\Gamma_\kappa(x')\\ 0<y_n<\varepsilon}}
|u(y)|,\qquad x'\in\mathbb{R}^{n-1}.
\end{equation}

Throughout the paper we use the symbol $|E|$ to denote the Lebesgue measure of Lebesgue measurable
set $E\subset{\mathbb{R}}^n$. The Lebesgue measure itself in ${\mathbb{R}}^n$ will be
denoted by ${\mathscr{L}}^n$. We let $Q$ denote open cubes in $\mathbb{R}^{n-1}$ with
sides parallel to the coordinate axes, and employ $\ell(Q)$ to denote its side-length.
We will also use the standard convention $\lambda\,Q$, with $\lambda>0$,
for the cube concentric with $Q$ whose side-length is $\lambda\,\ell(Q)$.
For any $Q$ and any $h\in L^1_{\rm loc}(\mathbb{R}^{n-1})$, we write
\begin{equation}\label{nota-aver}
h_Q:=\aver{Q} h\,d{\mathscr{L}}^{n-1}:=\frac1{|Q|}\int_{Q} h(x')\,dx'.
\end{equation}
If the function $h$ is $\mathbb{C}^M$-valued, the average is taken componentwise.
The Hardy-Littlewood maximal operator on $\mathbb{R}^{n-1}$ is defined as
\begin{equation}\label{MMax}
\mathcal{M}f(x'):=\sup_{Q\ni x'}\aver{Q}|f(y')|\,dy',\qquad x'\in\mathbb{R}^{n-1}.
\end{equation}
Also, we write
\begin{equation}\label{Uga2}
\mathcal{M}^{(2)}:=\mathcal{M}\circ\mathcal{M}
\end{equation}
for the two-fold composition of $\mathcal{M}$ with itself. We follow the customary notation
$A\approx B$ in order to indicate that each quantity $A,B$ is dominated by a fixed multiple
of the other (via constants independent of the essential parameters intervening in $A,B$).

\begin{lemma}\label{lemma:M-ball}
For $x'\in\mathbb{R}^{n-1}$ one has
\begin{equation}\label{eq:M-ball}
\mathcal{M}\big({\bf 1}_{B_{n-1}(0',1)}\big)(x')
\approx\frac{1}{1+|x'|^{n-1}},
\end{equation}
and
\begin{equation}\label{eq:M2-ball}
\mathcal{M}^{(2)}\big({\bf 1}_{B_{n-1}(0',1)}\big)(x')
\approx\frac{1+\log_{+}|x'|}{1+|x'|^{n-1}},
\end{equation}
where the implicit constants depend only on $n$.
\end{lemma}

\begin{proof}
The proof of \eqref{eq:M-ball} is elementary but we include it for completeness.
Note first that for every $x'\in\mathbb{R}^{n-1}$, if we denote by $Q_{x'}$ the
cube in $\mathbb{R}^{n-1}$ centered at the origin and with side-length $2\,(|x'|+1)$,
then $x'\in Q_{x'}$ and $B_{n-1}(0',1)\subset Q_{x'}$. Thus, we easily obtain
\begin{multline}\label{ytDd.1}
\mathcal{M}\big({\bf 1}_{B_{n-1}(0',1)}\big)(x')
\geq\aver{Q_{x'}} {\bf 1}_{B_{n-1}(0',1)}(y')\,dy'
\\
=\frac{|B_{n-1}(0',1)|}{|Q_{x'}|}
\geq\frac{C_n}{1+|x'|^{n-1}}.
\end{multline}
To obtain the converse inequality we first observe that, clearly,
\begin{equation}\label{ytDd.2}
\mathcal{M}\big({\bf 1}_{B_{n-1}(0',1)}\big)(x')
\leq 1\le\frac{C_n}{1+|x'|^{n-1}},
\qquad\mbox{whenever }\,\,|x'|\leq 2.
\end{equation}
Suppose next that $|x'|>2$. Notice that if $x'\in Q\subset\mathbb{R}^{n-1}$ and
there is some $y'\in Q\cap B_{n-1}(0',1)$ then
\begin{equation}\label{uyafav}
|x'|\leq|x'-y'|+|y'|\leq\sqrt{n}\,\ell(Q)+1\leq\sqrt{n}\,\ell(Q)+|x'|/2.
\end{equation}
Therefore $\ell(Q)>|x'|/(2\,\sqrt{n})$, which entails
\begin{equation}\label{yrd7u}
\aver{Q} {\bf 1}_{B_{n-1}(0',1)}(y')\,dy'\leq\frac{|B_{n-1}(0',1)|}{|Q|}
\leq\frac{C_n}{|x'|^{n-1}}\leq\frac{C_n}{1+|x'|^{n-1}}.
\end{equation}
The same inequality trivially holds in the case when $Q$ is disjoint from $B_{n-1}(0',1)$.
Taking the supremum of the most extreme sides of \eqref{yrd7u} over all cubes $Q$
containing $x'$ then yields the upper estimate in \eqref{eq:M-ball} in the case
when $|x'|>2$. This finishes the proof of \eqref{eq:M-ball}.

Turning to the proof of \eqref{eq:M2-ball}, we first invoke an auxiliary estimate
whose proof can be found in \cite{curbera-garcia-cuerva-martell-perez06}:
\begin{equation}\label{eq:M2-MLogL}
\begin{array}{c}
\mathcal{M}^{(2)}f(x')\approx\mathcal{M}_{L\,\log L} f(x'):=\sup_{Q\ni x'}\|f\|_{L\,\log L,Q}
\\[6pt]
\mbox{uniformly for $f\in L^1_{\rm loc}({\mathbb{R}}^{n-1})$ and $x'\in{\mathbb{R}}^{n-1}$,}
\end{array}
\end{equation}
where $\|\cdot\|_{L\,\log L,Q}$ stands for the localized and normalized Luxemburg norm
\begin{equation}\label{YGv-u6g}
\|f\|_{L\,\log L,Q}:=\inf\left\{\lambda>0:\,
\aver{Q}\Phi\left(\frac{|f(x')|}{\lambda}\right)\,dx'\leq 1\right\},
\end{equation}
with $\Phi(t):=t\,\log(e+t)$, $t\geq 0$.
Defining $\varphi(t):=\big(\Phi^{-1}(t^{-1})\big)^{-1}$ for $t\in(0,\infty)$
and $\varphi(0):=0$, easy calculations lead to
\begin{equation}\label{LlogL-1B}
\|1_{B_{n-1}(0',1)}\|_{L\,\log L,Q}
=\varphi\left(\frac{|B_{n-1}(0',1)\cap Q|}{|Q|}\right)
=\varphi\left(\aver{Q}{\bf 1}_{B_{n-1}(0',1)}(y')\,dy'\right).
\end{equation}
Using then \eqref{eq:M2-MLogL}, \eqref{LlogL-1B}, the fact that $\varphi$ is a continuous
strictly increasing function in $[0,\infty)$, and \eqref{eq:M-ball}, we conclude that
\begin{align}\label{M2-estimate}
\mathcal{M}^{(2)}\big(1_{B_{n-1}(0',1)}\big)(x')
&\approx\sup_{Q\ni x'}\varphi\left(\aver{Q}{\bf 1}_{B_{n-1}(0',1)}(y')\,dy'\right)
\\[4pt]
&=\varphi\left(\sup_{Q\ni x'}\aver{Q}{\bf 1}_{B_{n-1}(0',1)}(y')\,dy'\right)
\nonumber\\[4pt]
&=\varphi\big(\mathcal{M}(1_{B_{n-1}(0',1)})(x')\big)
\approx\varphi\left(\frac1{1+|x'|^{n-1}}\right),\nonumber
\end{align}
uniformly for $x'\in{\mathbb{R}}^{n-1}$.
Thus, to complete the proof of \eqref{eq:M2-ball} we only need to find a suitable estimate
for the last term above. To this end, one can easily check that $\Phi^{-1}(t)\approx t/\log(e+t)$
which gives that $\varphi(t)\approx t\,\log(e+t^{-1})$. This and \eqref{M2-estimate}
then yield
\begin{equation}\label{LbvVA}
\mathcal{M}^{(2)}\big(1_{B_{n-1}(0',1)}\big)(x')
\approx\frac1{1+|x'|^{n-1}}\,\log(e+1+|x'|^{n-1})\approx
\frac{1+\log_{+}|x'|}{1+|x'|^{n-1}},
\end{equation}
uniformly for $x'\in{\mathbb{R}}^{n-1}$, as desired.
\end{proof}

We next introduce the class of {\tt Muckenhoupt} {\tt weights}. Call a real-valued function $w$
defined on $\mathbb{R}^{n-1}$ a {\tt weight} if it is non-negative and measurable. Given a weight $w$
and $p\in[1,\infty]$, we write $L^p({\mathbb{R}}^{n-1},\,w)=L^p({\mathbb{R}}^{n-1},\,w\,dx')$.
If $1<p<\infty$, a weight $w$ belongs to the Muckenhoupt class $A_p=A_p({\mathbb{R}}^{n-1})$ if
\begin{equation}\label{Ap-CaD1}
[w]_{A_p}:=\sup_{Q\subset\mathbb{R}^{n-1}}\Bigl(\aver{Q} w(x')\,dx'\Bigr)
\Bigl(\aver{Q} w(x')^{1-p'}\,dx'\Bigr)^{p-1}<\infty,
\end{equation}
where $p'=p/(p-1)$ denotes the conjugate exponent of $p$. Corresponding to $p=1$,
the class $A_1=A_1({\mathbb{R}}^{n-1})$ is then defined as the collection of all weights $w$
in ${\mathbb{R}}^{n-1}$ for which
\begin{equation}\label{Ap-CaD2}
[w]_{A_1}:=\sup_{Q\subset\mathbb{R}^{n-1}}\Bigl(\essinf_Q\,w\Bigr)^{-1}
\Bigl(\aver{Q} w(x')\,dx'\Bigr)<\infty.
\end{equation}
In particular,
\begin{equation}\label{Ap-CaD3}
\aver{Q} w(y')\,dy'\leq [w]_{A_1}\,w(x')\,\,\mbox{ for a.e. }\,x'\in Q,
\end{equation}
for every cube $Q\subset\mathbb{R}^{n-1}$. Equivalently,
\begin{equation}\label{A1-alt}
\mathcal{M}w(x')\leq[w]_{A_1}\,w(x')\,\,\mbox{ for a.e. }\,x'\in\mathbb{R}^{n-1}.
\end{equation}
Finally, corresponding to $p=\infty$, we let $A_\infty$ stand for $\bigcup_{1\leq p<\infty}A_p$.

We summarize a number of well-known facts which are relevant for us here.
See, e.g., \cite{GCRF85} for a more detailed discussion, including the following basic properties:
\begin{list}{$(\theenumi)$}{\usecounter{enumi}\leftmargin=1cm
\labelwidth=1cm\itemsep=0.2cm\topsep=.2cm
\renewcommand{\theenumi}{\roman{enumi}}}
\item given $1<p<\infty$ and a weight $w$, then $w\in A_p$ if and only if ${\mathcal{M}}$
is bounded on $L^p({\mathbb{R}}^{n-1},\,w)$;
\item given $1<p<\infty$ and a weight $w$, then $w\in A_p$ if and only if $w^{1-p'}\in A_{p'}$,
and $[w^{1-p'}]_{A_{p'}}=[w]^{p'-1}_{A_p}$;
\item if $w_1,w_2\in A_1$ and $1\leq p<\infty$, then $w_1\,w_2^{1-p}\in A_p$ and
$[w_1\,w_2^{1-p}]_{A_p}\leq [w_1]_{A_1}\, [w_2]^{p-1}_{A_1}$;
\item the classes $A_p$, $1\leq p<\infty$, may be equivalently defined using balls in $\mathbb{R}^{n-1}$
(in place of cubes), in which scenario $[w]_{A_p}^{\rm balls}\approx [w]_{A_p}$ with
implicit constants depending only on $n$ and $p$.
\end{list}

\medskip

In the last part of this section we discuss the notion of Poisson kernel in ${\mathbb{R}}^n_{+}$
for an operator $L$ as in \eqref{L-def}-\eqref{L-ell.X}.

\begin{definition}[Poisson Kernel for $L$ in $\mathbb{R}^{n}_{+}$]\label{defi:Poisson}
Let $L$ be a second-order elliptic system with complex coefficients as in
\eqref{L-def}-\eqref{L-ell.X}.
A {\tt Poisson kernel} for $L$ in $\mathbb{R}^{n}_{+}$ is a matrix-valued function
$P^L=\big(P^L_{\alpha\beta}\big)_{1\leq\alpha,\beta\leq M}:
\mathbb{R}^{n-1}\to\mathbb{C}^{M\times M}$ such that the following conditions hold:
\begin{list}{$(\theenumi)$}{\usecounter{enumi}\leftmargin=.8cm
\labelwidth=.8cm\itemsep=0.2cm\topsep=.1cm
\renewcommand{\theenumi}{\alph{enumi}}}
\item there exists $C\in(0,\infty)$ such that
\begin{equation}
\displaystyle|P^L(x')|\leq\frac{C}{(1+|x'|^2)^{\frac{n}2}} \quad\mbox{for each
$x'\in\mathbb{R}^{n-1}$;}
\label{eq:po,gwet}
\end{equation}
\item the function $P^L$ is Lebesgue measurable and
$\displaystyle\int_{\mathbb{R}^{n-1}}P^L(x')\,dx'=I_{M\times M}$,
the $M\times M$ identity matrix;
\item if $K^L(x',t):=P^L_t(x'):=t^{1-n}P^L(x'/t)$, for each
$x'\in\mathbb{R}^{n-1}$ and $t\in(0,\infty)$, then the function
$K^L=\big(K^L_{\alpha\beta}\big)_{1\leq\alpha,\beta\leq M}$
satisfies {\rm (}in the sense of distributions{\rm )}
\begin{equation}\label{uahgab-UBVCX}
LK^L_{\cdot\beta}=0\,\,\mbox{ in }\,\,\mathbb{R}^{n}_{+}
\,\,\mbox{ for each }\,\,\beta\in\{1,\dots,M\},
\end{equation}
where $K^L_{\cdot\beta}:=\big(K^L_{\alpha\beta}\big)_{1\leq\alpha\leq M}$.
\end{list}
\end{definition}

\vskip 0.06in
\begin{remark}\label{Ryf-uyf}
The following comments pertain to Definition~\ref{defi:Poisson}.
\begin{list}{$(\theenumi)$}{\usecounter{enumi}\leftmargin=.8cm
\labelwidth=.8cm\itemsep=0.2cm\topsep=.1cm
\renewcommand{\theenumi}{\roman{enumi}}}
\item Condition $(a)$ ensures that the integral in part $(b)$ is
absolutely convergent.
\item Condition $(c)$ and the ellipticity of the operator $L$ ensure
(cf. \cite[Theorem~10.9, p.\,318]{DM}) that
$K^L\in{\mathscr{C}}^\infty(\mathbb{R}^{n}_{+})$. In particular, \eqref{uahgab-UBVCX}
holds in a pointwise sense. Also, given that $P^L(x')=K^L(x',1)$ for each
$x'\in{\mathbb{R}}^{n-1}$, we deduce that $P^L\in{\mathscr{C}}^\infty(\mathbb{R}^{n-1})$.
\item Condition $(b)$ is equivalent to $\lim\limits_{t\to 0^{+}}P^L_t(x')
=\delta_{0'}(x')\,I_{M\times M}$ in ${\mathcal{D}}'({\mathbb{R}}^{n-1})$,
where $\delta_{0'}$ is Dirac's distribution with mass at the origin $0'$
of ${\mathbb{R}}^{n-1}$.
\item For all $x\in{\mathbb{R}}^n_{+}$ and $\lambda>0$ we have $K^L(\lambda x)=\lambda^{1-n}K^L(x)$.
\end{list}
\end{remark}

Poisson kernels for elliptic boundary value problems in a half-space have
been studied extensively in \cite{ADNI}, \cite{ADNII}, \cite[\S{10.3}]{KMR2},
\cite{Sol}, \cite{Sol1}, \cite{Sol2}. Here we record a corollary of more general
work done by S.\,Agmon, A.\,Douglis, and L.\,Nirenberg in \cite{ADNII}.

\begin{theorem}\label{ya-T4-fav}
Any elliptic differential operator $L$ as in \eqref{L-def}-\eqref{L-ell.X} has a Poisson
kernel $P^L$ in the sense of Definition~\ref{defi:Poisson}, which has the additional
property that the function
\begin{equation}\label{eq:KDEF}
K^L(x',t):=P^L_t(x')\quad\mbox{for all }\,\,(x',t)\in{\mathbb{R}}^n_{+},
\end{equation}
satisfies $K^L\in{\mathscr{C}}^\infty\big(\overline{{\mathbb{R}}^n_{+}}\setminus B(0,\varepsilon)\big)$
for every $\varepsilon>0$.
\end{theorem}

\begin{remark}\label{YTrrea}
As a consequence of part ${\it (iv)}$ in Remark~\ref{Ryf-uyf} and the regularity of $K$
stated in Theorem~\ref{ya-T4-fav}, we have that for each multi-index $\alpha\in{\mathbb{N}}_0^n$
there exists $C_\alpha\in(0,\infty)$ with the property that
\begin{equation}\label{eq:KjG}
\big|(\partial^\alpha K^L)(x)\big|\leq C_\alpha\,|x|^{1-n-|\alpha|},\,\,\,\mbox{ for every }\,\,
x\in{\overline{{\mathbb{R}}^n_{+}}}\setminus\{0\}.
\end{equation}
In this respect, we wish to note that this estimate is stronger than
what a direct application of the properties of Poisson kernels listed in
Definition~\ref{defi:Poisson} would imply. Specifically,
as noted in part ${\it (ii)}$ of Remark~\ref{Ryf-uyf}, we have
$K^L\in{\mathscr{C}}^\infty(\mathbb{R}^{n}_{+})$ which, in concert with
part ${\it (iv)}$ of Remark~\ref{Ryf-uyf}, shows that \eqref{eq:KjG} holds for
$x\in\Gamma_\kappa(0')$, for each $\kappa>0$, with a constant also depending on the parameter $\kappa$.
\end{remark}

\section{Tools for Existence and Uniqueness}
\setcounter{equation}{0}
\label{S-3}

This section is devoted to proving the results stated in
Theorems~\ref{thm:existence}-\ref{thm:uniqueness} below. Here and elsewhere,
the convolution between two functions, which are matrix-valued and vector-valued, respectively,
takes into account the algebraic multiplication between a matrix and a vector in a natural fashion.

\begin{theorem}[Main Tool for the Existence Part]\label{thm:existence}
Let $L$ be a system as in \eqref{L-def}-\eqref{L-ell.X}.
Given a Lebesgue measurable function $f:\mathbb{R}^{n-1}\rightarrow\mathbb{C}^M$ satisfying
\begin{equation}\label{exist:f}
\int_{\mathbb{R}^{n-1}}\frac{|f(x')|}{1+|x'|^n}\,dx'<\infty,
\end{equation}
set
\begin{equation}\label{exist:u}
u(x',t):=(P^L_t\ast f)(x'),\qquad\forall\,(x',t)\in{\mathbb{R}}^n_{+},
\end{equation}
where $P^L$ is the Poisson kernel for $L$ in $\mathbb{R}^{n}_{+}$ from Theorem~\ref{ya-T4-fav}.
Then $u$ is meaningfully defined via an absolutely convergent integral,
\begin{equation}\label{exist:u2}
u\in\mathscr{C}^\infty(\mathbb{R}^n_{+}),\quad
Lu=0\,\,\mbox{ in }\,\,\mathbb{R}^{n}_{+},\quad
u\big|_{\partial\mathbb{R}^{n}_{+}}^{{}^{\rm n.t.}}=f\,\,\mbox{ a.e.~in }\,\,\mathbb{R}^{n-1}
\end{equation}
(convergence holds, for instance, in the set of Lebesgue points of $f$), and there exists
a constant $C=C(n,L)\in(0,\infty)$ with the property that
\begin{equation}\label{exist:Nu-Mf}
\mathcal{N} u(x')\leq C\,\mathcal{M} f(x'),
\qquad\forall\,x'\in\mathbb{R}^{n-1}.
\end{equation}
\end{theorem}

\medskip

\begin{theorem}[Main Tool for the Uniqueness Part]\label{thm:uniqueness}
Let $L$ be a system as in \eqref{L-def}-\eqref{L-ell.X}. Assume that
$u\in\mathscr{C}^\infty(\mathbb{R}^n_{+})$ is such that $Lu=0$ in $\mathbb{R}_{+}^n$,
its nontangential maximal function $\mathcal{N}u$ satisfies
\begin{equation}\label{unq:Nu}
\int_{\mathbb{R}^{n-1}}\mathcal{N}u(x')\,\frac{1+\log_{+}|x'|}{1+|x'|^{n-1}}\,\,dx'<\infty,
\end{equation}
and that $u\big|_{\partial\mathbb{R}^{n}_{+}}^{{}^{\rm n.t.}}=0$ a.e.~in $\mathbb{R}^{n-1}$.
Then $u\equiv 0$ in $\mathbb{R}^n_{+}$.
\end{theorem}

In preparation to presenting the proof of Theorem~\ref{thm:existence} we first deal with
a purely real variable lemma pertaining to the stability of the first weighted $L^1$ space
appearing in \eqref{Fi-AN.1} under convolutions with a fixed (matrix-valued) function whose
size is controlled by the harmonic Poisson kernel. In the same context, we also deal with
nontangential maximal function estimates and nontangential limits.

\begin{lemma}\label{lennii}
Let $P=\big(P_{\alpha\beta}\big)_{1\leq\alpha,\beta\leq M}:
\mathbb{R}^{n-1}\to\mathbb{C}^{M\times M}$ be a Lebesgue measurable
function satisfying, for some $c\in(0,\infty)$,
\begin{equation}\label{vznv.ADF}
|P(x')|\leq\frac{c}{(1+|x'|^2)^{\frac{n}2}}
\,\,\,\mbox{ for each }\,\,\,x'\in\mathbb{R}^{n-1},
\end{equation}
and recall that $P_t(x'):=t^{1-n}P(x'/t)$ for each $x'\in\mathbb{R}^{n-1}$ and $t\in(0,\infty)$.
Then, for each $t\in(0,\infty)$ fixed, the operator
\begin{equation}\label{eq:Eda}
L^1\Big({\mathbb{R}}^{n-1}\,,\,\frac{1}{1+|x'|^n}\,dx'\Big)\ni f\mapsto P_t\ast f\in
L^1\Big({\mathbb{R}}^{n-1}\,,\,\frac{1}{1+|x'|^n}\,dx'\Big)
\end{equation}
is well-defined, linear and bounded, with operator norm controlled by $C(t+1)$. Moreover,
for every $\kappa>0$ there exists a finite constant $C_\kappa>0$ with the property that
for each $x'\in\mathbb{R}^{n-1}$,
\begin{equation}\label{exTGFVC}
\sup_{|x'-y'|<\kappa t}\big|(P_t\ast f)(y')\big|\leq C_\kappa\,\mathcal{M} f(x'),
\qquad\forall\,f\in L^1\Big({\mathbb{R}}^{n-1}\,,\,\frac{1}{1+|x'|^n}\,dx'\Big).
\end{equation}
Finally, given any function
\begin{equation}\label{eq:aaAa}
f=(f_\beta)_{1\leq\beta\leq M}\in L^1\Big({\mathbb{R}}^{n-1}\,,\,\frac{1}{1+|x'|^n}dx'\Big)
\subset L^1_{\rm loc}({\mathbb{R}}^{n-1}),
\end{equation}
at every Lebesgue point $x'_0\in{\mathbb{R}}^{n-1}$ of $f$ there holds
\begin{equation}\label{exTGFVC.2s}
\lim_{\substack{(x',\,t)\to(x'_0,0)\\ |x'-x'_0|<\kappa t}}(P_t\ast f)(x')
=\left(\int_{\mathbb{R}^{n-1}}P(x')\,dx'\right)f(x'_0),
\end{equation}
and the function
\begin{equation}\label{exTefef}
{\mathbb{R}}^n_{+}\ni(x',t)\mapsto(P_t\ast f)(x')\in{\mathbb{C}}^M
\,\,\,\mbox{ is locally integrable in }\,\,{\mathbb{R}}^n_{+}.
\end{equation}
\end{lemma}

\begin{proof}
Pick a function $f$ as in \eqref{eq:aaAa} and fix some $t\in(0,\infty)$.
First, consider the issue whether $P_t\ast f$ is well-defined, via an absolutely
convergent integral. In this regard, note that for any $x',y'\in\mathbb{R}^{n-1}$
and $t\in(0,\infty)$ one has $|y'|\leq(1+|x'|/t)\,(t+|x'-y'|)$ and $1\leq (1/t)(t+|x'-y'|)$, hence
\begin{equation}\label{est1-Pt}
1+|y'|\leq(1+|x'|/t+1/t)\,(t+|x'-y'|).
\end{equation}
Thus, for each fixed $x'\in\mathbb{R}^{n-1}$ and $t\in(0,\infty)$, we have
\begin{multline}\label{Pt-abs-conv}
\int_{\mathbb{R}^{n-1}}\frac{t}{(t+|x'-y'|)^n}\,|f(y')|\,dy'
\\
\leq C\,t(1+|x'|/t+1/t)^n\,\int_{\mathbb{R}^{n-1}}\frac{|f(y')|}{1+|y'|^n}\,dy'<\infty,
\end{multline}
which, in light of \eqref{vznv.ADF}, shows that $P_t\ast f$ is meaningfully defined via
an absolutely convergent integral. To proceed, observe from \eqref{vznv.ADF} and
\eqref{Uah-TTT} that there exists some $C\in(0,\infty)$ with the property that
\begin{equation}\label{vznv.ADF.22}
|P_t(x')|\leq C P^{\Delta}_t(x')\,\,\,\mbox{ for all }\,\,x'\in\mathbb{R}^{n-1},\,\,t\in(0,\infty).
\end{equation}
Consequently,
\begin{align}\label{yagtyt}
\int_{{\mathbb{R}}^{n-1}}\frac{\big|(P_t\ast f)(x')\big|}{1+|x'|^n}\,dx'
& \leq C\int_{{\mathbb{R}}^{n-1}}\big(P^{\Delta}_t\ast|f|\big)(x')P^{\Delta}_1(x')\,dx'
\\[4pt]
& = C\Big(\big(P^{\Delta}_t\ast|f|\big)\ast P^{\Delta}_1\Big)(0')
\nonumber\\[4pt]
&
=C\Big(\big(P^{\Delta}_t\ast P^{\Delta}_1\big)\ast|f|\Big)(0')
\nonumber\\[4pt]
& =C\big(P^{\Delta}_{t+1}\ast|f|\big)(0')
\nonumber\\[4pt]
& \leq C\,(t+1)\,\int_{\mathbb{R}^{n-1}}\frac{|f(y')|}{1+|y'|^n}\,dy',\nonumber
\end{align}
where we have used the semigroup property for the harmonic Poisson kernel
(cf., e.g., \cite[(vi), p.\,62]{St70}), and where
the last inequality follows from \eqref{Pt-abs-conv} written with $t+1$ in place of $t$
and $x'=0'$. Now all desired conclusions concerning \eqref{eq:Eda} are seen from \eqref{yagtyt}.

Before proceeding with the rest of the proof, let us momentarily digress in order to note that,
once some $\kappa>0$ has been fixed, \eqref{vznv.ADF} self-improves in the sense that there
exists $C_\kappa\in(0,\infty)$ such that, for every $x'\in{\mathbb{R}}^{n-1}$ and $t\in(0,\infty)$,
\begin{equation}\label{TUhfg}
|P_t(x'-y')|\leq C_\kappa\,\frac{t}{(t^2+|x'|^2)^{\frac{n}{2}}}\qquad
\mbox{ whenever }\,\,|y'|<\kappa t.
\end{equation}
Indeed, this follows from the fact that $|x'|\leq \max\{1,\kappa\}(t+|x'-y'|)$ whenever
$x',y'\in{\mathbb{R}}^{n-1}$ and $t\in(0,\infty)$ are such that $|y'|<\kappa t$ which,
in turn, is easily justified by the triangle inequality.

To deal with \eqref{exTGFVC}, pick a function $f$ as in \eqref{eq:aaAa}. Also,
fix $x'\in{\mathbb{R}}^{n-1}$ and let $y'\in{\mathbb{R}}^{n-1}$ and $t\in(0,\infty)$
satisfy $|x'-y'|<\kappa t$.  Granted \eqref{TUhfg}, this implies
\begin{equation}\label{TUhfg.MM}
|P_t(y'-z')|\leq C_\kappa\,\frac{t}{(t^2+|x'-z'|^2)^{\frac{n}{2}}}\qquad
\mbox{ for every }\,\,z'\in{\mathbb{R}}^{n-1}.
\end{equation}
Based on \eqref{TUhfg.MM} we may then estimate
\begin{align}\label{Drsgy-1jab}
\big|(P_t\ast f)(y')\big| 
&\leq \int_{{\mathbb{R}}^{n-1}}|P_t(y'-z')|\,|f(z')|\,dz'
\\
&
\leq C_\kappa\int_{{\mathbb{R}}^{n-1}}\frac{t}{(t^2+|x'-z'|^2)^{\frac{n}{2}}}\,|f(z')|\,dz'
\nonumber\\[4pt]
&\leq C_\kappa\,\aver{B_{n-1}(x',t)}|f(z')|\,dz'
\nonumber\\[4pt]
&\hskip1cm
+\sum_{j=0}^{\infty}\int_{B_{n-1}(x',2^{j+1}t)\setminus B_{n-1}(x',2^{j}t)}
\frac{t}{(t^2+|x'-z'|^2)^{\frac{n}{2}}}\,|f(z')|\,dz'
\nonumber\\[4pt]
&\leq C_\kappa\,\sum_{j=0}^\infty 2^{-j}\aver{B_{n-1}(x',2^{j}t)}\,|f(z')|\,dz'
\leq C_\kappa{\mathcal{M}}f(x'),\nonumber
\end{align}
from which \eqref{exTGFVC} follows.

Let us now deal with \eqref{exTGFVC.2s}. To this end, abbreviate
\begin{equation}\label{vzJBb}
A:=\int_{\mathbb{R}^{n-1}}P(x')\,dx'\in{\mathbb{C}}^{M\times M}.
\end{equation}
Also, select a function $f$ as in \eqref{eq:aaAa} and introduce
\begin{equation}\label{defi-u-general}
u(x',t):=(P_t\ast f)(x')\,\,\mbox{ for each }\,\,\,(x',t)\in\mathbb{R}^{n}_{+}.
\end{equation}
From what we have proved already, this function is well-defined by an absolutely convergent
integral. In the remainder of the proof, we shall adapt the argument in \cite[p.\,198]{St70},
where the case $L=\Delta$ and $f\in L^p({\mathbb{R}}^{n-1})$, $1\leq p\leq\infty$, has been
treated. Specifically, fix a Lebesgue point $x'_0\in{\mathbb{R}}^{n-1}$ of $f$ and let
$\varepsilon>0$ be arbitrary. Then there exists $\delta>0$ such that
\begin{equation}\label{Drsgy}
\aver{B_{n-1}(0',r)}\big|f(z'+x'_0)-f(x'_0)\big|\,dz'<\varepsilon,\qquad
\forall\,r\in(0,\delta].
\end{equation}
In particular, if we set
\begin{equation}\label{eq:Gbab}
g:=\big[f(\cdot+x'_0)-f(x'_0)\big]{\bf 1}_{B_{n-1}(0',\delta)}\,\,\,
\mbox{ in }\,\,{\mathbb{R}}^{n-1},
\end{equation}
then \eqref{Drsgy} implies (for some dimensional constant $c_n>0$)
\begin{equation}\label{Drsgy-0}
{\mathcal{M}}g(0')\leq c_n\varepsilon.
\end{equation}
Then, bearing in mind \eqref{vzJBb}, for each $y'\in{\mathbb{R}}^{n-1}$ and
$t\in(0,\infty)$ we may write
\begin{align}\label{Drsgy-11}
u(y'+x'_0,t)-Af(x'_0)
&= \int_{{\mathbb{R}}^{n-1}}P_t(y'+x'_0-z')[f(z')-f(x'_0)]\,dz'
\\[4pt]
&= \int_{{\mathbb{R}}^{n-1}}P_t(y'-z')[f(z'+x'_0)-f(x'_0)]\,dz'.\nonumber
\end{align}
In turn, this and \eqref{TUhfg} then imply that, under the assumption that
$y'\in{\mathbb{R}}^{n-1}$ and $t\in(0,\infty)$ satisfy $|y'|<\kappa t$, we have
\begin{align}\label{Drsgy-10}
&\big|u(y'+x'_0,t)-Af(x'_0)\big|
\\
&\qquad\qquad\leq C_\kappa\int_{\{z'\in{\mathbb{R}}^{n-1}:\,|z'|<\delta\}}\frac{t}{(t^2+|z'|^2)^{\frac{n}{2}}}
\,|f(z'+x'_0)-f(x'_0)|\,dz'
\nonumber\\[4pt]
&\qquad\qquad\qquad \,\,+C_\kappa\int_{\{z'\in{\mathbb{R}}^{n-1}:\,|z'|\geq\delta\}}\frac{t}{(t^2+|z'|^2)^{\frac{n}{2}}}
\,|f(z'+x'_0)-f(x'_0)|\,dz'
\nonumber\\[4pt]
&
\qquad\qquad=: I_1+I_2.\nonumber
\end{align}
Note that thanks to \eqref{eq:Gbab}, \eqref{Uah-TTT}, and \eqref{exTGFVC} (used with
$P=P^\Delta$, $f=g$, and $x'=y'=0$), for some constant $C_\kappa\in(0,\infty)$ independent of
$\varepsilon$ and $f$ we have
\begin{multline}\label{Drsgy-1jab.2}
I_1=C_\kappa\int_{{\mathbb{R}}^{n-1}}\frac{t}{(t^2+|z'|^2)^{\frac{n}{2}}}\,|g(z')|\,dz'
\\
=C_\kappa(P^{\Delta}_t\ast |g|)(0')\leq C_\kappa{\mathcal{M}}g(0')\leq C_\kappa\varepsilon,
\end{multline}
where the last inequality is \eqref{Drsgy-0}. As regards $I_2$,
we first observe that if $|z'|\ge \delta$ then
\begin{equation}
1+|z'+x_0'|\leq(1+|x_0'|)\,(1+|z'|)\leq(1+|x_0'|)\,(1+\delta^{-1})\,|z'|.
\end{equation}
Thus,
\begin{align}\label{Drsgy-8}
I_2 & \leq C_\kappa t\int_{\{z'\in{\mathbb{R}}^{n-1}:\,|z'|\geq\delta\}}\frac{1}{|z'|^n}
\,\big|f(z'+x'_0)-f(x'_0)\big|\,dz'
\\[4pt]
& \leq C t\,\Bigg(
\int_{\{z'\in{\mathbb{R}}^{n-1}:\,|z'|\geq\delta\}}\frac{|f(z'+x'_0)|}{(1+|z'+x_0|)^n}\,dz'
+\frac{|f(x_0')|}{\delta}\Bigg)
\nonumber\\[4pt]
& \leq C\,t\left(\int_{\mathbb{R}^{n-1}} \frac{|f(x')|}{1+|x'|^n}\,dx'+|f(x_0')|\right),
\nonumber
\end{align}
where $C$ depends only on $n,\kappa,x_0'$, and $\delta$.
Hence $\lim\limits_{t\to 0^+}I_2=0$. This, \eqref{Drsgy-1jab.2}, and \eqref{Drsgy-10} then imply
\begin{equation}\label{Drsgy-7}
\limsup_{|y'|<\kappa t,\,t\to 0^+}
\big|u(y'+x'_0,t)-Af(x'_0)\big|\leq C_\kappa\varepsilon,
\end{equation}
for some $C_\kappa\in(0,\infty)$ independent of $\varepsilon$ and $f$. Now the claim in
\eqref{exTGFVC.2s} is clear from \eqref{Drsgy-7} and \eqref{vzJBb}-\eqref{defi-u-general}.
Finally, \eqref{Pt-abs-conv} implies $u\in L^1_{\rm loc}({\mathbb{R}}^n_{+})$, and this
takes care of \eqref{exTefef}.
\end{proof}

After these preparations, the proof of Theorem~\ref{thm:existence} is short and straightforward.

\vskip 0.08in
\begin{proof}[Proof of Theorem~\ref{thm:existence}]
That $u$ in \eqref{exist:u} is well-defined and satisfies \eqref{exist:Nu-Mf} as well as
$u\bigl|_{\partial\mathbb{R}^{n}_{+}}^{{}^{\rm n.t.}}=f$ a.e.~in $\mathbb{R}^{n-1}$ follows
immediately from Lemma~\ref{lennii}, Theorem~\ref{ya-T4-fav}, and the normalization of the
Poisson kernel (cf. part $(b)$ in Definition~\ref{defi:Poisson}). Next, given a
multi-index $\alpha\in{\mathbb{N}}_0^n$, from \eqref{eq:KjG} if $|\alpha|\geq 1$ and from
\eqref{eq:KDEF} combined with part $({\rm a})$ in Definition~\ref{defi:Poisson} if $|\alpha|=0$
we see that there exists a constant $C_\alpha\in(0,\infty)$ with the property that
\begin{equation}\label{derv-K}
\big|(\partial^\alpha K^L)(x',t)\big|
\leq C_\alpha\,t^{-|\alpha|}\,\frac{t}{(t+|x'|)^n},\qquad\forall\,(x',t)\in{\mathbb{R}}^n_{+}.
\end{equation}
In concert with \eqref{Pt-abs-conv}, this justifies differentiation under the
integral defining $u$ so, ultimately, $u\in\mathscr{C}^\infty(\mathbb{R}^n_{+})$.
Moreover, $Lu=0$ in ${\mathbb{R}}^n_{+}$ by \eqref{exist:u}, part ${\it (c)}$ in
Definition~\ref{defi:Poisson}, and part ${\it (ii)}$ in Remark~\ref{Ryf-uyf}.
\end{proof}

\begin{remark}\label{hyFF.a}
In the proof of Theorem~\ref{thm:existence}, the construction of a function
$u$ satisfying \eqref{exist:u2} is based on the formula \eqref{exist:u} in which
$P^L$ is the Agmon-Douglis-Nirenberg Poisson kernel for $L$ from Theorem~\ref{ya-T4-fav}.
Such a choice ensured that \eqref{derv-K} holds which, as noted in Remark~\ref{YTrrea},
is not immediately clear for a ``generic" Poisson kernel in the sense of Definition~\ref{defi:Poisson}.
Later on, in Theorem~\ref{taf87h6g}, we shall actually show that there exists precisely
one Poisson kernel for the system $L$ in the sense of Definition~\ref{defi:Poisson},
so this issue will eventually become a moot point. This being said, in the proof of
Theorem~\ref{taf87h6g} it is important to know that
\begin{align}\label{jnabn88}
\begin{array}{c}
\mbox{for any $P^L$ as in Definition~\ref{defi:Poisson}, properties }
\\[4pt]
\mbox{\eqref{exist:u2} and \eqref{exist:Nu-Mf} remain valid for $u$ as in \eqref{exist:f}-\eqref{exist:u}.}
\end{array}
\end{align}
To see that this is indeed the case, assume that $P^L$ is as in Definition~\ref{defi:Poisson}.
Then, if $u$ is as in \eqref{exist:f}-\eqref{exist:u}, it follows from Lemma~\ref{lennii} that
$u\in L^1_{\rm loc}({\mathbb{R}}^n_{+})$, $u\big|_{\partial\mathbb{R}^{n}_{+}}^{{}^{\rm n.t.}}=f$
a.e.~in $\mathbb{R}^{n-1}$, and ${\mathcal{N}}u\leq C{\mathcal{M}}f$. As such, there remains
to show that $u\in\mathscr{C}^\infty(\mathbb{R}^n_{+})$ and $Lu=0$ in $\mathbb{R}^{n}_{+}$.
The strategy is to prove that $Lu=0$ in the sense of distributions in $\mathbb{R}^{n}_{+}$,
which then forces $u\in\mathscr{C}^\infty(\mathbb{R}^n_{+})$ by elliptic regularity
(cf. \cite[Theorem~10.9, p.\,318]{DM}).
With this goal in mind, pick an arbitrary vector-valued test function
$\varphi\in{\mathscr{C}}^\infty_0({\mathbb{R}}^n_{+})$ and,
with $L^\top$ denoting the transposed of $L$, compute
\begin{align}\label{eq:hCa}
&\int_{{\mathbb{R}}^n_{+}}\big\langle u(x)\,,\,(L^\top\varphi)(x)\big\rangle\,dx
\\
&\hskip.5cm=\int_{{\mathbb{R}}^n_{+}}\Big\langle\int_{{\mathbb{R}}^{n-1}}P^L_t(x'-y')f(y')\,dy'
\,,\,(L^\top\varphi)(x',t)\Big\rangle\,dx'dt
\nonumber\\[4pt]
&\hskip.5cm=\int_{{\mathbb{R}}^{n-1}}\left(\int_{{\mathbb{R}}^n_{+}}\Big\langle P^L_t(x'-y')f(y')
\,,\,(L^\top\varphi)(x',t)\Big\rangle\,dx'dt\right)\,dy'
\nonumber\\[4pt]
&\hskip.5cm=\int_{{\mathbb{R}}^{n-1}}\left(\int_{{\mathbb{R}}^n_{+}}\Big\langle K^L(x)f(y')
\,,\,\big[L^\top\big(\varphi(\cdot+(y',0))\big)\big](x)\Big\rangle\,dx\right)\,dy'
\nonumber\\[4pt]
&\hskip.5cm=0.\nonumber
\end{align}
Above, the first equality uses \eqref{exist:u}, the second one is based on Fubini's theorem
(whose applicability is ensured by \eqref{Pt-abs-conv}), the third employs the definition of
$K^L$ and a natural change of variables, while the fourth one follows from \eqref{uahgab-UBVCX}.
Hence, $Lu=0$ in the sense of distributions in $\mathbb{R}^{n}_{+}$, and the proof of \eqref{jnabn88}
is complete.
\end{remark}

We now turn to the task of proving Theorem~\ref{thm:uniqueness}, which is the
key technical result of this paper. In the process, we shall make use of all the
auxiliary results from Appendix~\ref{sect:Green}, which the reader is invited to
review at this stage.

\vskip 0.08in
\begin{proof}[Proof of Theorem~\ref{thm:uniqueness}]
Fix $\kappa>0$ and let $u=(u_\beta)_{1\leq\beta\leq M}\in\mathscr{C}^\infty(\mathbb{R}^n_{+})$
be such that $Lu=0$ in $\mathbb{R}_{+}^n$, $\mathcal{N}_\kappa u$ satisfies \eqref{unq:Nu}, and
$u\big|_{\partial\mathbb{R}^{n}_{+}}^{{}^{\rm n.t.}}=0$ a.e. in ${\mathbb{R}}^{n-1}$.
The goal is to show that $u\equiv 0$ in ${\mathbb{R}}^n_{+}$.
To this end, fix an arbitrary point $x^\star\in\mathbb{R}^n_{+}$ and consider
the Green function $G=G(\,\cdot\,,x^\star)$ in $\mathbb{R}^n_{+}$ with
pole at $x^\star$ for $L^\top$, the transposed of the operator $L$
(cf.~Definition~\ref{ta.av-GGG} and Theorem~\ref{ta.av-GGG.2A} for
details on this matter). By design, this is a matrix-valued
function, say $G=(G_{\alpha\gamma})_{1\leq\alpha,\gamma\leq M}$.

We shall apply Theorem~\ref{theor:div-thm} to a suitably chosen vector field
and compact set. To set the stage, consider the compact set
\begin{equation}\label{eq:CpkT}
K_\star:=\overline{B(x^\star,r)}\subset{\mathbb{R}}^n_{+},\quad\mbox{ where }\,\,
r:=\tfrac34\,{\rm dist}\,(x^\star,\partial{\mathbb{R}}^n_{+}).
\end{equation}
Also, consider a function
\begin{equation}\label{TRavPPPi.i}
\begin{array}{c}
\psi\in{\mathscr{C}}^\infty(\mathbb{R})\,\,\mbox{ with the property that }\,\,
0\leq\psi\leq 1,
\\[4pt]
\psi(t)=0\,\,\mbox{ for }\,\,t\leq 1,
\,\,\,\mbox{ and }\,\,\,\psi(t)=1\,\,\mbox{ for }\,\,t\geq 2.
\end{array}
\end{equation}
Fix (for now) some $\varepsilon\in(0,r/4)$, and define
\begin{equation}\label{TRavPPPi.ii}
\psi_\varepsilon(x):=\psi(x_n/\varepsilon)\,\,\mbox{ for each }\,\,
x=(x_1,\dots,x_n)\in\mathbb{R}^n.
\end{equation}
In particular, the conditions on $\varepsilon$ and $r$ ensure that
\begin{equation}\label{oasg}
\psi_\varepsilon(x^\star)=1.
\end{equation}
To proceed, fix $\gamma\in\{1,\dots,M\}$ and define in ${\mathbb{R}}^n_{+}$ (as usual, using the summation
convention over repeated indices)
\begin{equation}\label{Yabnb-7t5.1}
\vec{F}:=\Big(\psi_\varepsilon\, G_{\alpha\gamma}\,a^{\alpha\beta}_{jk}\,
\partial_k u_\beta-\psi_\varepsilon\, u_\alpha a^{\beta\alpha}_{kj}\,
\partial_k G_{\beta\gamma}-u_\beta\,G_{\alpha\gamma}\,
a^{\alpha\beta}_{kj}\,\partial_k\psi_\varepsilon\Big)_{1\leq j\leq n}.
\end{equation}
From \eqref{maKnaTTGB}, \eqref{mainest2G.1}, \eqref{mainest2G.2}, and \eqref{Yabnb-7t5.1}
it follows that $\vec{F}\in L^1_{\rm loc}({\mathbb{R}}^n_{+},\mathbb{C}^n)$, and a direct
calculation shows that ${\rm div}\,\vec{F}$ (considered in the sense of distributions in
${\mathbb{R}}^n_{+}$) is given by
\begin{align}\label{Hab-ufaf}
{\rm div}\vec{F}
&=
(\partial_j\psi_\varepsilon)\,G_{\alpha\gamma}\,a^{\alpha\beta}_{jk}\,\partial_k u_\beta
+\psi_\varepsilon\,(\partial_j G_{\alpha\gamma})\,a^{\alpha\beta}_{jk}\,\partial_k u_\beta
\\[4pt]
&\qquad\quad
+\psi_\varepsilon\,G_{\alpha\gamma}\,a^{\alpha\beta}_{jk}\,(\partial_j\partial_k u_\beta)
-(\partial_j\psi_\varepsilon)\,u_\alpha\,a^{\beta\alpha}_{kj}\,\partial_k G_{\beta\gamma}
\nonumber\\[4pt]
&\qquad\quad
-\psi_\varepsilon\,(\partial_j u_\alpha)\,a^{\beta\alpha}_{kj}\,\partial_k G_{\beta\gamma}
-\psi_\varepsilon\,u_\alpha\,a^{\beta\alpha}_{kj}\,(\partial_j\partial_k G_{\beta\gamma})
\nonumber\\[4pt]
&\qquad\quad
-(\partial_j u_\beta)\,G_{\alpha\gamma}\,a^{\alpha\beta}_{kj}\,\partial_k\psi_\varepsilon
-u_\beta\,(\partial_j G_{\alpha\gamma})\,a^{\alpha\beta}_{kj}\,\partial_k\psi_\varepsilon
\nonumber\\[4pt]
&\qquad\quad
-u_\beta\,G_{\alpha\gamma}\,a^{\alpha\beta}_{kj}\,(\partial_j\partial_k\psi_\varepsilon)
\nonumber\\[4pt]
&
=:I_{1}+I_{2}+I_{3}+I_{4}+I_{5}+I_{6}+I_{7}+I_{8}+I_{9},\nonumber
\end{align}
where the last equality defines the $I_i$'s. Let us analyze some of these terms.
Changing variables $j'=k$ and $k'=j$ in $I_1$ yields
\begin{equation}\label{BaYNB-1}
I_1=(\partial_{k'}\psi_\varepsilon)\,G_{\alpha\gamma}\,
a^{\alpha\beta}_{k'j'}\,\partial_{j'} u_\beta=-I_7.
\end{equation}
For $I_2$ we change variables $j'=k$, $k'=j$, $\alpha'=\beta$,
$\beta'=\alpha$ in order to write
\begin{equation}\label{BaYNB-2}
I_2=\psi_\varepsilon\, (\partial_{k'} G_{\beta'\gamma})\,
a^{\beta'\alpha'}_{k'j'}\,\partial_{j'} u_{\alpha'}=-I_5.
\end{equation}
As regards $I_3$, we have
\begin{equation}\label{BaYNB-3}
I_3=\psi_\varepsilon\,G_{\alpha\gamma}\,(Lu)_\alpha=0,
\end{equation}
by the assumptions on $u$. For $I_6$ we observe that
(with $G_{\cdot\,\gamma}:=(G_{\mu\gamma})_{\mu}$)
\begin{equation}\label{BaYNB-4}
I_6=-\psi_\varepsilon\,u_\alpha (L^\top G_{\cdot\,\gamma})_\alpha
=-\psi_\varepsilon\,u_\alpha\delta_{\alpha\gamma}\delta_{x^\star}
=-\psi_\varepsilon\,u_\gamma\,\delta_{x^\star},
\end{equation}
thanks to \eqref{GHCewd-23.RRe} where we recall that $G=G(\cdot,x^\star)$ is the
Green function for $L^\top$ with pole at $x^\star$. Collectively, these
equalities permit us to conclude that
\begin{multline}\label{div-F}
{\rm div}\vec{F}  =-\psi_\varepsilon\,u_\gamma\,\delta_{x^\star}
-(\partial_j\psi_\varepsilon)\,u_\alpha a^{\beta\alpha}_{kj}\,\partial_k G_{\beta\gamma}
\\[4pt]
 \,\quad
-u_\beta\,(\partial_j G_{\alpha\gamma})a^{\alpha\beta}_{kj}\,\partial_k\psi_\varepsilon
-u_\beta\,G_{\alpha\gamma}\,a^{\alpha\beta}_{kj}\,(\partial_j\partial_k\psi_\varepsilon)
\,\,\,\mbox{ in }\,\,{\mathcal{D}}'({\mathbb{R}}^n_{+}).
\end{multline}
Notice that the first term in the right-hand side is a distribution supported at
the singleton $\{x^\star\}$ and therefore is in $\mathcal{E}'_{K_\star}({\mathbb{R}}^n_{+})$.
The remaining terms are in $L^1({\mathbb{R}}^n_{+})$, as seen from estimates
\eqref{Gruah}, \eqref{Gruah.3} established below. Thus, condition
$(a)$ in Theorem~\ref{theor:div-thm} holds.

To verify condition $(c)$ in Theorem~\ref{theor:div-thm}
we first observe that $\psi_\varepsilon\equiv 0$ in the horizontal strip
$\{x=(x_1,\dots,x_n)\in{\mathbb{R}}^n:\,0<x_n<\varepsilon\}$. In light of
\eqref{Yabnb-7t5.1}, this clearly implies that
\begin{equation}\label{Yagav-8iah}
\vec{F}\big|_{\partial{\mathbb{R}}^n_{+}}^{{}^{\rm n.t.}}=0
\,\,\mbox{ everywhere on }\,\,\partial\mathbb{R}^n_{+}.
\end{equation}
Let us now turn our attention to condition $(b)$ in Theorem~\ref{theor:div-thm}.
This is a purely qualitative membership, so bounds depending on $\varepsilon$
and $x^\star$ are permissible. We first observe from \eqref{Yabnb-7t5.1}
that there exists $C\in(0,\infty)$ such that
\begin{equation}\label{Yagav-8tds}
|\vec{F}|\leq C{\bf 1}_{\{x_n\geq\varepsilon\}}|G|\,|\nabla u|
+C{\bf 1}_{\{x_n\geq\varepsilon\}}|u|\,|\nabla G|
+C\varepsilon^{-1}{\bf 1}_{\{\varepsilon\leq x_n\leq 2\varepsilon\}}|u||G|,
\end{equation}
in ${\mathbb{R}}^n_{+}$.
Pick a point $x'\in\mathbb{R}^{n-1}$, and select
$y=(y',y_n)\in\Gamma_{\kappa/4}(x')$ with $y_n\geq\varepsilon$ (where $\kappa>0$ was fixed above).
Let $\rho:=\min\big\{1/4\,,\,9\kappa/(16+4\kappa)\big\}$. We claim that
\begin{equation}\label{eq:BFB}
B(y,\rho\,\varepsilon)\subset\Gamma_\kappa(x').
\end{equation}
Indeed, if $z=(z',z_n)\in B(y,\rho\,\varepsilon)$ then $z_n>3\,\varepsilon/4$
and $y_n<z_n+\varepsilon\,\rho$. Hence, $y_n<(4\,\rho/3+1)\,z_n$. Since
$|y'-x'|<(\kappa/4)\,y_n$ also holds, we obtain
\begin{equation}\label{Iy6ty}
|z'-x'|\leq|z-y|+|y'-x'|<\rho\,\varepsilon+(\kappa/4)\,y_n
<\big(4\,\rho/3+\kappa\rho/3+\kappa/4\big)\,z_n\leq\kappa\,z_n,
\end{equation}
ultimately proving \eqref{eq:BFB}. Using this and the interior estimates from Theorem~\ref{ker-sbav}
we may therefore write
\begin{multline}\label{TDY-igfd}
|\nabla u(y)|\leq C\,(\rho\,\varepsilon)^{-1}\aver{B(y,\rho\,\varepsilon)}|u(z)|\,dz
\\
\leq C\,\varepsilon^{-1}\sup_{z\in B(y,\rho\,\varepsilon)}|u(z)|
\leq C\,\varepsilon^{-1}\,{\mathcal{N}}_{\kappa}u(x').
\end{multline}
Next, consider a point $y=(y',y_n)\in\Gamma_{\kappa/4}(x')\setminus K_\star$
with $y_n\geq\varepsilon$. Then, as before, \eqref{eq:BFB} holds.
Let us also note that any $z\in B(y,2\,\rho\,\varepsilon)$ satisfies
$|z-x^\star|>3r/4$ since, upon recalling that $\varepsilon<r/2$ and $\rho\leq 1/4$,
we may estimate
\begin{equation}\label{yrdfd}
r<|y-x^\star|\leq|y-z|+|z-x^\star|
<2\,\rho\,\varepsilon+|z-x^\star|
\leq r/4+|z-x^\star|.
\end{equation}
Thus,
\begin{equation}\label{eq:BFB2}
B(y,2\,\rho\,\varepsilon)\cap\overline{B(x^\star,3r/4)}=\emptyset.
\end{equation}
In particular, $L^\top G=0$ in $B(y,2\,\rho\,\varepsilon)$. As such, we can use interior
estimates for $G$ (cf.~Theorem~\ref{ker-sbav}) in this ball and \eqref{bound-NK-G}
in order to write (with the help of \eqref{eq:BFB} and \eqref{eq:BFB2})
\begin{align}\label{TDY-igfd.2}
|\nabla G(y)| & \leq C\,(\rho\,\varepsilon)^{-1}\aver{B(y,\rho\,\varepsilon)}|G(z)|\,dz
\leq C\,\varepsilon^{-1}\sup_{z\in B(y,\rho\,\varepsilon)}|G(z)|
\\[4pt]
& \leq C\,\varepsilon^{-1}\,{\mathcal{N}}_{\kappa}^{K^c}G(x')
\leq C\,\varepsilon^{-1}\,\frac{1+\log_{+}|x'|}{1+|x'|^{n-1}},\nonumber
\end{align}
where
$
K:=\overline{B(x^\star,3r/4)}\subset{\mathbb{R}}^n_{+}.
$

From \eqref{Yagav-8tds}, \eqref{TDY-igfd}, \eqref{TDY-igfd.2}, and \eqref{bound-NK-G} we deduce that
\begin{align}\label{ajabIHB}
{\mathcal{N}}_{\kappa/4}^{K_\star^c}\vec{F}(x')\leq
C_{\varepsilon,\kappa,x^\star}\,{\mathcal{N}}_{\kappa}u(x')\,
\frac{1+\log_{+}|x'|}{1+|x'|^{n-1}},\qquad\forall\,x'\in{\mathbb{R}}^{n-1}.
\end{align}
Consequently, based on \eqref{ajabIHB} and the assumption on $\mathcal{N}_\kappa u$ in
\eqref{unq:Nu}, we obtain that
\begin{align}\label{ajabIHB.Dv}
\int_{\mathbb{R}^{n-1}} {\mathcal{N}}_{\kappa/4}^{K_\star^c}\vec{F}(x')\,dx'
\leq C_{\varepsilon,\kappa,x^\star}\,\int_{\mathbb{R}^{n-1}}{\mathcal{N}}_{\kappa}u(x')\,
\frac{1+\log_{+}|x'|}{1+|x'|^{n-1}}\,dx'<\infty.
\end{align}
The above estimate shows that ${\mathcal{N}}_{\kappa/4}^{K_\star^c}
\vec{F}\in L^1(\partial{\mathbb{R}}^n_{+})$ which, together with \eqref{N-Nal.11a},
implies ${\mathcal{N}}_{\kappa}^{K_\star^c}\vec{F}\in L^1(\partial{\mathbb{R}}^n_{+})$.
Hence, condition $(b)$ in Theorem~\ref{theor:div-thm} holds as well.

Having verified all hypotheses in Theorem~\ref{theor:div-thm}, from
\eqref{eqn:div-form}, \eqref{Yagav-8iah}, \eqref{div-F}, and \eqref{oasg}, we obtain that
\begin{align}\label{Uahg-ihG}
0 &=-\int_{\partial\mathbb{R}^n_{+}}e_n\cdot
\bigl(\vec F\,\big|^{{}^{\rm n.t.}}_{\partial\mathbb{R}^n_{+}}\bigr)\,d{\mathscr{L}}^{n-1}
={}_{(\mathscr{C}_b^\infty(\mathbb{R}^n_{+}))^\ast}\big\langle{\rm div}\vec{F},1
\big\rangle_{\mathscr{C}_b^\infty(\mathbb{R}^n_{+})}
\\[4pt]
& =-u_\gamma(x^\star)\,-\int_{\mathbb{R}^n_{+}}(\partial_j\psi_\varepsilon)\,
u_\alpha a^{\beta\alpha}_{kj}\,\partial_k G_{\beta\gamma}\,d\mathscr{L}^n
\nonumber\\[4pt]
& \qquad
-\int_{\mathbb{R}^n_{+}}u_\beta\,(\partial_j G_{\alpha\gamma})a^{\alpha\beta}_{kj}\,
\partial_k\psi_\varepsilon\,d\mathscr{L}^n
-\int_{\mathbb{R}^n_{+}}u_\beta\,G_{\alpha\gamma}\,a^{\alpha\beta}_{kj}\,
(\partial_j\partial_k\psi_\varepsilon)\,d\mathscr{L}^n.\nonumber
\end{align}
We claim that the three integrals in the rightmost side of \eqref{Uahg-ihG} converge to zero
as $\varepsilon\to 0^+$. This, in turn, will imply that $u(x^\star)=0$ and
since $x^\star\in\mathbb{R}^n_{+}$ is an arbitrary point we may ultimately conclude
that $u\equiv 0$ in ${\mathbb{R}}^n_{+}$, as desired.

An inspection of the aforementioned integrals reveals that we need to prove that
\begin{equation}\label{Green-2terms}
\lim_{\varepsilon\to 0^{+}}\frac1{\varepsilon}\,
\int_{\{x=(x',\,x_n)\in\mathbb{R}^n:\,\varepsilon<x_n<2\,\varepsilon\}}
|u|\,|\nabla G|d\mathscr{L}^n=0,
\end{equation}
and
\begin{equation}\label{Green-2terms.H}
\lim_{\varepsilon\to 0^{+}}
\frac1{\varepsilon^2}\,\int_{\{x=(x',\,x_n)\in\mathbb{R}^n:\,\varepsilon<x_n<2\,\varepsilon\}}
|u|\,|G|\,d\mathscr{L}^n=0.
\end{equation}
From \eqref{ghagUGDS}, and the Mean Value Theorem we have
\begin{align}\label{Gruah}
& \frac1{\varepsilon^2}\,\int_{\{x=(x',\,x_n)\in\mathbb{R}^n:\,\varepsilon<x_n<2\,\varepsilon\}}
|u|\,|G|\,d\mathscr{L}^n
\\[4pt]
& \hskip 0.250in
=\frac1{\varepsilon^2}\,\int_{\{x=(x',\,x_n)\in\mathbb{R}^n:\,\varepsilon<x_n<2\,\varepsilon\}}
|u(x',x_n)|\,|G(x',x_n)-G(x',0)|\,dx
\nonumber\\[4pt]
& \hskip 0.250in
\leq\frac{C}{\varepsilon}\,\int_{\{x=(x',\,x_n)\in\mathbb{R}^n:\,\varepsilon<x_n<2\,\varepsilon\}}
|u(x',x_n)|\,\Big(\sup_{0<t<2\varepsilon}|(\partial_{x_n}G)(x',t)|\,\Big)\,dx
\nonumber\\[4pt]
& \hskip 0.250in
\leq\frac{C}{\varepsilon}\,\int_{\{x=(x',\,x_n)\in\mathbb{R}^n:\,\varepsilon<x_n<2\,\varepsilon\}}
|u(x',x_n)|\,\Big(\sup_{0<t<2\varepsilon}|(\nabla G)(x',t)|\,\Big)\,dx.\nonumber
\end{align}
Note that the integral in \eqref{Green-2terms} can also be controlled by the
last quantity in \eqref{Gruah}. Therefore, matters have been reduced to showing that
\begin{equation}\label{Gruah.2}
\lim_{\varepsilon\to 0^{+}}\frac{1}{\varepsilon}\,
\int_{\{x=(x',\,x_n)\in\mathbb{R}^n:\,\varepsilon<x_n<2\,\varepsilon\}}
|u(x',x_n)|\,\Big(\sup_{0<t<2\varepsilon}|(\nabla G)(x',t)|\,\Big)\,dx=0.
\end{equation}
In this regard, by \eqref{bouMNN} we have
\begin{align}\label{Gruah.3}
& \frac{C}{\varepsilon}\,\int_{\{x=(x',\,x_n)\in\mathbb{R}^n:\,\varepsilon<x_n<2\,\varepsilon\}}
|u(x',x_n)|\,\Big(\sup_{0<t<2\varepsilon}|(\nabla G)(x',t)|\,\Big)\,dx
\\[4pt]
& \hskip 0.50in
\leq C\,\int_{\mathbb{R}^{n-1}}\mathcal{N}^{(2\varepsilon)}_\kappa u(x')\,
\mathcal{N}^{(2\varepsilon)}_\kappa(\nabla G)(x')\,dx'
\nonumber\\[4pt]
& \hskip 0.50in
\leq C\,\int_{\mathbb{R}^{n-1}}\mathcal{N}^{(2\varepsilon)}_\kappa u(x')\,
\mathcal{N}^{K^c}_\kappa(\nabla G)(x')\,dx'
\nonumber\\[4pt]
& \hskip 0.50in
\leq C\,\int_{\mathbb{R}^{n-1}}\mathcal{N}^{(2\varepsilon)}_\kappa u(x')\,
\frac{1}{1+|x'|^{n-1}}\,dx'
\nonumber\\[4pt]
& \hskip 0.50in
\leq C\,\int_{\mathbb{R}^{n-1}}\mathcal{N}^{(2\varepsilon)}_\kappa u(x')\,
\frac{1+\log_{+}|x'|}{1+|x'|^{n-1}}\,dx',\nonumber
\end{align}
where $\mathcal{N}^{(2\varepsilon)}_\kappa$ is the truncated nontangential
maximal function defined as in \eqref{Gruah.4}. Notice that $\lim\limits_{\varepsilon\to 0^{+}}
\mathcal{N}^{(2\varepsilon)}_\kappa u(x')=0$ for a.e.~$x'\in{\mathbb{R}}^{n-1}$
by the fact that $u\big|_{\partial\mathbb{R}^{n}_{+}}^{{}^{\rm n.t.}}=0$. On the other hand,
$0\leq\mathcal{N}^{(2\varepsilon)}_\kappa u(x')\leq\mathcal{N}_\kappa u(x')$
for each $x'\in{\mathbb{R}}^{n-1}$, and therefore the integrand is uniformly controlled by
an $L^1(\mathbb{R}^{n-1})$ function thanks to our assumption in \eqref{unq:Nu}.
Thus, the desired formula \eqref{Gruah.2} follows on account of \eqref{Gruah.3},
\eqref{unq:Nu}, and Lebesgue's Dominated Convergence Theorem.
This finishes the proof of Theorem~\ref{thm:uniqueness}.
\end{proof}

We conclude this section with a remark which, in particular, shows that in the special case when $L=\Delta$
the hypotheses in Theorem~\ref{thm:uniqueness} may be slightly relaxed.

\begin{remark}\label{Tgg-GRR}
If the system $L$ (assumed to be as in \eqref{L-def}-\eqref{L-ell.X}) is such that
its fundamental solution $E^L$ from Theorem~\ref{FS-prop} is a radial function
when restricted to ${\mathbb{R}}^n\setminus\{0\}$, then the logarithm in \eqref{unq:Nu} may be omitted.
This is seen by inspecting the proof of Theorem~\ref{thm:uniqueness} and making use of
part {\it (1)} in Theorem~\ref{ta.av-GGG.2A}.
\end{remark}

\section{Well-posedness for the Dirichlet problem}
\setcounter{equation}{0}
\label{S-4}

In this section, Theorems~\ref{thm:existence} and \ref{thm:uniqueness} will be used
to prove Theorem~\ref{Them-General}, Corollary~\ref{tfav.tRR}, Corollary~\ref{tfav.tRR.BEUR},
and Theorem~\ref{corol:Xw-Dir}.

\vskip 0.08in
\begin{proof}[Proof of Theorem~\ref{Them-General}]
For existence, invoke Theorem~\ref{thm:existence} (whose applicability is ensured by the
first condition in \eqref{Fi-AN.1}) and note that if $u$ is as in \eqref{exist:u}
then the first, second, and last conditions in \eqref{Dir-BVP-XY} are satisfied.
In addition, \eqref{eJHBb} is simply \eqref{exist:Nu-Mf}. Together, \eqref{Fi-AN.2} and \eqref{eJHBb}
then permit us to conclude that the third condition in \eqref{Dir-BVP-XY} is also satisfied.
Hence, $u$ solves \eqref{Dir-BVP-XY}. For uniqueness, assume that both $u_1$ and $u_2$ solve
\eqref{Dir-BVP-XY} for the same datum $f$ and set $u:=u_1-u_2\in{\mathscr{C}}^\infty({\mathbb{R}}^n_{+})$.
Then $Lu=0$ in ${\mathbb{R}}^n_{+}$ and, since
${\mathcal{N}}u_1,{\mathcal{N}}u_2\in{\mathbb{Y}}$, the estimate
$0\leq{\mathcal{N}}u\leq{\mathcal{N}}u_1+{\mathcal{N}}u_2\leq 2\max\big\{{\mathcal{N}}u_1\,,\,{\mathcal{N}}u_2\big\}$
forces ${\mathcal{N}}u\in{\mathbb{Y}}$ by the properties of the function lattice ${\mathbb{Y}}$.
Granted this, Theorem~\ref{thm:uniqueness} applies (thanks to the second condition in \eqref{Fi-AN.1})
and gives that $u\equiv 0$ in ${\mathbb{R}}^n_{+}$, hence $u_1=u_2$ as wanted.
\end{proof}

Before presenting the proof of Corollary~\ref{tfav.tRR}, some comments are in order.
Having fixed some $q\in(1,\infty)$ (whose actual choice is ultimately immaterial),
one may define the {\tt Hardy} {\tt space} $H^1(\mathbb{R}^{n-1})$ as
\begin{align}\label{eq:f-H1-atoms}
H^1(\mathbb{R}^{n-1}):=\Big\{f&\in L^1({\mathbb{R}}^{n-1}):\,
f=\sum_{j\in{\mathbb{N}}}\lambda_j\,a_j\quad\mbox{a.e.~in }\mathbb{R}^{n-1},
\\[0pt]
& \mbox{for some $(1,q)$-atoms $\{a_j\}_{j\in{\mathbb{N}}}$ and scalars
$\{\lambda_j\}_{j\in{\mathbb{N}}}\in\ell^1$}\Big\}.\nonumber
\end{align}
For each $f\in H^1(\mathbb{R}^{n-1})$ we then set
$\|f\|_{H^1(\mathbb{R}^{n-1})}:=\inf\sum_{j\in{\mathbb{N}}}|\lambda_j|$ with the
infimum taken over all atomic representations of $f$ as $\sum_{j\in{\mathbb{N}}}\lambda_j\,a_j$.

Recall that a Lebesgue measurable function $a:\mathbb{R}^{n-1}\rightarrow\mathbb{C}$
is said to be an $(1,q)$-{\tt atom} if, for some cube $Q\subset\mathbb{R}^{n-1}$, one has
\begin{equation}\label{defi-atom}
{\rm supp}\,a\subset Q,\qquad\|a\|_{L^q(\mathbb{R}^{n-1})}\leq |Q|^{1/q-1},
\qquad\int_{\mathbb{R}^{n-1}}a(y')\,dy'=0.
\end{equation}

\vskip 0.08in
\begin{proof}[Proof of Corollary~\ref{tfav.tRR}]
Having already established Theorem~\ref{Them-General}, we only need to check that
the implication in \eqref{eq:BV333} holds if ${\mathbb{X}}=H^1({\mathbb{R}}^{n-1})$ and
${\mathbb{Y}}=L^1({\mathbb{R}}^{n-1})$ (recall that we have
$H^1({\mathbb{R}}^{n-1})\subset L^1({\mathbb{R}}^{n-1})$, thus
\eqref{Fi-AN.1} and the first condition
in \eqref{Fi-AN.2} are clear for this choice). To this end, assume first that
\begin{equation}\label{eq:Fb}
u(x',t):=(P_t\ast a)(x'),\qquad\forall\,(x',t)\in{\mathbb{R}}^n_{+},
\end{equation}
where $a:\mathbb{R}^{n-1}\rightarrow\mathbb{C}^M$ is a Lebesgue measurable function (whose scalar
components are) as in \eqref{defi-atom}. Then \eqref{exist:Nu-Mf}, H\"older's inequality, the $L^q$-boundedness
of the Hardy-Littlewood maximal function, and the normalization of the atom permit us to write
\begin{align}\label{eq:NBV1uj}
\int_{2\sqrt{n}Q}{\mathcal{N}}_\kappa u\,d{\mathscr{L}}^{n-1}
& \leq C\int_{2\sqrt{n}Q}{\mathcal{M}}a\,d{\mathscr{L}}^{n-1}
\\[4pt]
&
\leq C|Q|^{1/q'}\Big(\int_{2\sqrt{n}Q}\big({\mathcal{M}}a\big)^q\,d{\mathscr{L}}^{n-1}\Big)^{1/q}
\nonumber\\[4pt]
&\leq C|Q|^{1/q'}\Big(\int_{\mathbb{R}^{n-1}}\big({\mathcal{M}}a\big)^q\,d{\mathscr{L}}^{n-1}\Big)^{1/q}
\nonumber\\[4pt]
&\leq C|Q|^{1/q'}\|a\|_{L^q(\mathbb{R}^{n-1})}\leq C,\nonumber
\end{align}
for some constant $C\in(0,\infty)$ depending only on $n,L,\kappa$.
To proceed, fix an arbitrary point $x'\in{\mathbb{R}}^{n-1}\setminus 2\sqrt{n}Q$.
If $\ell(Q)$ and $x'_Q$ are, respectively, the side-length and center of the cube $Q$,
this choice entails
\begin{equation}\label{eq:rEEb}
|z'-x_Q'|\leq\max\{\kappa,2\}\big(t+|z'-\xi'|\big),
\qquad\forall\,(z',t)\in\Gamma_\kappa(x'),\quad\forall\,\xi'\in Q.
\end{equation}
Indeed, if $(z',t)\in\Gamma_\kappa(x')$ and $\xi'\in Q$ then, first,
$|z'-x'_Q|\leq |z'-\xi'|+|\xi'-x'_Q|$ and, second,
\begin{equation}
|\xi'-x'_Q|\leq\frac{\sqrt{n}}{2}\ell(Q)\leq\frac{1}{2}|x'-x'_Q|\leq\frac{1}{2}(|x'-z'|+|z'-x'_Q|)
\leq\frac{1}{2}(\kappa t+|z'-x'_Q|),
\end{equation}
from which \eqref{eq:rEEb} follows.
Next, using \eqref{eq:KDEF}, the vanishing moment condition for the atom, the Mean Value Theorem
together with \eqref{derv-K} and \eqref{eq:rEEb}, H\"older's inequality and, finally, the
support and normalization of the atom, for each $(z',t)\in\Gamma_\kappa(x')$ we may estimate
\begin{align}\label{rrff4rf}
|(P^L_t\ast a)(z')| & =\Big|\int_{{\mathbb{R}}^{n-1}}\big[K^L(z'-y',t)-K^L(z'-x'_Q,t)\big]a(y')\,dy'\Big|
\\[4pt]
&\leq\int_{Q}\big|K^L(z'-y',t)-K^L(z'-x'_Q,t)\big||a(y')|\,dy'
\nonumber\\[4pt]
&\leq C\frac{\ell(Q)}{\big(t+|z'-x'_Q|\big)^n}\int_{Q}|a(y')|\,dy'
\nonumber\\[4pt]
&\leq C\frac{\ell(Q)}{\big(t+|z'-x'_Q|\big)^n}\,
|Q|^{1/q'}\|a\|_{L^q(\mathbb{R}^{n-1})}
\nonumber\\[4pt]
&\leq\frac{C\ell(Q)}{\big(t+|z'-x'_Q|\big)^n}.\nonumber
\end{align}
In turn, \eqref{rrff4rf} implies that for each
$x'\in{\mathbb{R}}^{n-1}\setminus 2\sqrt{n}Q$ we have
\begin{multline}\label{rrff4rf.2}
\big({\mathcal{N}}_\kappa u\big)(x')=\sup_{(z',t)\in\Gamma_\kappa(x')}|(P^L_t\ast a)(z')|
\\
\leq\sup_{(z',t)\in\Gamma_\kappa(x')}\frac{C\ell(Q)}{\big(t+|z'-x'_Q|\big)^n}
\le\frac{C\ell(Q)}{|x'-x'_Q|^n},
\end{multline}
hence
\begin{equation}\label{eq:NBVGaa}
\int_{{\mathbb{R}}^{n-1}\setminus 2\sqrt{n}Q}{\mathcal{N}}_\kappa u\,d{\mathscr{L}}^{n-1}
\leq C\int_{{\mathbb{R}}^{n-1}\setminus 2\sqrt{n}Q}\frac{\ell(Q)}{|x'-x'_Q|^n}\,dx'=C,
\end{equation}
for some constant $C\in(0,\infty)$ depending only on $n,L$. From \eqref{eq:NBV1uj} and
\eqref{eq:NBVGaa} we deduce that whenever $u$ is as in \eqref{eq:Fb} then
\begin{equation}\label{eq:NBhf}
\int_{{\mathbb{R}}^{n-1}}{\mathcal{N}}_\kappa u\,d{\mathscr{L}}^{n-1}\leq C,
\end{equation}
for some constant $C\in(0,\infty)$ independent of the atom.

To conclude, for each $p\in[1,\infty)$ define the tent spaces
\begin{equation}\label{eq:REWx}
{\mathcal{T}}^p({\mathbb{R}}^n_{+}):=\big\{u:{\mathbb{R}}^n_{+}\to{\mathbb{C}}^M:\,
\mbox{$u$ measurable and ${\mathcal{N}}_\kappa u\in L^p({\mathbb{R}}^{n-1})$}\big\}
\end{equation}
equipped with the norm $\|u\|_{{\mathcal{T}}^p({\mathbb{R}}^n_{+})}:=
\|{\mathcal{N}}_\kappa u\|_{L^p({\mathbb{R}}^{n-1})}$. It may be actually checked that the pair
$\big({\mathcal{T}}^p({\mathbb{R}}^n_{+})\,,\,\|\cdot\|_{{\mathcal{T}}^p({\mathbb{R}}^n_{+})}\big)$
is a K\"othe function space, relative to the background measure space $({\mathbb{R}}^n_{+},{\mathscr{L}}^n)$.
In this context, with $q\in(1,\infty)$ the exponent intervening in \eqref{eq:f-H1-atoms},
consider the assignment
\begin{equation}\label{Fi-AN.aman.2}
\begin{array}{c}
T:L^q({\mathbb{R}}^{n-1})\longrightarrow{\mathcal{T}}^q({\mathbb{R}}^n_{+})\,\,\mbox{ given by}
\\[4pt]
\mbox{$Tf:=u$, where $u$ is associated with $f$ as in \eqref{exist:u}}.
\end{array}
\end{equation}
Thanks to Theorem~\ref{thm:existence}, $T$ is a well-defined linear and bounded operator
and, given what we have just proved in \eqref{eq:NBhf}, it has the property that
$\|Ta\|_{{\mathcal{T}}^1({\mathbb{R}}^{n-1})}\leq C$
for every $(1,q)$-atom $a$, for some constant $C\in(0,\infty)$ independent of the atom in question.
Granted these, it follows (see \cite{AM} for very general results of this nature) that $T$ extends
as a linear and bounded operator from $H^1({\mathbb{R}}^{n-1})$ into
${\mathcal{T}}^1({\mathbb{R}}^{n-1})$. In light of \eqref{Fi-AN.aman.2}, this shows
that the implication in \eqref{eq:BV333} is indeed true
if ${\mathbb{X}},{\mathbb{Y}}$ are as in \eqref{eq:BVCc.2}. This proves that,
for each system $L$ as in \eqref{L-def}-\eqref{L-ell.X}, the corresponding
$(H^1,L^1)$-Dirichlet boundary value problem in $\mathbb{R}^{n}_{+}$,
\begin{equation}\label{Dir-BVP-H1L1}
\left\{
\begin{array}{l}
u\in{\mathscr{C}}^\infty({\mathbb{R}}^n_{+}),
\\[4pt]
Lu=0\,\,\mbox{ in }\,\,\mathbb{R}^{n}_{+},
\\[4pt]
\mathcal{N}_\kappa u\in L^1({\mathbb{R}}^{n-1}),
\\[6pt]
u\bigl|_{\partial\mathbb{R}^{n}_{+}}^{{}^{\rm n.t.}}=f\in H^1({\mathbb{R}}^{n-1}),
\end{array}
\right.
\end{equation}
has a unique solution. Moreover, the above argument also shows that the following
naturally accompanying estimate holds
\begin{equation}\label{Jbvv75r}
\|{\mathcal{N}}_\kappa u\|_{L^1({\mathbb{R}}^{n-1})}\leq C\,\|f\|_{H^1({\mathbb{R}}^{n-1})},
\end{equation}
for some $C=C(n,L,\kappa)\in(0,\infty)$. Hence, \eqref{Dir-BVP-H1L1} is well-posed.

In closing, we wish to note that one can give a proof of \eqref{Jbvv75r}, and also of \eqref{eq:BV333},
which avoids working with tent spaces by reasoning directly as follows (incidentally, this is also
going to be useful later on, in the proof of Corollary~\ref{tfav.tRR.BEUR}).
Let $f\in H^1({\mathbb{R}}^{n-1})$ and consider a quasi-optimal atomic decomposition,
say $f=\sum_{j\in{\mathbb{N}}}\lambda_ja_j$ with
\begin{equation}\label{Jbvv75r.A}
\frac{1}{2}\sum_{j\in{\mathbb{N}}}|\lambda_j|\leq\|f\|_{H^1({\mathbb{R}}^{n-1})}
\leq\sum_{j\in{\mathbb{N}}}|\lambda_j|.
\end{equation}
For each $N\in{\mathbb{N}}$, write $f_N:=\sum_{j=1}^N\lambda_ja_j$.
Clearly, $f_N\to f$ in $L^1({\mathbb{R}}^{n-1})$ as $N\to\infty$, hence also a.e.~after eventually
passing to a subsequence. To prove \eqref{Jbvv75r} in the case when
$u$ is defined as in \eqref{eq:FDww} we proceed as follows.
For each $N\in{\mathbb{N}}$, introduce $u_N(x',t):=(P^L_t\ast f_N)(x')$ for all
$(x',t)\in{\mathbb{R}}^n_{+}$ (this makes sense since $f_N\in L^1({\mathbb{R}}^{n-1})$).
Then, using \eqref{exist:Nu-Mf}, we may write
\begin{align}\label{Jbvv75r.B}
\|{\mathcal{N}}_\kappa u-{\mathcal{N}}_\kappa (u_N)\|_{L^{1,\infty}({\mathbb{R}}^{n-1})}
& \leq\|{\mathcal{N}}_\kappa (u-u_N)\|_{L^{1,\infty}({\mathbb{R}}^{n-1})}
\\[4pt]
& \leq C\,\|{\mathcal{M}}(f-f_N)\|_{L^{1,\infty}({\mathbb{R}}^{n-1})}
\nonumber\\[4pt]
& \leq C\,\|f-f_N\|_{L^1({\mathbb{R}}^{n-1})}\longrightarrow 0
\,\,\,\mbox{ as }\,\,N\to\infty.\nonumber
\end{align}
Thus, ${\mathcal{N}}_\kappa (u_N)\to{\mathcal{N}}_\kappa u$ in $L^{1,\infty}({\mathbb{R}}^{n-1})$
as $N\to\infty$ and, by passing to a subsequence $\{N_j\}_{j\in{\mathbb{N}}}$, we may ensure that
${\mathcal{N}}_\kappa (u_{N_j})\to{\mathcal{N}}_\kappa u$ pointwise a.e.~in ${\mathbb{R}}^{n-1}$
as $j\to\infty$ (cf., e.g., the discussion in \cite[Example~6, pp.\,4776-4777]{MMMZ}).
In turn, if for each $j\in{\mathbb{N}}$ we set $v_j(x',t):=(P^L_t\ast a_j)(x')$ for
$(x',t)\in{\mathbb{R}}^n_{+}$, this readily gives
\begin{equation}\label{Jbvv75r.C}
{\mathcal{N}}_\kappa u\leq\sum_{j\in{\mathbb{N}}}|\lambda_j|\,{\mathcal{N}}_\kappa (v_j)
\,\,\,\mbox{ a.e.~in }\,\,{\mathbb{R}}^{n-1}.
\end{equation}
From \eqref{Jbvv75r.C}, \eqref{eq:NBhf}, and \eqref{Jbvv75r.A} we then conclude that
\begin{equation}\label{Jbvv75r.D}
\|{\mathcal{N}}_\kappa u\|_{L^1({\mathbb{R}}^{n-1})}
\leq C\,\sum_{j\in{\mathbb{N}}}|\lambda_j|\leq C\|f\|_{H^1({\mathbb{R}}^{n-1})},
\end{equation}
finishing the alternative proof of \eqref{Jbvv75r} and the implication in \eqref{eq:BV333}.
\end{proof}

As a preamble to the proof of Corollary~\ref{tfav.tRR.BEUR} we first properly define the spaces
intervening in \eqref{eq:BVCc.2BEUR}. Given $p\in(1,\infty)$ define the {\tt Beurling} {\tt space}
${\rm A}^p(\mathbb{R}^{n-1})$ as the collection of $p$-th power locally
integrable functions $f$ in $\mathbb{R}^{n-1}$ satisfying (with $p'$ denoting the H\"older
conjugate exponent of $p$)
\begin{equation}\label{PP-h7y}
\|f\|_{{\rm A}^p(\mathbb{R}^{n-1})}:=\sum_{k=0}^\infty 2^{k(n-1)/p'}
\|f{\mathbf{1}}_{C_k}\|_{L^p({\mathbb{R}}^{n-1})}<\infty,
\end{equation}
where $C_0:=B_{n-1}(0',1)$ and
$C_k:=B_{n-1}(0',2^k)\setminus\overline{B_{n-1}(0',2^{k-1})}$ for each $k\in{\mathbb{N}}$.
This readily implies that
\begin{equation}\label{eq:rgGH}
\begin{array}{c}
\mbox{$\big({\rm A}^p(\mathbb{R}^{n-1}),\|\cdot\|_{{\rm A}^p(\mathbb{R}^{n-1})}\big)$
is a Banach space, which is a function }
\\[4pt]
\mbox{lattice, and embeds continuously into $L^1({\mathbb{R}}^{n-1})$}.
\end{array}
\end{equation}
Next, call a function $a\in L^1_{\rm loc}({\mathbb{R}}^{n-1})$ a
{\tt central} $(1,p)$-{\tt atom} provided there exists a cube $Q$ in ${\mathbb{R}}^{n-1}$,
centered at the origin and having side-length $\ell(Q)\geq 1$ such that
\begin{equation}\label{eq:ADr}
{\rm supp}\,a\subseteq Q,\quad\|a\|_{L^p({\mathbb{R}}^{n-1})}\leq |Q|^{1/p-1}
\,\,\mbox{and }\,\,\int_{{\mathbb{R}}^{n-1}}a(x')\,dx'=0.
\end{equation}
Then, following \cite{GC}, we define the Beurling-Hardy space as
\begin{multline}\label{eq:f-H1-atoms.BEUR}
{\rm HA}^p({\mathbb{R}}^{n-1}):=\Big\{f\in L^1({\mathbb{R}}^{n-1}):\,
f=\sum_{j\in{\mathbb{N}}}\lambda_j\,a_j\ \ \mbox{a.e.~in }\mathbb{R}^{n-1},\mbox{ for some}
\\[0pt]
\mbox{central $(1,p)$-atoms $\{a_j\}_{j\in{\mathbb{N}}}$ and
$\{\lambda_j\}_{j\in{\mathbb{N}}}\in\ell^1$}\Big\},
\end{multline}
and for each $f\in {\rm HA}^p(\mathbb{R}^{n-1})$ set
$\|f\|_{{\rm HA}^p(\mathbb{R}^{n-1})}:=\inf\sum_{j\in{\mathbb{N}}}|\lambda_j|$ with the
infimum taken over representations of $f=\sum_{j\in{\mathbb{N}}}\lambda_j\,a_j$ as in
\eqref{eq:f-H1-atoms.BEUR}. Various alternative characterizations of ${\rm HA}^p(\mathbb{R}^{n-1})$
may be found in \cite[Theorem~3.1, p.\,505]{GC}. Here we only wish to note that, as is apparent
from definitions,
\begin{equation}\label{eq:RTTT}
{\rm HA}^p(\mathbb{R}^{n-1})\hookrightarrow H^1(\mathbb{R}^{n-1}).
\end{equation}

Given a system $L$ as in \eqref{L-def}-\eqref{L-ell.X} and having fixed some $\kappa>0$
and $p\in(1,\infty)$, the $({\rm HA}^p,{\rm A}^p)$-Dirichlet boundary value problem
for $L$ in $\mathbb{R}^{n}_{+}$ is then formulated as
\begin{equation}\label{Dir-BVP-BEUR}
\left\{
\begin{array}{l}
u\in{\mathscr{C}}^\infty({\mathbb{R}}^n_{+}),
\\[4pt]
Lu=0\,\,\mbox{ in }\,\,\mathbb{R}^{n}_{+},
\\[4pt]
\mathcal{N}_\kappa u\in {\rm A}^p({\mathbb{R}}^{n-1}),
\\[6pt]
u\bigl|_{\partial\mathbb{R}^{n}_{+}}^{{}^{\rm n.t.}}=f\in{\rm HA}^p({\mathbb{R}}^{n-1}).
\end{array}
\right.
\end{equation}
We are now ready to present the proof of Corollary~\ref{tfav.tRR.BEUR} dealing with
the well-posedness of \eqref{Dir-BVP-BEUR} for each $p\in(1,\infty)$.

\vskip 0.08in
\begin{proof}[Proof of Corollary~\ref{tfav.tRR.BEUR}]
Fix $p\in(1,\infty)$. Granted Theorem~\ref{Them-General}, we are left with verifying that
the implication in \eqref{eq:BV333} holds if ${\mathbb{X}}={\rm HA}^p({\mathbb{R}}^{n-1})$ and
${\mathbb{Y}}={\rm A}^p({\mathbb{R}}^{n-1})$ (since \eqref{Fi-AN.1} and the first condition
in \eqref{Fi-AN.2} are clear for this choice, thanks to \eqref{eq:rgGH}). With this goal
in mind, pick a Lebesgue measurable function $a:\mathbb{R}^{n-1}\rightarrow\mathbb{C}^M$
whose scalar components are as in \eqref{eq:ADr} and define $u$ as in \eqref{eq:Fb}.
Also, with the cube $Q\subset{\mathbb{R}}^{n-1}$ centered at the origin and side-length
$\ell(Q)\geq 1$ as in \eqref{eq:ADr}, let $N_a$ be the smallest nonnegative integer which
is larger than or equal to ${\rm log}_2(n\,\ell(Q))$. Using \eqref{exist:Nu-Mf}, the $L^p$-boundedness
of the Hardy-Littlewood maximal function, and the normalization of the central $(1,p)$-atom
we obtain
\begin{align}\label{eq:NBEUR}
\sum_{k=0}^{N_a}2^{k(n-1)/p'}\big\|\big({\mathcal{N}}_\kappa u\big){\mathbf{1}}_{C_k}\big\|_{L^p({\mathbb{R}}^{n-1})}
& \leq C\big\|{\mathcal{M}}a\big\|_{L^p({\mathbb{R}}^{n-1})}\sum_{k=0}^{N_a}2^{k(n-1)/p'}
\\[4pt]
&\leq C\|a\|_{L^p(\mathbb{R}^{n-1})}2^{N_a(n-1)/p'}
\nonumber\\[4pt]
&\leq C|Q|^{1/p-1}\big(2^{{\rm log}_2(n\ell(Q))}\big)^{(n-1)/p'}
\nonumber\\[4pt]
&=C<\infty,\nonumber
\end{align}
where the constant $C$ is independent of the central $(1,p)$-atom $a$.
Next, fix an arbitrary integer $k\geq N_a+1$ along with some point $x'\in C_k$.
This choice entails $|x'|>2^{k-1}\geq 2^{N_a}\geq n\ell(Q)$ which, in turn, forces
$x'\in{\mathbb{R}}^{n-1}\setminus 2\sqrt{n}Q$. Granted this, the same type of estimates as in
\eqref{rrff4rf}-\eqref{rrff4rf.2} (this time with $x'_Q=0'$) lead to the conclusion that
there exists a constant $C\in(0,\infty)$ depending only on $n,L,\kappa$ with the property that
\begin{equation}\label{rrff4rf.2BEU}
\big({\mathcal{N}}_\kappa u\big)(x')\leq\frac{C\ell(Q)}{|x'|^n}\,\,\mbox{ whenever $x'\in C_k$ with
$k\geq N_a+1$}.
\end{equation}
Having established this we may then estimate
\begin{align}\label{eq:NBEUR.45t}
&\sum_{k=N_a+1}^\infty 2^{k(n-1)/p'}
\big\|\big({\mathcal{N}}_\kappa u\big){\mathbf{1}}_{C_k}\big\|_{L^p({\mathbb{R}}^{n-1})}
\\[4pt]
&\hskip2cm\leq C\ell(Q)\sum_{k=N_a+1}^\infty 2^{k(n-1)/p'}\Big(\int_{C_k}|x'|^{-np}\,dx'\Big)^{1/p}
\nonumber\\[4pt]
&\hskip2cm\leq C\ell(Q)\sum_{k=N_a+1}^\infty 2^{k(n-1)/p'}2^{-n(k-1)}2^{k(n-1)/p}
\nonumber\\[4pt]
&\hskip2cm\leq C\ell(Q)\sum_{k=N_a+1}^\infty 2^{-k}
\nonumber\\[4pt]
&\hskip2cm\leq C\ell(Q)2^{-N_a}=C<\infty,\nonumber
\end{align}
for some constant $C$ independent of $a$. In concert, \eqref{eq:NBEUR} and
\eqref{eq:NBEUR.45t} prove that there exists some constant $C\in(0,\infty)$
such that whenever $u$ is as in \eqref{eq:Fb} for some central $(1,p)$-atom $a$ then
\begin{equation}\label{eq:NBhf.BEU}
\big\|{\mathcal{N}}_\kappa u\big\|_{{\rm A}^p({\mathbb{R}}^{n-1})}\leq C.
\end{equation}

Going further, we shall make use of \eqref{eq:NBhf.BEU} in order to show that, if
for an arbitrary function $f\in{\rm HA}^p({\mathbb{R}}^{n-1})$ we set $u(x',t):=(P^L_t\ast f)(x')$
for all $(x',t)\in{\mathbb{R}}^n_{+}$, then
\begin{equation}\label{JbvvBEU}
\|{\mathcal{N}}_\kappa u\|_{{\rm A}^p({\mathbb{R}}^{n-1})}
\leq C\|f\|_{{\rm HA}^p({\mathbb{R}}^{n-1})},
\end{equation}
for some finite constant $C>0$ independent of $f$. Specifically, \eqref{JbvvBEU} is justified
with the help of \eqref{eq:NBhf.BEU}, \eqref{eq:RTTT}, and \eqref{eq:rgGH} by reasoning almost
verbatim as in \eqref{Jbvv75r.A}-\eqref{Jbvv75r.D}. This proves the implication in \eqref{eq:BV333}
in the current context, and shows that \eqref{Dir-BVP-BEUR} is well-posed.
\end{proof}

The proof of Theorem~\ref{corol:Xw-Dir} requires some prerequisites, and
we begin by discussing rearrangement invariant spaces. To set the stage, let $\mu_f$
denote the distribution function of a given $f\in\mathbb{M}$, i.e.,
\begin{equation}
\mu_f(\lambda):=\big|\{x'\in\mathbb{R}^{n-1}:\,|f(x')|>\lambda\}\big|,
\qquad\forall\,\lambda\geq 0.
\end{equation}
Call two functions $f,g\in\mathbb{M}$ equimeasurable provided $\mu_f=\mu_g$.
A {\tt rearrangement} {\tt invariant} {\tt space} (or, r.i.~space for short) is a
K\"othe function space $\mathbb{X}$ with the property that equimeasurable functions
have the same function norm in ${\mathbb{X}}$ (i.e., if $\|f\|_\mathbb{X}=\|g\|_\mathbb{X}$
for all $f,\,g\in\mathbb{X}$ such that $\mu_f=\mu_g$). In particular, if $\mathbb{X}$
is an r.i.~space, one can check that its K\"othe dual space $\mathbb{X}'$ is also rearrangement invariant.

Given $f\in\mathbb{M}$, the {\tt decreasing} {\tt rearrangement} of $f$ with respect
to the Lebes\-gue measure in ${\mathbb{R}}^{n-1}$ is the function $f^\ast$, with domain $[0,\infty)$, defined by
\begin{equation}\label{Ygav-iyg}
f^\ast(t):=\inf\,\{\lambda\geq 0:\,\mu_f(\lambda)\leq t\},\qquad 0\leq t<\infty.
\end{equation}
Relative to the original function $f$, its decreasing rearrangement satisfies, for each $\lambda\geq 0$,
\begin{equation}\label{trra-1}
\big|\{x'\in{\mathbb{R}}^{n-1}:\,|f(x')|>\lambda\}\big|
=\big|\{t\in[0,\infty):\,f^\ast(t)>\lambda\}\big|.
\end{equation}
Applying the Luxemburg representation theorem yields the following: given
an r.i.~space $\mathbb{X}$, there exists a unique r.i.~space $\overline{\mathbb{X}}$ on
$[0,\infty)$ such that for each $f\in{\mathbb{M}}$ one has
$f\in\mathbb{X}$ if and only if $f^\ast\in\overline{\mathbb{X}}$ and, in this case,
$\|f\|_{\mathbb{X}}=\|f^\ast\|_{\overline{\mathbb{X}}}$.
Furthermore, $(\overline{\mathbb{X}})'=\overline{\mathbb{X}'}$, and so
$\|f\|_{\mathbb{X}'}=\|f^\ast\|_{\overline{\mathbb{X}}'}$ for every $f\in\mathbb{M}$.

Using this representation we can now introduce the Boyd indices of an r.i.~space
$\mathbb{X}$.  Given $f\in\overline{\mathbb{X}}$, consider the dilation operator $D_t$,
$0<t<\infty$, by setting $D_t f(s):=f(s/t)$ for each $s\geq 0$. Writing
\begin{equation}\label{trra-2}
h_\mathbb{X}(t):=\sup\big\{\|D_{t}f\|_{\overline{\mathbb{X}}}:\,f\in\overline{\mathbb{X}}
\,\,\mbox{ with }\,\,\|f\|_{\overline{\mathbb{X}}}\leq 1\big\},\qquad t\in(0,\infty),
\end{equation}
the lower and upper {\tt Boyd} {\tt indices} may, respectively, be defined as
\begin{equation}\label{trra-3}
p_\mathbb{X}:=\lim_{t\rightarrow\infty}\frac{\log t}{\log h_\mathbb{X}(t)}
=\sup_{1<t<\infty}\frac{\log t}{\log h_\mathbb{X}(t)},
\end{equation}
and
\begin{equation}
\label{eq:dADEDEA}
q_\mathbb{X}:=\lim_{t\rightarrow 0^+}\frac{\log t}{\log h_\mathbb{X}(t)}
=\inf_{0<t<1}\frac{\log t}{\log h_\mathbb{X}(t)}.
\end{equation}
By design, $1\leq p_\mathbb{X}\leq q_\mathbb{X}\leq\infty$.
The Boyd indices for $\mathbb{X}$ are related to those for $\mathbb{X}'$ via
\begin{equation}\label{Jbab-in}
p_{\mathbb{X}'}=(q_\mathbb{X})'\,\,\,\mbox{ and }\,\,\,q_{\mathbb{X}'}=(p_\mathbb{X})'.
\end{equation}

\begin{remark}\label{trra-4}
Some authors (including \cite{bennett-sharpley88}) define the Boyd indices as the
reciprocals of $p_\mathbb{X}$ and $q_\mathbb{X}$ defined above. We have chosen the present
definition since it yields $p_\mathbb{X}=q_\mathbb{X}=p$ if $\mathbb{X}=L^p({\mathbb{R}}^{n-1})$.
\end{remark}

The importance of Boyd indices stems from the fact that they play a significant role in interpolation
(see, e.g., \cite[Chapter~3]{bennett-sharpley88}). For example, the classical result of Lorentz-Shimogaki
states that the Hardy-Littlewood maximal operator ${\mathcal{M}}$ is bounded on an r.i.~space $\mathbb{X}$
if and only if $p_\mathbb{X}>1$. Additionally, Boyd's theorem asserts that the Hilbert transform is
bounded on an r.i.~space $\mathbb{X}$ on ${\mathbb{R}}$ if and only if $1<p_\mathbb{X}\leq q_\mathbb{X}<\infty$.
See \cite[Chapter~3]{bennett-sharpley88} for the precise statements and complete references.

Given an r.i.~space $\mathbb{X}$ on $\mathbb{R}^{n-1}$, we wish to introduce a weighted
version $\mathbb{X}(w)$ of $\mathbb{X}$ via an analogous definition in which the underlying
measure in $\mathbb{R}^{n-1}$ now is $d\mu(x'):=w(x')\,dx'$. These spaces appeared
in~\cite{curbera-garcia-cuerva-martell-perez06} as an abstract generalization of a variety
of weighted function spaces. Specifically, fix a weight $w\in A_\infty({\mathbb{R}}^{n-1})$
(in particular, $0<w<\infty$ a.e.). Given $f\in{\mathbb{M}}$, let $w_f$ denote the distribution
function of $f$ with respect to the measure $w(x')\,dx'$:
\begin{equation}\label{eq:aaa.1}
w_f(\lambda):=w\big(\{x'\in\mathbb{R}^{n-1}:\,|f(x')|>\lambda\}\big),\qquad\lambda\geq 0.
\end{equation}
We also let $f_w^\ast$ denote the decreasing rearrangement of $f$ with respect to the measure $w(x')\,dx'$,
i.e.,
\begin{equation}\label{eq:aaa.2}
f_w^\ast(t):=\inf\big\{\lambda\geq 0:\,w_f(\lambda)\leq t\big\},\qquad 0\leq t<\infty.
\end{equation}
Granted these, define the weighted space $\mathbb{X}(w)$ by
\begin{equation}\label{eq:aaa.3}
\mathbb{X}(w):=\big\{f\in\mathbb{M}:\,\|f^\ast_w\|_{\overline{\mathbb{X}}}<\infty\big\}.
\end{equation}
This may be viewed as a K\"othe function space, but
with underlying measure $w(x')\,dx'$, and with the function norm
\begin{equation}\label{eq:aaa.4}
\|f\|_{\mathbb{X}(w)}:=\|f_w^\ast\|_{\overline{\mathbb{X}}}.
\end{equation}
Note that if $\mathbb{X}$ is the Lebesgue space $L^p({\mathbb{R}}^{n-1})$, $p\in(1,\infty)$,
it follows that $\mathbb{X}(w)=L^p({\mathbb{R}}^{n-1},\,w)$, the Lebesgue space of $p$-th power integrable
functions in the measure space $\big(\mathbb{R}^{n-1},\,w(x')\,dx'\big)$.

\medskip

The boundedness of the Hardy-Littlewood maximal operator on these weighted r.i.~spaces
was considered in \cite{curbera-garcia-cuerva-martell-perez06}, \cite{CMP}, from which
we quote the following result.

\begin{lemma}[{\cite{curbera-garcia-cuerva-martell-perez06}}]\label{lemma:M:Xw}
Let $\mathbb{X}$ be an r.i.~space whose lower Boyd index satisfies
$p_\mathbb{X}>1$. Then for every $w\in A_{p_\mathbb{X}}({\mathbb{R}}^{n-1})$,
the Hardy-Littlewood maximal operator $\mathcal{M}$ is bounded on $\mathbb{X}(w)$.
\end{lemma}

\medskip

At the heart of the proof of Theorem~\ref{corol:Xw-Dir} there is an analog of \eqref{Lpw-hyp}
valid in the context of weighted rearrangement invariant function spaces. We are going to
derive this in Lemma~\ref{lemma:Xw-decay} below, by relying on the following
Rubio de Francia extrapolation for r.i.~spaces obtained in \cite{CMP}:

\begin{theorem}[\cite{CMP}]\label{thm:CMP-extrapol-Xw}
Let $\mathcal{F}$ be a given family of pairs $(f,g)$ of non-negative, measurable functions
that are not identically zero. Suppose that for some fixed exponent $p_0\in[1,\infty)$
and every weight $w_0\in A_{p_0}({\mathbb{R}}^{n-1})$, one has
\begin{equation}\label{p0-Ap-X}
\int_{\mathbb{R}^{n-1}}f(x')^{p_0}\,w_0(x')\,dx'
\leq C_{w_0}\int_{\mathbb{R}^{n-1}}g(x')^{p_0}\,w_0(x')\,dx',
\qquad\forall\,(f,g)\in\mathcal{F}.
\end{equation}

Then if $\mathbb{X}$ is an r.i.~space such that
$1<p_\mathbb{X}\leq q_\mathbb{X}<\infty$, it follows that for
each weight $w\in A_{p_\mathbb{X}}({\mathbb{R}}^{n-1})$ there holds
\begin{equation}\label{p-Ap-X}
\|f\|_{\mathbb{X}(w)}\leq C_w\|g\|_{\mathbb{X}(w)},
\qquad\forall\,(f,g)\in\mathcal{F}.
\end{equation}
\end{theorem}

\begin{remark}\label{trra-5}
As discussed in \cite{CMP}, inequalities of the form \eqref{p0-Ap-X} or \eqref{p-Ap-X}
(both in hypotheses and in the conclusion) are assumed to hold for any $(f,g)\in\mathcal{F}$
for which the left-hand side is finite.
\end{remark}

\begin{lemma}\label{lemma:Xw-decay}
Let $\mathbb{X}$ be an r.i.~space with the property that its lower and upper Boyd indices satisfy
$1<p_\mathbb{X}\leq q_\mathbb{X}<\infty$. For every $w\in A_{p_\mathbb{X}}({\mathbb{R}}^{n-1})$,
there exists $C=C(n,\mathbb{X},w)\in(0,\infty)$ such that for each $h\in\mathbb{X}(w)$ there holds
\begin{equation}\label{eqn:f-M(1B):Xw}
\int_{\mathbb{R}^{n-1}} |h(x')|\,\mathcal{M}^{(2)}\big({\bf 1}_{B_{n-1}(0',1)}\big)(x')\,dx'
\leq C\,\big\|{\bf 1}_{B_{n-1}(0',1)}\big\|_{\mathbb{X}(w)}^{-1}\,\|h\|_{\mathbb{X}(w)}.
\end{equation}
In particular, from \eqref{eqn:f-M(1B):Xw} and Lemma~\ref{lemma:M-ball}
one has the continuous inclusion
\begin{equation}\label{eqn:f-M(1B):Xw.45}
\mathbb{X}(w)\hookrightarrow L^1\Big({\mathbb{R}}^{n-1}\,,\,\frac{1+\log_{+}|x'|}{1+|x'|^{n-1}}\,dx'\Big).
\end{equation}
\end{lemma}

\begin{proof}
We obtain this result via extrapolation using Theorem~\ref{thm:CMP-extrapol-Xw}.
Fix a sufficiently large integer $N$ and, for every $h\in\mathbb{M}$,
set $h_N:=h\,{\bf 1}_{\{x'\in B_{n-1}(0',N):\,|h(x')|\leq N\}}$. In particular,
\begin{multline}\label{Dcvb.1}
I_N(h):=\int_{\mathbb{R}^{n-1}}|h_N(x')|\,\mathcal{M}^{(2)}\big({\bf 1}_{B_{n-1}(0',1)}\big)(x')\,dx'
\\
\leq N\,|B_{n-1}(0',N)|=C_N<\infty.
\end{multline}
We now consider the family of pairs:
\begin{equation}\label{Dcvb.2}
\mathcal{F}_N:=\big\{(F_1,F_2)=\big(I_N(h)\,{\bf 1}_{B_{n-1}(0',1)}\,,\,|h_N|\big):\,h\in\mathbb{M}\big\}.
\end{equation}
Given $p\in(1,\infty)$ and $w\in A_p({\mathbb{R}}^{n-1})$, there exists
$C=C(n,p,w)\in(0,\infty)$ such that for every $N\geq 1$ we may write
(as in \eqref{Lpw-hyp} with $h_N$ replacing $h$)
\begin{align}\label{Dcvb.3}
I_N(h)\leq C\,\|h_N\|_{L^p({\mathbb{R}}^{n-1},\,w)}\,w\big(B_{n-1}(0',1)\big)^{-\frac1p}.
\end{align}
Thus, for every $(F_1,F_2)\in\mathcal{F}_N$ we have
\begin{align}\label{Dcvb.4}
\int_{\mathbb{R}^{n-1}}F_1(x')^p\,w(x')\,dx'
&=I_N(h)^p\,w\big(B_{n-1}(0',1)\big)
\leq C\,\|h_N\|_{L^p({\mathbb{R}}^{n-1},\,w)}^p
\\[4pt]
&= C\,\int_{\mathbb{R}^{n-1}}F_2(x')^p\,w(x')\,dx',\nonumber
\end{align}
where $C\in(0,\infty)$ is independent of $N$. Notice that the left-hand side of the previous
estimate is finite thanks to \eqref{Dcvb.1}. Granted this, we may invoke Theorem~\ref{thm:CMP-extrapol-Xw}
to conclude that for $\mathbb{X}$ as in the statement and every $w\in A_{p_\mathbb{X}}({\mathbb{R}}^{n-1})$
we have
\begin{multline}\label{eqn:Xw-decay:N}
I_N(h)\,\big\|{\bf 1}_{B_{n-1}(0',1)}\big\|_{\mathbb{X}(w)}
=\|F_1\|_{\mathbb{X}(w)}\leq C\,\|F_2\|_{\mathbb{X}(w)}\\
=C\,\|h_N\|_{\mathbb{X}(w)}\leq C\,\|h\|_{\mathbb{X}(w)},
\end{multline}
with $C\in(0,\infty)$ independent of $N$. Note that this estimate holds for every $h\in\mathbb{M}$
since the left-hand side is always finite by \eqref{Dcvb.1}. Consequently, \eqref{eqn:f-M(1B):Xw}
follows from \eqref{eqn:Xw-decay:N} and Lebesgue's Monotone Convergence Theorem upon letting $N\to\infty$.
\end{proof}

We are finally ready to present the proof of Theorem~\ref{corol:Xw-Dir}.

\vskip 0.08in
\begin{proof}[Proof of Theorem~\ref{corol:Xw-Dir}]
The idea is to invoke Theorem~\ref{Them-General} with $\mathbb{X}=\mathbb{Y}=\mathbb{X}(w)$.
Note that \eqref{eqn:f-M(1B):Xw.45} takes care of the second embedding in \eqref{Fi-AN.1} from which,
as pointed out before, the first embedding in \eqref{Fi-AN.1} also follows. The two conditions in
\eqref{Fi-AN.2} are verified upon noting that, by design, $\mathbb{X}(w)$ is a function lattice,
and by referencing Lemma~\ref{lemma:M:Xw}. As such, Theorem~\ref{Them-General} applies and
yields existence and uniqueness for the Dirichlet problem \eqref{Dir-BVP-Xw} in the desired manner.
To complete the proof of Theorem~\ref{corol:Xw-Dir} there remains to observe that the bound
in \eqref{Dir-BVP-Xw:bounds} is a direct consequence of \eqref{eJHBb} and Lemma~\ref{lemma:M:Xw}.
\end{proof}

\section{Return to the Poisson Kernel}
\setcounter{equation}{0}
\label{sect:Poisson}

One aspect left open by Theorem~\ref{ya-T4-fav} is the uniqueness of the
Agmon-Douglis-Nirenberg Poisson kernel in the conceivably larger class of such kernels
outlined by Definition~\ref{defi:Poisson}. The goal here is to address this issue
and also establish the semigroup property for this unique Poisson kernel.

\begin{theorem}\label{taf87h6g}
Let $L$ be a second-order elliptic system with complex coefficients as in
\eqref{L-def}-\eqref{L-ell.X}. Then there exists a unique Poisson kernel
$P^L$ for $L$ in $\mathbb{R}^{n}_{+}$ in the sense of Definition~\ref{defi:Poisson}.
Moreover, this Poisson kernel satisfies the semigroup property
\begin{equation}\label{eq:re4fd}
P^L_{t_1}\ast P^L_{t_2}=P^L_{t_1+t_2}\,\,\,\mbox{ for every }\,\,t_1,t_2>0.
\end{equation}
\end{theorem}

\vskip 0.08in

The convolution between the two matrix-valued functions in \eqref{eq:re4fd} is understood
in a natural fashion, taking into account the algebraic multiplication of matrices.
On this note, one significant consequence of identity \eqref{eq:re4fd} is the
commutativity of the convolution product for the matrix-valued functions $P^L_{t_1}$
and $P^L_{t_2}$, i.e., $P^L_{t_1}\ast P^L_{t_2}=P^L_{t_2}\ast P^L_{t_1}$ for each $t_1,t_2>0$.
We shall further elaborate on this topic after discussing the proof of Theorem~\ref{taf87h6g}.
Here we only wish to remark that, in the classical case $L=\Delta$, the semigroup property
\eqref{eq:re4fd} is proved in \cite[(vi) p.\,62]{St70} making use of the explicit formula for the
Fourier transform of $P^\Delta$. Instead, in the case of an arbitrary system $L$ as in
\eqref{L-def}-\eqref{L-ell.X}, our strategy is to rely on the well-posedness of the
$L^p$-Dirichlet problem \eqref{Dir-BVP-Lpw-intro}.

\vskip 0.08in
\begin{proof}[Proof of Theorem~\ref{taf87h6g}]
Let $P^L$ stand for the Agmon-Douglis-Nirenberg Poisson kernel for $L$ from Theorem~\ref{ya-T4-fav}
and assume that $Q^L$ is another Poisson kernel for $L$ in $\mathbb{R}^{n}_{+}$ in
the sense of Definition~\ref{defi:Poisson}. Fix an arbitrary vector-valued function
$f\in{\mathscr{C}}^\infty_0({\mathbb{R}}^{n-1})$ and define for each $(x',t)\in{\mathbb{R}}^n_{+}$
\begin{equation}\label{eBJkj}
u_1(x',t):=(P^L_t\ast f)(x')\,\,\,\mbox{ and }\,\,\, u_2(x',t):=(Q^L_t\ast f)(x').
\end{equation}
Then Theorem~\ref{thm:existence} and \eqref{jnabn88} in Remark~\ref{hyFF.a} imply that,
for any given $p\in(1,\infty)$, both $u_1$ and $u_2$ solve the $L^p$-Dirichlet boundary
value problem in ${\mathbb{R}}^n_{+}$ as formulated in \eqref{Dir-BVP-Lpw-intro}.
The well-posedness of this boundary value problem (cf. the discussion in
Example~1 in \S\ref{S-1}) then forces $u_1=u_2$ in ${\mathbb{R}}^n_{+}$ which further
translates into $(P^L_t\ast f)(x')=(Q^L_t\ast f)(x')$ for all $(x',t)\in{\mathbb{R}}^n_{+}$
and all $f\in{\mathscr{C}}^\infty_0({\mathbb{R}}^{n-1})$. In turn, this yields
$P^L=Q^L$ a.e.~in ${\mathbb{R}}^{n-1}$, hence everywhere by the continuity of $P^L$ and $Q^L$
(see part ${\it (ii)}$ in Remark~\ref{Ryf-uyf}). This finishes the proof of the first
claim in the statement of theorem.

Consider now the semigroup property \eqref{eq:re4fd}. To get started, fix $t_2>0$ and
pick an arbitrary vector-valued function $f\in{\mathscr{C}}^\infty_0({\mathbb{R}}^{n-1})$.
Let $P^L$ be the unique Poisson kernel for $L$ and, for each $(x',t)\in{\mathbb{R}}^n_{+}$,
define this time
\begin{equation}\label{eBJkj.2ww}
u_1(x',t):=\big(P^L_t\ast\big(P^L_{t_2}\ast f)\big)(x')
\,\,\,\mbox{ and }\,\,\,u_2(x',t):=(P^L_{t+t_2}\ast f)(x').
\end{equation}
Fix some $p\in(1,\infty)$ and observe that $P^L_{t_2}\ast f\in L^p({\mathbb{R}}^{n-1})$ by
\eqref{exTGFVC}. Finally, consider the $L^p$-Dirichlet boundary value problem
\begin{equation}\label{Dir-BVP-Lpw-inRF}
\left\{
\begin{array}{l}
u\in{\mathscr{C}}^\infty(\mathbb{R}^{n}_{+}),
\\[4pt]
Lu=0\,\,\mbox{ in }\,\,\mathbb{R}^{n}_{+},
\\[6pt]
\mathcal{N}u\in L^p({\mathbb{R}}^{n-1}),
\\[4pt]
u\big|_{\partial\mathbb{R}^{n}_{+}}^{{}^{\rm n.t.}}=P^L_{t_2}\ast f\in L^p({\mathbb{R}}^{n-1}).
\end{array}
\right.
\end{equation}
From the discussion in Example~1 in \S\ref{S-1} we know that
$u_1$ is the unique solution of \eqref{Dir-BVP-Lpw-inRF} and we claim that $u_2$
also solves \eqref{Dir-BVP-Lpw-inRF}. Assuming this momentarily, it follows that $u_1=u_2$
in ${\mathbb{R}}^n_{+}$, hence $\big((P^L_t\ast P^L_{t_2})\ast f\big)(x')=(P^L_{t+t_2}\ast f)(x')$
for all $x'\in{\mathbb{R}}^{n-1}$, all $t\in(0,\infty)$, and each
$f\in{\mathscr{C}}^\infty_0({\mathbb{R}}^{n-1})$. Much as before, this readily implies
$P^L_t\ast P^L_{t_2}=P^L_{t+t_2}$ in ${\mathbb{R}}^{n-1}$ for each $t\in(0,\infty)$,
and \eqref{eq:re4fd} follows from this by taking $t:=t_1$.

To finish the proof, there remains to check that, as claimed, $u_2$ from \eqref{eBJkj.2ww}
is a solution of \eqref{Dir-BVP-Lpw-inRF}. To this end, introduce
$v(x',t):=(P^L_{t}\ast f)(x')$ for $(x',t)\in{\mathbb{R}}^n_{+}$ and note that, by
Theorem~\ref{thm:existence}, $v$ satisfies
\begin{equation}\label{exist:u2FV}
v\in\mathscr{C}^\infty(\mathbb{R}^n_{+}),\quad
Lv=0\,\,\mbox{ in }\,\,\mathbb{R}^{n}_{+},\quad
\mathcal{N} v(x')\leq C\,\mathcal{M} f(x')
\,\,\mbox{ for all }\,x'\in\mathbb{R}^{n-1}.
\end{equation}
Since, by design, $u_2(x',t)=v(x',t+t_2)$ for all $(x',t)\in{\mathbb{R}}^n_{+}$,
we easily deduce from \eqref{exist:u2FV} that
\begin{equation}\label{exist:u2FV3}
u_2\in\mathscr{C}^\infty(\overline{\mathbb{R}^n_{+}}),\ \ 
Lu_2=0\,\mbox{ in }\,\mathbb{R}^{n}_{+},\ \ 
\mathcal{N}u_2(x')\leq C\,\mathcal{M} f(x')
\,\mbox{ for all }\,x'\in\mathbb{R}^{n-1}.
\end{equation}
Hence, $\mathcal{N}u_2\in L^p({\mathbb{R}}^{n-1})$ and since for each $x'\in{\mathbb{R}}^{n-1}$ we have
\begin{equation}\label{eq:6yH}
\big(u_2\big|_{\partial\mathbb{R}^{n}_{+}}^{{}^{\rm n.t.}}\big)(x')=u_2(x',0)
=v(x',t_2)=(P_{t_2}\ast f)(x'),
\end{equation}
it follows that $u_2$ solves \eqref{Dir-BVP-Lpw-inRF}. This finishes the proof of
Theorem~\ref{taf87h6g}.
\end{proof}

Theorem~\ref{taf87h6g} has several consequences of independent
interest, and here we wish to single out the following result.

\begin{corollary}\label{VCXga}
Let $L$ be a homogeneous second-order elliptic system with {\rm (}complex{\rm )}
constant coefficients, and let $P^L$ denote its unique Poisson kernel in $\mathbb{R}^{n}_{+}$
{\rm (}cf. Theorem~\ref{taf87h6g}{\rm )}.
Also, let $\mathbb{X}$ be a K\"othe function space with the property that
${\mathcal{M}}$ is bounded on $\mathbb{X}$. Then the family $\{T_t\}_{t>0}$, where for
each $t\in(0,\infty)$,
\begin{equation}\label{eq:Taghb8}
T_t:\mathbb{X}\to\mathbb{X},\quad T_tf(x'):=(P^L_t\ast f)(x')\,\,\mbox{ for every }\,
f\in\mathbb{X},\,\,x'\in{\mathbb{R}}^{n-1},
\end{equation}
is a semigroup of bounded linear operators on ${\mathbb{X}}$ which satisfies
\begin{equation}\label{eq:Taghb8.77}
\sup_{t>0}\big\|T_t\big\|_{{\mathcal{L}}(\mathbb{X})}<\infty,
\end{equation}
where ${\mathcal{L}}(\mathbb{X})$ is the Banach space of linear and bounded
operators on $\mathbb{X}$.

Furthermore, under the additional assumption that the function norm
in ${\mathbb{X}}$ is absolutely continuous,
meaning that for any given $f\in{\mathbb{X}}$ there holds
\begin{equation}\label{AB-CTxxx}
\left.
\begin{array}{r}
(A_j)_{j\in{\mathbb{N}}}\,\,\mbox{ measurable subsets of }\,\,{\mathbb{R}}^{n-1}
\\[4pt]
\mbox{with }\,\,{\mathbf{1}}_{A_j}\to 0\,\,\mbox{ a.e. in ${\mathbb{R}}^{n-1}$ as $j\to\infty$}
\end{array}
\right\}\Longrightarrow\,\lim_{j\to\infty}\big\|\,|f|\cdot{\mathbf{1}}_{A_j}\big\|_{{\mathbb{X}}}=0,
\end{equation}
it follows that $\{T_t\}_{t>0}$ is a strongly continuous semigroup in the sense that
\begin{equation}\label{eq:Rfav}
\lim_{t\to 0^{+}}T_tf=f\,\,\,\mbox{ in }\,\,{\mathbb{X}},\,\,\,\mbox{ for each }\,\,f\in{\mathbb{X}}.
\end{equation}
\end{corollary}

\begin{proof}
From \eqref{eq:jlk}, \eqref{Lab-2}, the assumptions on ${\mathbb{X}}$, and
\eqref{eq:Eda}-\eqref{exTGFVC} in Lemma~\ref{lennii} it follows that for each $t\in(0,\infty)$
the operator $T_t:\mathbb{X}\to\mathbb{X}$ is well-defined, linear, and bounded.
Moreover, there exists a finite constant $C>0$ with the property that for each
$x'\in\mathbb{R}^{n-1}$,
\begin{equation}\label{exTsgs}
\sup_{t>0}\big|(T_t f)(x')\big|\leq C\,\mathcal{M} f(x'),
\qquad\forall\,f\in\mathbb{X}\subseteq L^1\Big({\mathbb{R}}^{n-1}\,,\,\frac{1}{1+|x'|^n}\,dx'\Big).
\end{equation}
Bearing in mind the assumptions on ${\mathbb{X}}$,
this readily gives \eqref{eq:Taghb8.77}.
The semigroup property for the family $\{T_t\}_{t>0}$
is then a consequence of \eqref{eq:re4fd}, \eqref{eq:jlk}, and \eqref{eq:Eda}.

Concerning the strong continuity property of the semigroup $\{T_t\}_{t>0}$,
fix an arbitrary $f\in{\mathbb{X}}$ and note that, as a consequence of the last
condition in \eqref{exist:u2}, we have $T_tf\to f$ a.e. in ${\mathbb{R}}^{n-1}$ as $t\to 0^{+}$.
In addition, $|T_tf|\leq C{\mathcal{M}}f\in{\mathbb{X}}$ by \eqref{exTsgs}
and the assumptions on ${\mathbb{X}}$. From these and Lebesgue's Dominated
Convergence Theorem in ${\mathbb{X}}$ (itself equivalent to the absolute continuity of the
function norm in ${\mathbb{X}}$; cf. \cite[Proposition~3.6, p.\,16]{bennett-sharpley88}),
it follows that \eqref{eq:Rfav} holds.
\end{proof}

For a thorough discussion pertaining to the absolute continuity of the function norm in
a K\"othe ${\mathbb{X}}$ the interested reader is referred to
\cite[Chapter~1, \S\,3]{bennett-sharpley88}. We conclude by giving a list of
examples of scales of spaces satisfying all hypotheses in Corollary~\ref{VCXga}
(i.e., K\"othe spaces with an absolutely continuous function norm on which the Hardy-Littlewood
maximal operator is bounded):

\begin{enumerate}
\item[$(i)$] Ordinary Lebesgue spaces $L^p({\mathbb{R}}^{n-1})$ with $p\in(1,\infty)$.
\item[$(ii)$] Variable exponent Lebesgue spaces $L^{p(\cdot)}(\mathbb{R}^{n-1})$ on which
the Hardy-Little\-wood maximal operator is bounded (see \cite[Theorem~2.62, p.\,47]{CU-F}
for the absolute continuity of the function norm in this setting).
\item[$(iii)$] Lorentz spaces $L^{p,q}({\mathbb{R}}^{n-1})$ with $1<p,q<\infty$
(which in this range are reflexive, hence have absolutely continuous function norms by
\cite[Chapter~1, \S\,4]{bennett-sharpley88}).
\item[$(iv)$]  Orlicz spaces $L^\Phi(\mathbb{R}^{n-1})$, where $\Phi$ is a Young function.
\end{enumerate}

For technical reasons, the weighted Lebesgue spaces $L^p({\mathbb{R}}^{n-1},\,w(x')\,dx')$,
with $p\in(1,\infty)$ and $w\in A_p({\mathbb{R}}^{n-1})$, do not fall directly under
the scope of Corollary~\ref{VCXga} (since they fail to be K\"othe spaces in the
ordinary sense adopted in this paper, i.e., with respect to the background measure
space $({\mathbb{R}}^{n-1},\,dx')$). Nonetheless, the same type of conclusions as in
Corollary~\ref{VCXga} hold, and this is actually the case for a more general scale of
weighted spaces. Specifically, consider
\begin{equation}\label{eq:Uyeaa}
\begin{array}{c}
\mbox{a rearrangement invariant space $\mathbb{X}$ with lower Boyd index $p_\mathbb{X}>1$}
\\[4pt]
\mbox{and also fix some Muckenhoupt weight $w\in A_{p_\mathbb{X}}({\mathbb{R}}^{n-1})$.}
\end{array}
\end{equation}
Finally, recall the weighted version $\mathbb{X}(w)$ of $\mathbb{X}$ defined in \eqref{eq:aaa.3},
and consider the condition that for every $f\in{\mathbb{X}}(w)$ one has
\begin{equation}\label{AB-CTxRfhg}
\left.
\begin{array}{r}
(A_j)_{j\in{\mathbb{N}}}\,\,\mbox{ measurable subsets of }\,\,{\mathbb{R}}^{n-1}
\\[4pt]
\mbox{with }\,\,{\mathbf{1}}_{A_j}\to 0\,\,\mbox{ a.e. in ${\mathbb{R}}^{n-1}$ as $j\to\infty$}
\end{array}
\right\}\Longrightarrow\,\lim_{j\to\infty}\big\|\,|f|\cdot{\mathbf{1}}_{A_j}\big\|_{{\mathbb{X}}(w)}=0.
\end{equation}
Then a cursory inspection of the proof of \cite[Proposition~3.6, p.\,16]{bennett-sharpley88}
reveals that \eqref{AB-CTxRfhg} implies Lebesgue's Dominated Convergence Theorem in ${\mathbb{X}}(w)$.
Based on this, Lemma~\ref{lennii}, and Lemma~\ref{lemma:M:Xw}, the same type of
reasoning as in the proof of Corollary~\ref{VCXga} works and yields the following
result.

\begin{corollary}\label{VCXga.222}
Assuming \eqref{eq:Uyeaa} and that the system $L$ is as
in \eqref{L-def}-\eqref{L-ell.X}, the family $\{T_t\}_{t>0}$, where for each $t\in(0,\infty)$,
\begin{equation}\label{eq:Taghb8.BB}
T_t:\mathbb{X}(w)\to\mathbb{X}(w)
\end{equation}
and
\begin{equation}\label{eq:Taghb8.BBaa}
T_tf(x'):=(P^L_t\ast f)(x')\,\,\mbox{ for every }\,
f\in\mathbb{X}(w),\,\,x'\in{\mathbb{R}}^{n-1},
\end{equation}
is a semigroup of bounded linear operators on ${\mathbb{X}}(w)$, satisfying
\begin{equation}\label{eq:Tkbbvv}
\sup_{t>0}\big\|T_t\big\|_{{\mathcal{L}}(\mathbb{X}(w))}<\infty.
\end{equation}
Moreover, under the additional assumption that \eqref{AB-CTxRfhg} holds,
this semigroup is strongly continuous.
\end{corollary}

Of course, Corollary~\ref{VCXga.222} contains as particular cases the scale of weighted
Lebesgue spaces $L^p({\mathbb{R}}^{n-1},\,w)$ with $p\in(1,\infty)$ and $w\in A_p({\mathbb{R}}^{n-1})$,
as well as the scales of weighted Lorentz spaces that are reflexive
(cf. \cite[Corollary~4.4, p.\,23]{bennett-sharpley88}) and weighted Orlicz spaces,
discussed in the last part of \S\ref{S-1}.

\section{A Fatou Type Theorem}
\setcounter{equation}{0}
\label{sect:Fatou}

The goal in this section is to use the tools developed in \S\ref{S-3}
in order to prove the following Fatou type result.

\begin{theorem}\label{tuFatou}
Let $L$ be a system as in \eqref{L-def}-\eqref{L-ell.X} and let $P^L$ be its Poisson kernel
in ${\mathbb{R}}^n_{+}$. Assume that
\begin{align}\label{Tafva.111}
\begin{array}{l}
u\in{\mathscr{C}}^\infty({\mathbb{R}}^n_{+}),\quad
Lu=0\,\mbox{ in }\,{\mathbb{R}}^n_{+},\quad
{\mathcal{N}}u\in L^1\Big({\mathbb{R}}^{n-1}\,,\,\frac{1+\log_{+}|x'|}{1+|x'|^{n-1}}\,dx'\Big),
\\[6pt]
\mbox{and there exists a sequence $\{t_j\}_{j\in{\mathbb{N}}}\subset(0,\infty)$
satisfying $\lim\limits_{j\to\infty}t_j=0$}
\\[6pt]
\mbox{and such that ${\mathcal{M}}u(\cdot,t_j)\in
L^1\Big({\mathbb{R}}^{n-1}\,,\,\frac{1+\log_{+}|x'|}{1+|x'|^{n-1}}\,dx'\Big)$ for every $j\in{\mathbb{N}}$.}
\end{array}
\end{align}
Then
\begin{align}\label{Tafva.2222}
\begin{array}{l}
u\big|^{{}^{\rm n.t.}}_{\partial{\mathbb{R}}^n_{+}}\,\,\mbox{ exists a.e.~in }\,\,
{\mathbb{R}}^{n-1},
\\[10pt]
u\big|^{{}^{\rm n.t.}}_{\partial{\mathbb{R}}^n_{+}}\in
L^1\Big({\mathbb{R}}^{n-1}\,,\,\frac{1+\log_{+}|x'|}{1+|x'|^{n-1}}\,dx'\Big),
\\[12pt]
u(x',t)=\Big(P^L_t\ast\big(u\big|^{{}^{\rm n.t.}}_{\partial{\mathbb{R}}^n_{+}}\big)\Big)(x'),\quad
\forall\,(x',t)\in{\mathbb{R}}^n_{+}.
\end{array}
\end{align}

In particular, the conclusions in \eqref{Tafva.2222} hold whenever
\begin{align}\label{Tafva}
u\in{\mathscr{C}}^\infty({\mathbb{R}}^n_{+}),\ \
Lu=0\,\mbox{ in }\,{\mathbb{R}}^n_{+},\ \ 
{\mathcal{M}}\big({\mathcal{N}}u\big)\in
L^1\Big({\mathbb{R}}^{n-1}\,,\,\frac{1+\log_{+}|x'|}{1+|x'|^{n-1}}\,dx'\Big).
\end{align}
\end{theorem}

Prior to presenting the proof of Theorem~\ref{tuFatou}, we isolate a useful weak compactness
result. To state it, denote by ${\mathscr{C}}_{\rm van}({\mathbb{R}}^{n-1})$
the space of continuous functions in ${\mathbb{R}}^{n-1}$ vanishing at infinity.

\begin{lemma}\label{ydadHBB}
Let $v:{\mathbb{R}}^{n-1}\to(0,\infty)$ be a Lebesgue measurable function. Consider a sequence $\{f_j\}_{j\in{\mathbb{N}}}\subset L^1({\mathbb{R}}^{n-1}\,,v)$
such that $F:=\sup\limits_{j\in{\mathbb{N}}}|f_j|\in L^1({\mathbb{R}}^{n-1}\,,v)$.
Then there exists a subsequence $\big\{f_{j_k}\big\}_{k\in{\mathbb{N}}}$ of $\{f_j\}_{j\in{\mathbb{N}}}$
and a function $f\in L^1({\mathbb{R}}^{n-1}\,,v)$ with the property that
\begin{equation}\label{eq:16t44}
\int_{{\mathbb{R}}^{n-1}}f_{j_k}(x')\varphi(x')v(x')\,dx'\longrightarrow
\int_{{\mathbb{R}}^{n-1}}f(x')\varphi(x')v(x')\,dx'\,\,\mbox{ as }\,\,k\to\infty,
\end{equation}
for every $\varphi\in{\mathscr{C}}_{\rm van}({\mathbb{R}}^{n-1})$.
\end{lemma}

\begin{proof}
Set $\widetilde{f}_j:=f_jv$ for each $j\in{\mathbb{N}}$, and $\widetilde{F}:=Fv$. Then
\begin{equation}\label{eq:utfc1}
|\widetilde{f}_j|\leq\widetilde{F}\in L^1({\mathbb{R}}^{n-1})\,\,\,
\mbox{ for each }\,\,j\in{\mathbb{N}}.
\end{equation}
Let ${\mathscr{M}}$ be the space of finite Borel regular measures in ${\mathbb{R}}^{n-1}$,
viewed as a Banach space when equipped with the norm induced by the total variation.
Then
\begin{equation}\label{eq:utfc2}
L^1({\mathbb{R}}^{n-1})\hookrightarrow
{\mathscr{M}}=\big({\mathscr{C}}_{\rm van}({\mathbb{R}}^{n-1})\big)^\ast.
\end{equation}
From \eqref{eq:utfc1}-\eqref{eq:utfc2} and Alaoglu's theorem it follows that
there exists a subsequence $\big\{\widetilde{f}_{j_k}\big\}_{k\in{\mathbb{N}}}$
and some $\mu\in{\mathscr{M}}$ with the property that
\begin{equation}\label{eq:1533}
\int_{{\mathbb{R}}^{n-1}}\widetilde{f}_{j_k}(x')\varphi(x')\,dx'\longrightarrow
\int_{{\mathbb{R}}^{n-1}}\varphi(x')\,d\mu(x')\,\,\mbox{ as }\,\,k\to\infty,
\end{equation}
for every $\varphi\in{\mathscr{C}}_{\rm van}({\mathbb{R}}^{n-1})$. We claim that
\begin{equation}\label{eq:7gf5f}
\mu\ll{\mathscr{L}}^{n-1}.
\end{equation}
To justify this claim, fix a Lebesgue measurable set $E_0\subset{\mathbb{R}}^{n-1}$
with ${\mathscr{L}}^{n-1}(E_0)=0$. Given the goals we have in mind, there is
no loss of generality in assuming that $E_0$ is bounded. To proceed, pick an arbitrary
$\varepsilon>0$. Since $\widetilde{F}$ is a nonnegative function in $L^1({\mathbb{R}}^{n-1})$,
there exists $\delta>0$ such that
\begin{equation}\label{eq:hre}
\int_{U}\widetilde{F}\,d{\mathscr{L}}^{n-1}<\varepsilon,\ \
\mbox{for each measurable set $U\subset{\mathbb{R}}^{n-1}$ with
${\mathscr{L}}^{n-1}(U)<\delta$.}
\end{equation}
By the outer regularity of ${\mathscr{L}}^{n-1}$, there exists an open and bounded subset $U_0$
of ${\mathbb{R}}^{n-1}$ containing $E_0$ and such that ${\mathscr{L}}^{n-1}(U_0)<\delta$.
For any $\varphi\in{\mathscr{C}}({\mathbb{R}}^{n-1})$ supported in $U_0$ we may then
use \eqref{eq:1533} and \eqref{eq:hre} to estimate
\begin{multline}\label{eq:4d4d}
\Big|\int_{{\mathbb{R}}^{n-1}}\varphi\,d\mu\Big| =\lim_{k\to\infty}
\Big|\int_{{\mathbb{R}}^{n-1}}\widetilde{f}_{j_k}\varphi\,d{\mathscr{L}}^{n-1}\Big|
\\[4pt]
 \leq\|\varphi\|_{L^\infty({\mathbb{R}}^{n-1})}
\int_{U_0}\widetilde{F}\,d{\mathscr{L}}^{n-1}
\leq\|\varphi\|_{L^\infty({\mathbb{R}}^{n-1})}\,\varepsilon.
\end{multline}
In turn, this forces $|\mu|(U_0)\leq\varepsilon$, hence $|\mu|(E_0)\leq\varepsilon$.
Since $\varepsilon>0$ is arbitrary, we conclude that $|\mu|(E_0)=0$ and \eqref{eq:7gf5f}
follows. Next, from \eqref{eq:7gf5f} and the Radon-Nikodym Theorem we conclude that
there exists $\widetilde{f}\in L^1({\mathbb{R}}^{n-1})$ such that
\begin{equation}\label{eq:7gf5f.2}
d\mu=\widetilde{f}\,d{\mathscr{L}}^{n-1}.
\end{equation}
At this stage, \eqref{eq:16t44} follows with $f:=\widetilde{f}/v\in L^1({\mathbb{R}}^{n-1}\,,v)$
based on \eqref{eq:1533} and \eqref{eq:7gf5f.2}.
\end{proof}

We are now ready to tackle the proof of Theorem~\ref{tuFatou}.

\vskip 0.08in
\begin{proof}[Proof of Theorem~\ref{tuFatou}]
By assumption, the function $u$ satisfies
\begin{align}\label{Fi-AN.1uttr}
{\mathcal{N}}u\in L^1({\mathbb{R}}^{n-1}\,,v)\quad\mbox{ where }\,\,
v(x'):=\frac{1+\log_{+}|x'|}{1+|x'|^{n-1}},\quad\forall\,x'\in{\mathbb{R}}^{n-1}.
\end{align}
For each $j\in{\mathbb{N}}$ consider the function $u_j$ defined by $u_j(x',t):=u(x',t+t_j)$ for
each $(x',t)\in{\mathbb{R}}^n_{+}$.  Observe that, for each $j\in{\mathbb{N}}$, the function
$u_j$ belongs to ${\mathscr{C}}^\infty\big(\overline{{\mathbb{R}}^n_{+}}\,\big)$, thus
\begin{equation}\label{eq:443ee}
f_j:=u_j\big|^{{}^{\rm n.t.}}_{\partial{\mathbb{R}}^n_{+}}
=u_j\big|_{\partial{\mathbb{R}}^n_{+}}=u(\cdot,t_j)
\end{equation}
exists and satisfies
\begin{align}\label{cCh-cds}
|f_j|\leq{\mathcal{N}}u_j\leq{\mathcal{N}}u\quad\mbox{ in }\,\,{\mathbb{R}}^{n-1}.
\end{align}
In particular, we conclude from \eqref{Fi-AN.1uttr}-\eqref{cCh-cds} that
\begin{align}\label{cCh-cds.5443}
f_j,\, {\mathcal{N}}u_j\in L^1({\mathbb{R}}^{n-1}\,,v)\,\,\,\mbox{ for each }\,\,j\in{\mathbb{N}}.
\end{align}
Keeping in mind \eqref{Fi-AN.1uttr}-\eqref{cCh-cds.5443}, we may
then invoke Lemma~\ref{ydadHBB} (with $F\le{\mathcal{N}}u$) to conclude that there exists a
subsequence $\big\{f_{j_k}\big\}_{k\in{\mathbb{N}}}$ of $\big\{f_j\big\}_{j\in{\mathbb{N}}}$
and a function $f\in L^1({\mathbb{R}}^{n-1}\,,v)$ with the property that
\begin{equation}\label{eq:16t44.aa}
\int_{{\mathbb{R}}^{n-1}}f_{j_k}(x')\varphi(x')v(x')\,dx'\longrightarrow
\int_{{\mathbb{R}}^{n-1}}f(x')\varphi(x')v(x')\,dx'\,\,\mbox{ as }\,\,k\to\infty,
\end{equation}
for every $\varphi\in{\mathscr{C}}_{\rm van}({\mathbb{R}}^{n-1})$.

To proceed, let us observe that for each $k\in{\mathbb{N}}$ the function $f_{j_k}$
clearly satisfies \eqref{exist:f}, by \eqref{Fi-AN.1uttr} and \eqref{cCh-cds.5443}.
Next for each $k\in{\mathbb{N}}$ define
\begin{align}\label{cCh-cjbg}
U_k(x',t):=\big(P^L_t\ast f_{j_k}\big)(x'),\qquad
\forall\,(x',t)\in{\mathbb{R}}^n_{+}.
\end{align}
Note that, thanks to Theorem~\ref{thm:existence}, this entails
\begin{equation}\label{Tafva.644}
\begin{array}{c}
U_k\in{\mathscr{C}}^\infty({\mathbb{R}}^n_{+}),\qquad
LU_k=0\,\mbox{ in }\,{\mathbb{R}}^n_{+},
\\[8pt]
U_k\big|^{{}^{\rm n.t.}}_{\partial{\mathbb{R}}^n_{+}}
=f_{j_k}\,\,\mbox{ a.e.~in }\,\,{\mathbb{R}}^{n-1},\quad
{\mathcal{N}}U_k\in L^1({\mathbb{R}}^{n-1}\,,v),
\end{array}
\end{equation}
where the last condition is a consequence of \eqref{cCh-cjbg}, \eqref{exist:Nu-Mf},
\eqref{eq:443ee}, and the last line in \eqref{Tafva.111}.
On the other hand, for each $k\in{\mathbb{N}}$, the function $u_{j_k}$ satisfies
the same quartet of conditions as $U_k$ in \eqref{Tafva.644} (where, this time,
the condition ${\mathcal{N}}u_{j_k}\in L^1({\mathbb{R}}^{n-1}\,,v)$ is seen
straight from \eqref{cCh-cds.5443}). As such, Theorem~\ref{thm:uniqueness} applies
to the difference $u_{j_k}-U_k$ and yields $u_{j_k}=U_k$ in ${\mathbb{R}}^n_{+}$
for each $k\in{\mathbb{N}}$. Hence,
\begin{align}\label{cCh-cds-3}
u(x',t+t_{j_k})=u_{j_k}(x',t)=\big(P^L_t\ast f_{j_k}\big)(x'),\qquad
\forall\,(x',t)\in{\mathbb{R}}^n_{+},\quad\forall\,k\in{\mathbb{N}}.
\end{align}
Going further, fix $(x',t)\in{\mathbb{R}}^n_{+}$ and consider the function
\begin{equation}\label{eq:trF}
\varphi_{x',t}(y'):=\frac{1}{v(y')}P^L_t(x'-y')=\frac{1+|x'|^{n-1}}{1+\log_{+}|x'|}
P^L_t(x'-y'),\quad\forall\,y'\in{\mathbb{R}}^{n-1},
\end{equation}
and note that, from part $(a)$ in Definition~\ref{defi:Poisson} and part ({\it ii})
in Remark~\ref{Ryf-uyf}, it follows that $\varphi_{x',t}\in{\mathscr{C}}_{\rm van}({\mathbb{R}}^{n-1})$.
Granted this, by combining \eqref{cCh-cds-3}, \eqref{eq:trF}, \eqref{eq:16t44.aa}, and also bearing
in mind that $u$ is continuous in ${\mathbb{R}}^n_{+}$ and $\lim\limits_{k\to\infty}t_{j_k}=0$,
we conclude that for each $x'\in{\mathbb{R}}^{n-1}$ and each $t\in(0,\infty)$,
\begin{align}\label{cCh-cds-4}
u(x',t) & =\lim_{k\to\infty}u(x',t+t_{j_k})
=\lim_{k\to\infty}\int_{{\mathbb{R}}^{n-1}}P^L_t(x'-y')f_{j_k}(y')\,dy'
\\[4pt]
&=\lim_{k\to\infty}\int_{{\mathbb{R}}^{n-1}}\varphi_{x',t}(y')f_{j_k}(y')v(y')\,dy'
\nonumber\\[4pt]
&=\int_{{\mathbb{R}}^{n-1}}\varphi_{x',t}(y')f(y')v(y')\,dy'
=(P^L_t\ast f)(x').\nonumber
\end{align}
Hence, $u(x',t)=(P^L_t\ast f)(x')$ for each $(x',t)\in{\mathbb{R}}^n_{+}$ for some
$f\in L^1({\mathbb{R}}^{n-1}\,,v)$. Having established this, Lemma~\ref{lennii} (with $P=P^L$)
yields that the non-tangential limit of $u$ on $\partial{\mathbb{R}}^n_{+}$ exists and equals $f$,
proving the first conclusion in \eqref{Tafva.2222}. The second conclusion in \eqref{Tafva.2222}
is immediate from \eqref{cCh-cds-4} and \eqref{eq:aaAa}-\eqref{exTGFVC.2s}, keeping in mind
that $L^1({\mathbb{R}}^{n-1}\,,v)\subseteq
L^1\Big({\mathbb{R}}^{n-1}\,,\,\frac{1}{1+|x'|^n}\,dx'\Big)$.
Finally, the third conclusion in \eqref{Tafva.2222}
is implicit in \eqref{cCh-cds-4}.

There remains to show that \eqref{Tafva} implies \eqref{Tafva.111}
(parenthetically, we note that ${\mathcal{M}}$ acts in a meaningful way on $\mathbb{M}$,
hence on the lower semicontinuous function ${\mathcal{N}}u$).
Indeed, the membership in \eqref{Tafva} implies that
${\mathcal{M}}\big({\mathcal{N}}u)<\infty$ a.e.~in ${\mathbb{R}}^{n-1}$,
which further entails ${\mathcal{N}}u\in L^1_{\rm loc}({\mathbb{R}}^{n-1})$. From this and
Lebesgue's Differentiation Theorem we then deduce that ${\mathcal{N}}u\leq{\mathcal{M}}\big({\mathcal{N}}u\big)$
a.e.~in ${\mathbb{R}}^{n-1}$ which, in light of the last condition in \eqref{Tafva}, ultimately
yields the membership in the first line of \eqref{Tafva.111}. Moreover, the fact that
$|u(\cdot,t)|\leq{\mathcal{N}}u$ in ${\mathbb{R}}^{n-1}$ for each $t\in(0,\infty)$ implies
${\mathcal{M}}u(\cdot,t)\leq{\mathcal{M}}\big({\mathcal{N}}u\big)$ in ${\mathbb{R}}^{n-1}$
for each $t\in(0,\infty)$, so the last condition in \eqref{Tafva.111} also follows from \eqref{Tafva}.
\end{proof}

It is clear that the Fatou-type result from Theorem~\ref{tuFatou} (cf. \eqref{Tafva}, in particular)
is valid in the class of null-solutions $u$ of $L$ for which ${\mathcal{N}}u$ belongs to
weighted Lebesgue spaces as in Example~2, variable exponent Lebesgue spaces as in Example~3,
weighted Lorentz spaces as in Example~4, as well as weighted Orlicz spaces as in Example~5.
Indeed, the discussion in \S\ref{S-1} shows that the Fatou type result from Theorem~\ref{tuFatou}
holds in the settings of Theorem~\ref{Theorem-Nice} and Theorem~\ref{corol:Xw-Dir}.
The case of ordinary Lebesgue spaces deserves special mention, and a precise statement,
which also includes the end-point case $p=1$, is presented below.

\begin{corollary}\label{tuFatou.Lp}
Assume the system $L$ is as in \eqref{L-def}-\eqref{L-ell.X}. Then for each $p\in[1,\infty)$,
\begin{align}\label{Tafva.Lp}
\left.
\begin{array}{r}
u\in{\mathscr{C}}^\infty({\mathbb{R}}^n_{+})
\\[4pt]
Lu=0\,\mbox{ in }\,{\mathbb{R}}^n_{+}
\\[6pt]
{\mathcal{N}}u\in L^p({\mathbb{R}}^{n-1})
\end{array}
\right\}
\
\Longrightarrow\ 
\left\{
\begin{array}{l}
u\big|^{{}^{\rm n.t.}}_{\partial{\mathbb{R}}^n_{+}}\,\mbox{ exists a.e.~in }\,
{\mathbb{R}}^{n-1},
\\[6pt]
\mbox{belongs to $L^p({\mathbb{R}}^{n-1})$,}
\\[4pt]
\mbox{and }\,u(x',t)=\Big(P^L_t\ast\big(u\big|^{{}^{\rm n.t.}}_{\partial{\mathbb{R}}^n_{+}}\big)\Big)(x'),
\\[4pt]
\mbox{for every }(x',t)\in{\mathbb{R}}^n_{+},
\end{array}
\right.
\end{align}
where $P^L$ is the Poisson kernel for $L$ in ${\mathbb{R}}^n_{+}$.
\end{corollary}

\begin{proof}
For $p\in(1,\infty)$, the desired conclusion follows directly from the implication
\eqref{Tafva}$\Rightarrow$\eqref{Tafva.2222} in Theorem~\ref{tuFatou},
the boundedness of the Hardy-Littlewood maximal operator on $L^p({\mathbb{R}}^{n-1})$,
and the fact that $L^p({\mathbb{R}}^{n-1})\subset
L^1\Big({\mathbb{R}}^{n-1}\,,\,\frac{1+\log_{+}|x'|}{1+|x'|^{n-1}}\,dx'\Big)$ by H\"older's
inequality. There remains to treat the case $p=1$, and this will follow from the implication
\eqref{Tafva.111}$\Rightarrow$\eqref{Tafva.2222} in Theorem~\ref{tuFatou}
as soon as we check that, under the current assumptions, ${\mathcal{M}}u(\cdot,t)\in
L^1\Big({\mathbb{R}}^{n-1}\,,\,\frac{1+\log_{+}|x'|}{1+|x'|^{n-1}}\,dx'\Big)$ for
every fixed $t\in(0,\infty)$. To this end, from interior estimates (cf. Theorem~\ref{ker-sbav})
we first deduce that $\|u(\cdot,t)\|_{L^\infty({\mathbb{R}}^{n-1})}
\leq C_{L,n,t}\|{\mathcal{N}}u\|_{L^1({\mathbb{R}}^{n-1})}$. Since we also have
$|u(\cdot,t)|\leq{\mathcal{N}}u\in L^1({\mathbb{R}}^{n-1})$, it ultimately follows that
$u(\cdot,t)\in L^\infty({\mathbb{R}}^{n-1})\cap L^1({\mathbb{R}}^{n-1})
\subset L^2({\mathbb{R}}^{n-1})$. Hence, given that ${\mathcal{M}}$
is bounded on $L^2({\mathbb{R}}^{n-1})$, we conclude that
${\mathcal{M}}u(\cdot,t)\in L^2({\mathbb{R}}^{n-1})\subset
L^1\Big({\mathbb{R}}^{n-1}\,,\,\frac{1+\log_{+}|x'|}{1+|x'|^{n-1}}\,dx'\Big)$, as wanted.
\end{proof}

\appendix
\section{Auxiliary results}
\setcounter{equation}{0}
\label{sect:Green}

We begin by recording a suitable version of the divergence theorem recently obtained in \cite{Div-MMM}.
To state it requires a few preliminaries which we dispense with first. As usual, let
$\mathcal{D}'({\mathbb{R}}^n_{+})$ denote the space of distributions in ${\mathbb{R}}^n_{+}$
and write $\mathcal{E}'({\mathbb{R}}^n_{+})$ for the space of distributions in
${\mathbb{R}}^n_{+}$ that are compactly supported. Hence,
\begin{equation}\label{TDanab.4r}
\mathcal{E}'({\mathbb{R}}^n_{+})\hookrightarrow\mathcal{D}'({\mathbb{R}}^n_{+})
\,\,\,\mbox{ and }\,\,\,
L^1_{\rm loc}({\mathbb{R}}^n_{+})\hookrightarrow\mathcal{D}'({\mathbb{R}}^n_{+}).
\end{equation}
For each compact set $K\subset{\mathbb{R}}^n_{+}$, define
$\mathcal{E}'_K({\mathbb{R}}^n_{+}):=\big\{u\in\mathcal{E}'({\mathbb{R}}^n_{+})
:\,{\rm supp}\,u\subset K\big\}$ and consider
\begin{multline}\label{TDY-87764}
\mathcal{E}'_K({\mathbb{R}}^n_{+})+L^1({\mathbb{R}}^n_{+})
:=\big\{u\in\mathcal{D}'({\mathbb{R}}^n_{+}):\,\exists\,v_1\in
\mathcal{E}'_K({\mathbb{R}}^n_{+})\,\mbox{ and }\,
\exists\,v_2\in L^1({\mathbb{R}}^n_{+})
\\[4pt]
\mbox{such that }\,
u=v_1+v_2\,\mbox{ in }\,\mathcal{D}'({\mathbb{R}}^n_{+})\big\}.
\end{multline}
Also, introduce $\mathscr{C}_b^\infty({\mathbb{R}}^n_{+}):=
\mathscr{C}^\infty({\mathbb{R}}^n_{+})\cap L^\infty({\mathbb{R}}^n_{+})$
and let $\big(\mathscr{C}_b^\infty({\mathbb{R}}^n_{+})\big)^\ast$ denote its
algebraic dual. Moreover, we let
${}_{(\mathscr{C}_b^\infty({\mathbb{R}}^n_{+}))^\ast}\big\langle\cdot\,,\,\cdot
\big\rangle_{\mathscr{C}_b^\infty({\mathbb{R}}^n_{+})}$ denote the natural duality pairing
between these spaces. It is useful to observe that
for every compact set $K\subset{\mathbb{R}}^n_{+}$ one has
\begin{equation}\label{TDY-877ii}
\mathcal{E}'_K({\mathbb{R}}^n_{+})+L^1({\mathbb{R}}^n_{+})\subset
\big(\mathscr{C}_b^\infty({\mathbb{R}}^n_{+})\big)^\ast.
\end{equation}

\begin{theorem}[\cite{Div-MMM}]\label{theor:div-thm}
Assume that $K\subset{\mathbb{R}}^n_{+}$ is a compact set and that
$\vec{F}\in L^1_{\rm loc}({\mathbb{R}}^n_{+},\mathbb{C}^n)$ is a vector
field satisfying the following conditions:
\begin{enumerate}\itemsep=0.2cm
\item [(a)] ${\rm div}\,\vec{F}\in\mathcal{E}'_K({\mathbb{R}}^n_{+})+L^1({\mathbb{R}}^n_{+})$,
where the divergence is taken in the sense of distributions;
\item[(b)] ${\mathcal{N}}_\kappa^{K^c}\vec{F}\in L^1({\mathbb{R}}^{n-1})$,
where $\kappa>0$ and $K^c:={\mathbb{R}}^n_{+}\setminus K$;
\item [(c)] there exists $\vec{F}\big|_{\partial{\mathbb{R}}^n_{+}}^{{}^{\rm n.t.}}$
a.e.~in ${\mathbb{R}}^{n-1}$.
\end{enumerate}

\noindent Then, with $e_n:=(0,\dots,0,1)\in{\mathbb{R}}^n$ and ``dot" denoting the
standard inner product in ${\mathbb{R}}^n$,
\begin{equation}\label{eqn:div-form}
{}_{(\mathscr{C}_b^\infty({\mathbb{R}}^n_{+}))^\ast}\big\langle{\rm div}\,\vec{F},1
\big\rangle_{\mathscr{C}_b^\infty({\mathbb{R}}^n_{+})}
=-\int_{{\mathbb{R}}^{n-1}}e_n\cdot
\bigl(\vec F\,\big|^{{}^{\rm n.t.}}_{\partial{\mathbb{R}}^n_{+}}\bigr)\,d{\mathscr{L}}^{n-1}.
\end{equation}
\end{theorem}

The theorem below summarizes properties of a distinguished fundamental solution
for constant (complex) coefficient, homogeneous systems. A proof of the present formulation
may be found in \cite[Theorem~11.1, pp.\,347-348]{DM} and \cite[Theorem~7.54, pp.\,270-271]{DM}
(cf.~also \cite{IMM} and the references therein). Below, $S^{n-1}$ is the unit sphere centered
at the origin in ${\mathbb{R}}^n$, $\sigma$ is its canonical surface measure, and
$\omega_{n-1}:=\sigma(S^{n-1})$ denotes its area.

\begin{theorem}\label{FS-prop}
Fix $n,m,M\in\mathbb{N}$ with $n\geq 2$, and consider an $M\times M$
system of homogeneous differential operators of order $2m$,
\begin{equation}\label{op-Liii}
{\mathfrak{L}}:=\sum_{|\alpha|=2m}A_{\alpha}\partial^\alpha,
\end{equation}
with matrix coefficients $A_{\alpha}\in{\mathbb{C}}^{M\times M}$.
Assume that ${\mathfrak{L}}$ satisfies the weak ellipticity condition
\begin{equation}\label{LH}
{\rm det}\big[{\mathfrak{L}}(\xi)\big]\not=0,
\qquad\forall\,\xi\in{\mathbb{R}}^n\setminus\{0\},
\end{equation}
where
\begin{equation}\label{LH.2}
{\mathfrak{L}}(\xi):=\sum_{|\alpha|=2m}\xi^\alpha A_{\alpha}\in{\mathbb{C}}^{M\times M},
\qquad\forall\,\xi\in{\mathbb{R}}^n.
\end{equation}
Then the $M\times M$ matrix $E$ defined at
each $x\in{\mathbb{R}}^{n}\setminus\{0\}$ by
\begin{equation}\label{Def-ES1-GLOB}
E(x):=\frac{1}{4(2\pi\,i)^{n-1}(2m-1)!}\Delta_x^{(n-1)/2}
\int_{S^{n-1}}(x\cdot\xi)^{2m-1}\,{\rm sgn}\,(x\cdot\xi)
\big[{\mathfrak{L}}(\xi)\big]^{-1}\,d\sigma(\xi)\quad
\end{equation}
if $n$ is odd, and
\begin{equation}\label{Def-ES2-GLOB}
E(x):=\frac{-1}{(2\pi\,i)^{n}(2m)!}\Delta_x^{n/2}\int_{S^{n-1}}
(x\cdot\xi)^{2m}\ln|x\cdot\xi|\big[{\mathfrak{L}}(\xi)\big]^{-1}\,d\sigma(\xi)
\end{equation}
if $n$ is even, satisfies the following properties.
\begin{list}{(\theenumi)}{\usecounter{enumi}\leftmargin=.7cm
\labelwidth=.7cm\itemsep=0.2cm\topsep=.2cm}
\item[(1)] Each entry in $E$ is a tempered distribution in ${\mathbb{R}}^n$,
and a real-analytic function in $\mathbb{R}^n\setminus\{0\}$ {\rm (}hence, in particular,
it belongs to ${\mathscr{C}}^\infty(\mathbb{R}^n\setminus\{0\})${\rm )}. Moreover,
\begin{equation}\label{smmth-odd}
E(-x)=E(x)\quad\mbox{for all }\quad x\in{\mathbb{R}}^n\setminus\{0\}.
\end{equation}
\item[(2)] If $I_{M\times M}$ is the $M\times M$ identity matrix, then for each $y\in{\mathbb{R}}^n$
\begin{equation}\label{fs-GLOB}
{\mathfrak{L}}_x\bigl[E(x-y)\bigr]=\delta_y(x)\,I_{M\times M}
\end{equation}
in the sense of tempered distributions in ${\mathbb{R}}^n$,
where the subscript $x$ denotes the fact that the operator ${\mathfrak{L}}$ in
\eqref{fs-GLOB} is applied to each column of $E$ in the variable $x$.
\item[(3)] Define the $M\times M$ matrix-valued function
\begin{equation}\label{Fm-PjkX}
{\mathcal{P}}(x):=\frac{-1}{(2\pi\,i)^n(2m-n)!}\int_{S^{n-1}}
(x\cdot\xi)^{2m-n}\big[{\mathfrak{L}}(\xi)\big]^{-1}\,d\sigma(\xi),\ \ \forall\,x\in{\mathbb{R}}^n.
\end{equation}
Then the entries of ${\mathcal{P}}$ are identically zero when either $n$ is odd or $n>2m$,
and are homogeneous polynomials of degree $2m-n$ when $n\leq 2m$. Moreover,
there exists a ${\mathbb{C}}^{M\times M}$-valued function $\Phi$, with entries in
${\mathscr{C}}^\infty(\mathbb{R}^n\setminus\{0\})$, that
is positive homogeneous of degree $2m-n$ such that
\begin{equation}\label{fs-str}
E(x)=\Phi(x)+\bigl(\ln|x|\bigr){\mathcal{P}}(x),\qquad \forall\,x\in\mathbb{R}^n\setminus\{0\}.
\end{equation}
\item[(4)] For each $\beta\in\mathbb{N}_0^n$ with $|\beta|\geq 2m-1$,
the restriction to ${\mathbb{R}}^n\setminus\{0\}$ of the matrix distribution
$\partial^\beta E$ is of class ${\mathscr{C}}^\infty$ and positive homogeneous
of degree $2m-n-|\beta|$.
\item[(5)] For each $\beta\in\mathbb{N}_0^n$ there exists $C_\beta\in(0,\infty)$
such that the estimate
\begin{equation}\label{fs-est}
|\partial^\beta E(x)|\leq\left\{
\begin{array}{l}
\displaystyle\frac{C_\beta}{|x|^{n-2m+|\beta|}}
\quad\mbox{ if either $n$ is odd, or $n>2m$, or if $|\beta|>2m-n$},
\\[16pt]
\displaystyle\frac{C_{\beta}(1+|\ln|x||)}{|x|^{n-2m+|\beta|}}
\quad\mbox{ if }\,\,0\leq |\beta|\leq 2m-n,
\end{array}\right.
\end{equation}
holds for each $x\in{\mathbb{R}}^n\setminus\{0\}$.
\item[(6)] When restricted to ${\mathbb{R}}^n\setminus\{0\}$, the entries of
$\widehat{E}$ {\rm (}with ``hat" denoting the Fourier transform{\rm )}
are ${\mathscr{C}}^\infty$ functions and, moreover,
\begin{equation}\label{E-ftXC}
\widehat{E}(\xi)=(-1)^m\bigl[{\mathfrak{L}}(\xi)\bigr]^{-1}
\quad\mbox{for each}\quad\xi\in{\mathbb{R}}^n\setminus\{0\}.
\end{equation}
\item[(7)] Writing $E^{\mathfrak{L}}$ in place of $E$ to emphasize
the dependence on ${\mathfrak{L}}$, the fundamental solution $E^{\mathfrak{L}}$ with
entries as in \eqref{Def-ES1-GLOB}-\eqref{Def-ES2-GLOB} satisfies
\begin{equation}\label{E-Trans}
\begin{array}{c}
\bigl(E^{\mathfrak{L}}\bigr)^\top=E^{{\mathfrak{L}}^\top},\quad
\overline{E^{\mathfrak{L}}}=E^{\overline{{\mathfrak{L}}}}\,,\quad
\bigl(E^{\mathfrak{L}}\bigr)^\ast=E^{{\mathfrak{L}}^\ast},
\\[8pt]
\mbox{and}\quad E^{\lambda{\mathfrak{L}}}=\lambda^{-1} E^{\mathfrak{L}}
\,\,\mbox{ for each }\,\,\lambda\in{\mathbb{C}}\setminus\{0\},
\end{array}
\end{equation}
where ${\mathfrak{L}}^\top$, $\overline{{\mathfrak{L}}}$, and
${\mathfrak{L}}^\ast=\overline{{\mathfrak{L}}}^\top$ denote the
transposed, the complex conjugate, and the Hermitian adjoint of ${\mathfrak{L}}$,
respectively.
\item[(8)] Any fundamental solution ${\mathbb{E}}$ of the system ${\mathfrak{L}}$ in ${\mathbb{R}}^n$,
whose entries are tempered distributions in ${\mathbb{R}}^n$, is of the form
${\mathbb{E}}=E+Q$ where $E$ is as in \eqref{Def-ES1-GLOB}-\eqref{Def-ES2-GLOB} and
$Q$ is an $M\times M$ matrix whose entries are polynomials in ${\mathbb{R}}^n$ and whose columns,
$Q_k$, $k\in\{1,\dots,M\}$, satisfy the pointwise equations
${\mathfrak{L}}\,Q_k=0\in{\mathbb{C}}^M$ in ${\mathbb{R}}^n$ for each $k\in\{1,\dots,M\}$.
\item[(9)] In the particular case when $M=1$ and $m=1$, i.e., in the situation when
${\mathfrak{L}}={\rm div}A\nabla$ for some matrix $A=(a_{jk})_{1\leq j,k\leq n}\in{\mathbb{C}}^{n\times n}$,
and when in place of \eqref{LH} the strong ellipticity condition
\begin{equation}\label{sec-or-a}
{\rm Re}\Bigg[\sum\limits_{j,k=1}^n a_{jk}\xi_j\xi_k\Bigg]\geq C|\xi|^2,
\qquad\forall\,\xi=(\xi_1,\dots,\xi_n)\in{\mathbb{R}}^n,
\end{equation}
is imposed, the fundamental solution $E$ of ${\mathfrak{L}}$ from
\eqref{Def-ES1-GLOB}-\eqref{Def-ES2-GLOB} takes the explicit form
\begin{equation}\label{YTcxb-ytSH}
E(x)=\left\{
\begin{array}{ll}
-\dfrac{1}{(n-2)\omega_{n-1}\sqrt{{\rm det}\,(A_{\rm sym})}}
\Big[\big((A_{\rm sym})^{-1}x\big)\cdot x\Big]^{\frac{2-n}{2}}
& \mbox{if }\,n\geq 3,
\\[18pt]
\dfrac{1}{4\pi\sqrt{{\rm det}\,(A_{\rm sym})}}
\log\Big[\big((A_{\rm sym})^{-1}x\big)\cdot x\Big] & \mbox{if }\,n=2.
\end{array}
\right.
\end{equation}
Here, $A_{\rm sym}:=\frac{1}{2}(A+A^\top)$ stands for the symmetric part of
the coefficient matrix $A=(a_{rs})_{1\leq r,s\leq n}$
and $\log$ denotes the principal branch of the complex logarithm function
{\rm (}defined by the requirement that $z^t=e^{t\log z}$ for all
$z\in{\mathbb{C}}\setminus(-\infty,0]$ and all $t\in{\mathbb{R}}${\rm )}.
\end{list}
\end{theorem}

Before introducing the notion of Green function we discuss several pieces of notation.
First, ${\rm diag}:=\{(x,x):\,x\in{\mathbb{R}}^n_{+}\}$ denotes the diagonal in the
Cartesian product ${\mathbb{R}}^n_{+}\times{\mathbb{R}}^n_{+}$.
Second, given a function $G(\cdot,\cdot)$ of two vector variables,
$(x,y)\in{\mathbb{R}}^n_{+}\times{\mathbb{R}}^n_{+}\setminus{\rm diag}$,
for each $k\in\{1,\dots,n\}$ we agree to write $\partial_{X_k} G$ and $\partial_{Y_k} G$,
respectively, for the partial derivative of $G$ with respect to $x_k$, and $y_k$
(the $k$-th components of $x$ and $y$, respectively).
This convention may be iterated, lending a natural meaning to
$\partial^\alpha_X\partial^\beta_Y G$, for each pair of multi-indices
$\alpha,\beta\in{\mathbb{N}}_0^n$. Also, we shall interpret $\nabla_X G$, and
$\nabla_Y G$, as the gradients of $G$ with respect to $x$, and $y$.
Third, for each point $y\in{\mathbb{R}}^n_{+}$ define
\begin{equation}\label{BnaGVBBB.a}
B_y:=B\big(y,\tfrac12\,{\rm dist}\,(y,\partial{\mathbb{R}}^n_{+})\big)
\,\,\mbox{ and, as usual, set }\,\,B_y^c:={\mathbb{R}}^n_{+}\setminus B_y.
\end{equation}

Given a function $u$ which is absolutely integrable over bounded subsets of ${\mathbb{R}}^n_{+}$,
define (whenever meaningful) the Sobolev trace as
\begin{equation}\label{Veri-S2TG.3}
\big({\rm Tr}\,u\big)(x'):=
\lim\limits_{r\to 0^{+}}\aver{B((x',0),\,r)\cap{\mathbb{R}}^n_{+}}
u\,d{\mathscr{L}}^n,\qquad x'\in\partial{\mathbb{R}}^{n-1}.
\end{equation}
For each $p\in(1,\infty)$ let $W^{1,p}({\mathbb{R}}^n_{+})$ be the classical $L^p$-based Sobolev space
of order one in ${\mathbb{R}}^n_{+}$, and use the symbol $\mathring{W}^{1,p}({\mathbb{R}}^n_{+})$
for the closure of ${\mathscr{C}}^\infty_0({\mathbb{R}}^n_{+})$ in $W^{1,p}({\mathbb{R}}^n_{+})$.
Then for each function $u\in W^{1,p}({\mathbb{R}}^n_{+})$, $1<p<\infty$, the trace ${\rm Tr}\,u$ exists
a.e.~on $\partial{\mathbb{R}}^n_{+}$ and belongs to $B^{p,p}_{1-1/p}({\mathbb{R}}^{n-1})$,
where for each $p\in(1,\infty)$ and $s\in(0,1)$ the Besov space $B^{p,p}_s({\mathbb{R}}^{n-1})$
is defined as the collection of all measurable functions $f$ in ${\mathbb{R}}^{n-1}$ with the
property that
\begin{equation}\label{Veri-S2TG.4}
\|f\|_{B^{p,p}_s({\mathbb{R}}^{n-1})}:=
\|f\|_{L^p({\mathbb{R}}^{n-1})}+\left(\int_{{\mathbb{R}}^{n-1}}
\int_{{\mathbb{R}}^{n-1}}\frac{|f(x')-f(y')|^p}{|x'-y'|^{n-1+sp}}\,dx'dy'\right)^{1/p}<\infty.
\end{equation}
In fact, for each $p\in(1,\infty)$ the operator
\begin{equation}\label{Veri-S2TG.5}
{\rm Tr}:W^{1,p}({\mathbb{R}}^n_{+})
\longrightarrow B^{p,p}_{1-1/p}({\mathbb{R}}^{n-1})
\end{equation}
is well-defined, linear and bounded, and has a linear and bounded right-inverse.

\begin{definition}\label{ta.av-GGG}
Let $L$ be a constant coefficient, second-order, elliptic differential operator
as in \eqref{L-def}. Call $G(\cdot,\cdot):{\mathbb{R}}^n_{+}\times{\mathbb{R}}^n_{+}
\setminus{\rm diag}\to{\mathbb{C}}^{M\times M}$ a {\tt Green function for $L$}
in ${\mathbb{R}}^n_{+}$ provided for each $y\in\mathbb{R}^n_{+}$ the following
properties hold:
\begin{align}\label{GHCewd-22.RRe}
& G(\cdot\,,y)\in L^1_{\rm loc}({\mathbb{R}}^n_{+}),
\\[4pt]
& G(\cdot\,,y)\big|^{{}^{\rm n.t.}}_{\partial{\mathbb{R}}^n_{+}}=0
\,\,\mbox{ a.e.~in }\,\,{\mathbb{R}}^{n-1},
\label{GHCewd-24.RRe}
\\[4pt]
& {\mathcal{N}}^{B^c_y}\,G(\cdot\,,y)\in\bigcup_{1<p<\infty}
L^p({\mathbb{R}}^{n-1}),
\label{GHCewd-25.RRe}
\\[4pt]
& L\big[G(\cdot\,,y)\big]=\delta_{y}\,I_{M\times M}\,\,\mbox{ in }\,\,
{\mathcal{D}}'({\mathbb{R}}^n_{+}),
\label{GHCewd-23.RRe}
\end{align}
where $L$ acts in the ``dot" variable on the columns of $G$.
\end{definition}

We remark that, in the context of Definition~\ref{ta.av-GGG}, we always have
\begin{equation}\label{GHCvCbN.2T}
G(\cdot\,,y)\in{\mathscr{C}}^\infty({\mathbb{R}}^n_{+}\setminus\{y\})
\,\,\,\mbox{ for each }\,\,y\in\mathbb{R}^n_{+},
\end{equation}
by \eqref{GHCewd-22.RRe}, \eqref{GHCewd-23.RRe}, and elliptic regularity
(cf. \cite[Theorem~10.9, p.\,318]{DM}). Other basic properties of the Green
function are collected in our next result.

\begin{theorem}[\cite{MaMiMiMi}]\label{ta.av-GGG.2A}
Assume that $L$ is a constant {\rm (}complex{\rm )} coefficient, second-order,
elliptic differential operator as in \eqref{L-def}. Then there exists a unique
Green function $G(\cdot,\cdot)=G^L(\cdot,\cdot)$ for $L$ in $\mathbb{R}^n_{+}$,
in the sense of Definition~\ref{ta.av-GGG}. Moreover, this Green function
also satisfies the following additional properties:
\begin{list}{(\theenumi)}{\usecounter{enumi}\leftmargin=.8cm
\labelwidth=.8cm\itemsep=0.2cm\topsep=.2cm}
\item[(1)] Given $\kappa>0$, for each $y\in{\mathbb{R}}^n_{+}$
and each compact neighborhood $K$ of $y$ in ${\mathbb{R}}^n_{+}$ there exists
a finite constant $C=C(n,L,\kappa,K,y)>0$ such that for every $x'\in\mathbb{R}^{n-1}$ there hold
\begin{equation}\label{bound-NK-G}
\mathcal{N}^{K^c}_\kappa\big(G(\cdot,y)\big)(x')\leq
C\,\frac{1+\log_{+}|x'|}{1+|x'|^{n-1}}
\end{equation}
{\rm (}if the fundamental solution $E^L$ of $L$ from Theorem~\ref{FS-prop} is
a radial function in ${\mathbb{R}}^n\setminus\{0\}$, then the logarithm in
\eqref{bound-NK-G} may actually be omitted{\rm )}.
Moreover, for any multi-indices $\alpha,\beta\in\mathbb{N}_0^n$ such that
$|\alpha|+|\beta|>0$, there exists $C=C(n,L,\kappa,\alpha,\beta,K,y)\in(0,\infty)$
such that
\begin{equation}\label{bouMNN}
\mathcal{N}^{K^c}_\kappa\big((\partial^\alpha_X \partial^\beta_Y G)(\cdot,y)\big)(x')\le
\frac{C}{1+|x'|^{n-2+|\alpha|+|\beta|}}.
\end{equation}
In particular,
\begin{equation}\label{GHCewd-22.RRe.4}
{\mathcal{N}}_\kappa^{K^c}\big((\partial_X^\alpha\partial_Y^\beta G^L)(\cdot\,,y)\big)
\in\bigcap_{1<p\leq\infty}L^p({\mathbb{R}}^{n-1}),
\qquad\forall\,\alpha,\beta\in{\mathbb{N}}_0^n.
\end{equation}
\item[(2)] For each fixed $y\in{\mathbb{R}}^n_{+}$, there holds
\begin{equation}\label{Fvabbb-7tF}
G^L(\cdot\,,y)\in{\mathscr{C}}^\infty
\big(\overline{{\mathbb{R}}^n_{+}}\setminus B(y,\varepsilon)\big)
\,\,\mbox{ for every }\,\,\varepsilon>0.
\end{equation}
\item[(3)] The function $G^L(\cdot,\cdot)$ is translation invariant in
the tangential variables in the sense that
\begin{equation}\label{JKBvc-ut4}
\begin{array}{c}
G^L\big(x-(z',0),y-(z',0)\big)=G^L(x,y)\,\,\mbox{ for every}
\\[8pt]
(x,y)\in{\mathbb{R}}^n_{+}\times{\mathbb{R}}^n_{+}
\setminus{\rm diag}\,\,\mbox{ and }\,\,z'\in{\mathbb{R}}^{n-1}.
\end{array}
\end{equation}
\item[(4)] With ${\rm Tr}$ denoting the Sobolev trace on $\partial{\mathbb{R}}^n_{+}$
(cf.~\eqref{Veri-S2TG.3}-\eqref{Veri-S2TG.5}), one has
\begin{equation}\label{Aivb-gVV}
\begin{array}{c}
G^L(\cdot\,,y)\in\bigcap\limits_{k\in{\mathbb{N}}}
\bigcap\limits_{\frac{n}{n-1}<p<\infty}W^{k,p}({\mathbb{R}}^n_{+}\setminus K)
\,\,\mbox{ and }\,\,{\rm Tr}\,\big[G^L(\cdot\,,y)\big]=0,
\\[10pt]
\mbox{for every $y\in{\mathbb{R}}^n_{+}$ and any compact
$K\subset{\mathbb{R}}^n_{+}$ with $y\in K^\circ$}.
\end{array}
\end{equation}
\item[(5)] If $G^{L^\top}\!\!(\cdot,\cdot)$ denotes the {\rm (}unique, by the first part of
the statement{\rm )} Green function for $L^\top$ in ${\mathbb{R}}^n_{+}$, then
\begin{equation}\label{GHCewd-22.RRe.5}
G^L(x,y)=\Big[G^{L^\top}\!\!(y,x)\Big]^\top,\qquad\forall\,(x,y)\in
{\mathbb{R}}^n_{+}\times{\mathbb{R}}^n_{+}\setminus{\rm diag}.
\end{equation}
Hence, as a consequence of \eqref{GHCewd-22.RRe.5}, \eqref{GHCewd-24.RRe}, and
\eqref{Fvabbb-7tF}, for each fixed $x\in{\mathbb{R}}^n_{+}$ and $\varepsilon>0$,
\begin{equation}\label{ghagUGDS}
G^L(x,\cdot)\in{\mathscr{C}}^\infty
\big(\overline{{\mathbb{R}}^n_{+}}\setminus B(x,\varepsilon)\big)
\,\,\,\mbox{ and }\,\,\,G^L(x,\cdot)\Big|_{\partial{\mathbb{R}}^n_{+}}=0
\,\,\mbox{ on }\,\,{\mathbb{R}}^{n-1}.
\end{equation}
\item[(6)] If $E^L$ denotes the fundamental solution of $L$
from Theorem~\ref{FS-prop}, then the matrix-valued function
\begin{equation}\label{defRRR}
R_L(x,y):=E^L(x-y)-G^L(x,y),\qquad\forall\,(x,y)\in{\mathbb{R}}^n_{+}
\times{\mathbb{R}}^n_{+}\setminus{\rm diag},
\end{equation}
extends to a function $R_L(\cdot,\cdot)\in{\mathscr{C}}^\infty\big({\mathbb{R}}^n_{+}
\times{\mathbb{R}}^n_{+}\big)$ which satisfies the following estimate: for any
multi-indices $\alpha,\beta\in\mathbb{N}_0^n$ there exists a finite
constant $C_{\alpha\beta}>0$ with the property that
for every $(x,y)\in{\mathbb{R}}^n_{+}\times{\mathbb{R}}^n_{+}$,
\begin{equation}\label{mainest}
\big|\big(\partial^\alpha_X\partial^\beta_Y R_L\big)(x,y)\big|\leq\left\{
\begin{array}{ll}
C_{\alpha\beta}\,|x-\overline{y}|^{2-n-|\alpha|-|\beta|} & \mbox{if $|\alpha|+|\beta|>0$, or $n\geq 3$},
\\[10pt]
C+C\big|\!\ln|x-\overline{y}|\big| & \mbox{if $|\alpha|=|\beta|=0$ and $n=2$},
\end{array}
\right.
\end{equation}
where $C\in(0,\infty)$, and $\overline{y}:=(y',-y_n)$ if
$y=(y',y_n)\in\mathbb{R}^n_{+}$.

\vskip 0.08in
\item[(7)] For any multi-indices $\alpha,\beta\in\mathbb{N}_0^n$
there exists a finite constant $C_{\alpha\beta}>0$ such that
\begin{equation}\label{mainest2G}
\begin{array}{c}
\big|\big(\partial^\alpha_X\partial^\beta_Y G^L\big)(x,y)\big|
\leq C_{\alpha\beta}|x-y|^{2-n-|\alpha|-|\beta|},
\\[10pt]
\mbox{$\forall\,(x,y)\in{\mathbb{R}}^n_{+}\times{\mathbb{R}}^n_{+}\setminus{\rm diag}$,
if either $n\geq 3$, or $|\alpha|+|\beta|>0$},
\end{array}
\end{equation}
and, corresponding to $|\alpha|=|\beta|=0$ and $n=2$, there exists
$C\in(0,\infty)$ such that
\begin{equation}\label{maKnaTTGB}
\big|G^L(x,y)\big|\leq C+C\big|\!\ln|x-\overline{y}|\big|,
\quad\forall\,(x,y)\in{\mathbb{R}}^2_{+}\times{\mathbb{R}}^2_{+}\setminus{\rm diag}.
\end{equation}
\item[(8)] For each $\alpha,\beta\in\mathbb{N}_0^n$ one has
\begin{equation}\label{mainest3G}
\begin{array}{c}
\sup\limits_{y\in{\mathbb{R}}^n_{+}}\big\|\big(\partial^\alpha_X
\partial^\beta_Y G^L\big)(\cdot,y)\big\|_{L^{\frac{n}{n-2+|\alpha|+|\beta|},\,\infty}
({\mathbb{R}}^n_{+})}<+\infty,
\\[10pt]
\mbox{if either $n\geq 3$, or $|\alpha|+|\beta|>0$}.
\end{array}
\end{equation}
In particular,
\begin{equation}
\label{mainest2G.1}
G^L(\cdot,y)\in L^{\frac{n}{n-2},\,\infty}({\mathbb{R}}^n_{+}),
\,\,\mbox{ uniformly in }\,\,y\in{\mathbb{R}}^n_{+},\,\,\mbox{ if }\,\,n\geq 3,
\end{equation}
\begin{equation}
\nabla_X G^L(\cdot,y),\nabla_Y G^L(\cdot,y)
\in L^{\frac{n}{n-1},\,\infty}({\mathbb{R}}^n_{+}),
\,\,\mbox{ uniformly in }\,\,y\in{\mathbb{R}}^n_{+},
\label{mainest2G.2}
\end{equation}
and
\begin{equation}
\begin{array}{c}
 \nabla^2_X G^L(\cdot,y),\nabla_X\nabla_Y G^L(\cdot,y),\,\, \mbox{ and }\,\,
\nabla^2_Y G^L(\cdot,y)\\[10pt]
\mbox{belong to } L^{1,\infty}({\mathbb{R}}^n_{+}),\,\,
\mbox{ uniformly in }\,\,y\in{\mathbb{R}}^n_{+}.\end{array}
\label{mainest2G.3}
\end{equation}
\item[(9)] If $p\in\big[1,\frac{n}{n-1}\big)$,
then for each $\zeta\in{\mathscr{C}}^\infty_0({\mathbb{R}}^n)$ one has
\begin{equation}\label{mainest4G}
\begin{array}{c}
\zeta G^L(\cdot,y)\in\mathring{W}^{1,p}({\mathbb{R}}^n_{+})
\,\,\mbox{ for each }\,\,y\in{\mathbb{R}}^n_{+}
\\[10pt]
\mbox{ and }\,\,\sup\limits_{y\in{\mathbb{R}}^n_{+}}
\big\|\zeta G^L(\cdot,y)\big\|_{W^{1,p}({\mathbb{R}}^n_{+})}<\infty.
\end{array}
\end{equation}
\item[(10)] If the fundamental solution $E^L$ for $L$ from Theorem~\ref{FS-prop}
is a radial function in ${\mathbb{R}}^n\setminus\{0\}$, then
{\rm (}with $\overline{y}\in{\mathbb{R}}^n_{-}$ denoting the
reflection of $y\in{\mathbb{R}}^n_{+}$ across $\partial{\mathbb{R}}^n_{+}${\rm )}
\begin{equation}\label{JKBvc-ut4.YFav}
G^L(x,y)=E^L(x-y)-E^L(x-\overline{y}),\qquad
\forall\,(x,y)\in{\mathbb{R}}^n_{+}\times{\mathbb{R}}^n_{+}\setminus{\rm diag}.
\end{equation}
\item[(11)] If $n\geq 3$, then for every $x=(x',t)\in{\mathbb{R}}^n_{+}$
and every $y\in{\mathbb{R}}^n_{+}\setminus\{x\}$ one has {\rm (}with $P^L$ denoting
the Agmon-Douglis-Nirenberg Poisson kernel for $L$ in ${\mathbb{R}}^n_{+}$
from Theorem~\ref{ya-T4-fav}{\rm )}
\begin{equation}\label{GHCewd-2PiK}
G^L(x,y)=E^L(x-y)-
P^L_t\ast\Big(\big[E^L(\cdot-y)\big]\big|_{\partial{\mathbb{R}}^n_{+}}\Bigr)(x'),
\end{equation}
with the convolution applied to each column of the matrix inside
the round parentheses.

\vskip 0.08in
\item[(12)] The Agmon-Douglis-Nirenberg Poisson kernel
$P^L=\big(P^L_{\gamma\alpha}\big)_{1\leq\gamma,\alpha\leq M}$ for $L$ in ${\mathbb{R}}^n_{+}$
from Theorem~\ref{ya-T4-fav} is related to the Green function $G^L$ for $L$ in ${\mathbb{R}}^n_{+}$
according to the formula
\begin{equation}\label{Ua-eD.LBBw.3B}
\begin{array}{c}
P^L_{\gamma\alpha}(z')=-a^{\beta\alpha}_{nn}
\big(\partial_{Y_n} G^{L}_{\gamma\beta}\big)\big((z',1),0\big),
\quad\forall\,z'\in{\mathbb{R}}^{n-1},
\\[6pt]
\mbox{for each }\,\,\alpha,\gamma\in\{1,\dots,M\}.
\end{array}
\end{equation}
In particular, formulas \eqref{Ua-eD.LBBw.3B} and \eqref{JKBvc-ut4.YFav} imply that whenever the
fundamental solution $E^L=\big(E^L_{\gamma\beta}\big)_{1\leq\gamma,\beta\leq M}$
of $L$ from Theorem~\ref{FS-prop} is a radial function then for
each $\alpha,\gamma\in\{1,\dots,M\}$ one has
\begin{equation}\label{Ua-eD.kab}
P^L_{\gamma\alpha}(z')=2a^{\beta\alpha}_{nn}(\partial_n E^L_{\gamma\beta})(z',1),
\qquad\forall\,z'\in{\mathbb{R}}^{n-1}.
\end{equation}
\end{list}
\end{theorem}

We shall now record the following versatile version of interior estimates for
higher-order elliptic systems. A proof may be found in \cite[Theorem~11.9, p.\,364]{DM}.

\begin{theorem}\label{ker-sbav}
Consider a homogeneous, constant coefficient, higher-order system ${\mathfrak{L}}$
as in \eqref{op-Liii}, satisfying the weak ellipticity condition \eqref{LH}.
Then for each null-solution $u$ of ${\mathfrak{L}}$ in a ball $B(x,R)$
{\rm (}where $x\in{\mathbb{R}}^n$ and $R>0${\rm )}, $0<p<\infty$, $\lambda\in(0,1)$,
$\ell\in{\mathbb{N}}_0$, and $0<r<R$, one has
\begin{equation}\label{detraz}
\sup_{z\in B(x,\lambda r)}|\nabla^\ell u(z)|
\leq\frac{C}{r^\ell}\left(\aver{B(x,r)}|u|^p\,d{\mathscr{L}}^n\right)^{1/p},
\end{equation}
where $C=C(L,p,\ell,\lambda,n)>0$ is a finite constant.
\end{theorem}

Finally, we discuss the dependence of the size of the nontangential maximal function,
corresponding to various apertures, in weighted $L^p$ spaces.

\begin{proposition}\label{prop:cones-Lpw}
For every $\kappa,\kappa'>0$, $p\in(0,\infty)$ and $w\in A_\infty({\mathbb{R}}^{n-1})$, there exist
finite constants $C_0,C_1>0$ such that
\begin{equation}\label{N-Nal:w}
C_0\|{\mathcal{N}}_\kappa u\|_{L^p({\mathbb{R}}^{n-1},\,w)}
\leq\|{\mathcal{N}}_{\kappa'}\,u\|_{L^p({\mathbb{R}}^{n-1},\,w)}
\leq C_1\|{\mathcal{N}}_\kappa u\|_{L^p({\mathbb{R}}^{n-1},\,w)},
\end{equation}
for each function $u:{\mathbb{R}}^n_{+}\to{\mathbb{C}}$.
\end{proposition}

\begin{proof}
As in the unweighted case, the proof of this result is based on a point-of-density argument.
Fix $\lambda>0$ and for every $\kappa>0$ write
\begin{equation}\label{treee-6}
O_{\kappa}=\big\{x'\in\mathbb{R}^{n-1}:\,\big(\mathcal{N}_{\kappa}u\big)(x')>\lambda\big\}.
\end{equation}
It is easy to show that $O_{\kappa}$ is open. Pick $0<\gamma<1$ so that
$1-\big(\kappa/(\kappa+\kappa')\big)^{n-1}<\gamma<1$ and write
$A_\kappa:=\mathbb{R}^{n-1}\setminus O_\kappa$. Also, for every $\gamma\in(0,1)$ introduce
\begin{equation}\label{treee-7}
A_{\kappa}^\gamma:=\Big\{x'\in\mathbb{R}^{n-1}:\,
\frac{|A_\kappa\cap B_{n-1}(x',r)|}{|B_{n-1}(x',r)|}\geq\gamma\,\,\mbox{ for each }\,\,r>0\Big\}.
\end{equation}
We are going to show that $O_{\kappa'}\subset\mathbb{R}^{n-1}\setminus A_{\kappa}^\gamma$.
Given $x'\in O_{\kappa'}$, we can take $(y',t)\in\Gamma_{\kappa'}(x')$ such that $|u(y',t)|>\lambda$.
Note that $B_{n-1}(y',\kappa\,t)\subset B_{n-1}\big(x',(\kappa+\kappa')\,t\big)$.
On the other hand, we have that $B_{n-1}(y',\kappa\,t)\subset O_{\kappa}$: if
$z'\in B_{n-1}(y',\kappa\,t)$ then $(y',t)\in\Gamma_{\kappa}(z')$ and, therefore,
$\mathcal{N}_\kappa u(z')\geq|u(y',t)|>\lambda$. All these show that
$B_{n-1}(y',\kappa\,t)\subset O_\kappa\cap B_{n-1}\big(x',(\kappa+\kappa')\,t\big)$.
This implies that
\begin{multline}\label{treee-8}
\frac{\big|B_{n-1}\big(x',(\kappa+\kappa')\,t\big)\cap A_{\kappa}\big|}
{\big|B_{n-1}\big(x',(\kappa+\kappa')\,t\big)\big|}
=1-\frac{\big|B_{n-1}\big(x',(\kappa+\kappa')\,t\big)\cap O_\kappa\big|}
{\big|B_{n-1}\big(x',(\kappa+\kappa')\,t\big)\big|}
\\[4pt]
\leq 1-\frac{|B_{n-1}(y',\kappa\,t)|}{\big|B_{n-1}\big(x',(\kappa+\kappa')\,t\big)\big|}
=1-\Bigl(\frac{\kappa}{\kappa+\kappa'}\Bigr)^{n-1}<\gamma,
\end{multline}
which forces $x'\notin A_{\kappa}^\gamma$ in light of \eqref{treee-7}.
In turn, this shows that
\begin{equation}\label{treee-9}
O_{\kappa'}\subset\mathbb{R}^{n-1}\setminus A_{\kappa}^\gamma
\subseteq\big\{x'\in\mathbb{R}^{n-1}:\,\mathcal{M}({\bf 1}_{O_\kappa})(x')
\geq c_n(1-\gamma)\big\},
\end{equation}
for some dimensional constant $c_n\in(0,\infty)$ (whose appearance is due to the fact that
the Hardy-Littlewood maximal operator has been defined in \eqref{MMax} using cubes rather
than balls). Since $w\in A_\infty({\mathbb{R}}^{n-1})$, we can take $q\in(1,\infty)$ such that
$w\in A_q({\mathbb{R}}^{n-1})$. Thus, $\mathcal{M}$ is bounded on $L^q({\mathbb{R}}^{n-1},w)$
and, consequently,
\begin{multline}\label{treee-10}
w(O_{\kappa'})
 \leq w\Big(\big\{x'\in\mathbb{R}^{n-1}:\,\mathcal{M}({\bf 1}_{O_\kappa})(x')
\geq c_n(1-\gamma)\big\}\Big)
\\[4pt]
 \leq [c_n(1-\gamma)]^{-q}\,\|\mathcal{M}({\bf 1}_{O_\kappa})\|_{L^q({\mathbb{R}}^{n-1},\,w)}^q
\leq C\,w(O_{\kappa}),
\end{multline}
where $C\in(0,\infty)$ depends only on $n,\kappa,\kappa',q,w$.
The level set estimate just derived readily yields \eqref{N-Nal:w}.
\end{proof}

It follows from Proposition~\ref{prop:cones-Lpw} and \eqref{NT-Fct.23PPPP} that,
for every $\kappa,\kappa'>0$ and $p\in(0,\infty)$, there exist finite
constants $C_0,C_1>0$ such that
\begin{align}\label{N-Nal.11a}
C_0\|{\mathcal{N}}^E_\kappa u\|_{L^p({\mathbb{R}}^{n-1})}
\leq\|{\mathcal{N}}^E_{\kappa'}\,&u\|_{L^p({\mathbb{R}}^{n-1})}
\leq C_1\|{\mathcal{N}}^E_\kappa u\|_{L^p({\mathbb{R}}^{n-1})},
\\[4pt]
C_0\|{\mathcal{N}}^{(\varepsilon)}_\kappa u\|_{L^p({\mathbb{R}}^{n-1})}
\leq\|{\mathcal{N}}^{(\varepsilon)}_{\kappa'}\,&u\|_{L^p({\mathbb{R}}^{n-1})}
\leq C_1\|{\mathcal{N}}^{(\varepsilon)}_\kappa u\|_{L^p({\mathbb{R}}^{n-1})},
\end{align}
for each function $u$, set $E\subset{\mathbb{R}}^n$, and number $\varepsilon>0$.



{\makeatletter
\renewcommand*\@makefnmark{}
\footnotetext{\noindent The first author has been supported in part by MINECO Grant MTM2010-16518,
ICMAT Severo Ochoa project SEV-2011-0087. He also acknowledges that
the research leading to these results has received funding from the European Research
Council under the European Union's Seventh Framework Programme (FP7/2007-2013)/ ERC
agreement no. 615112 HAPDEGMT. The second author has been supported in part by a
Simons Foundation grant $\#\,$200750, the third author has
been supported in part by US NSF grant $\#\,$0547944, while the fourth author has been
supported in part by the Simons Foundation grant $\#\,$281566, and by a University of
Missouri Research Leave grant. This work has been possible thanks to the support and hospitality
of \textit{Temple University} (USA), \textit{University of Missouri} (USA), and
\textit{ICMAT, Consejo Superior de Investigaciones Cient{\'\i}ficas} (Spain).
The authors express their gratitude to these institutions.}
\makeatother}

\begin{thebibliography}{99}
\parskip=0.7pt







\newcommand{\RMIauthor}[1]{{\sc #1:}}
\newcommand{\RMIpaper}[1]{{#1.}}
\newcommand{\RMIbook}[1]{{\em #1.}}
\newcommand{\RMIjournal}[1]{{\em #1}}


\bibitem{ADNI} 
\RMIauthor{Agmon, S., Douglis, A., and Nirenberg, L.}
\RMIpaper{Estimates near the
boundary for solutions of elliptic partial differential equations satisfying general
boundary conditions, I}
\RMIjournal{Comm. Pure Appl. Math.} 
\textbf{12} (1959), 623--727.

\bibitem{ADNII} 
\RMIauthor{Agmon, S., Douglis, A., and Nirenberg, L.}
\RMIpaper{Estimates near the boundary for solutions of elliptic partial differential equations satisfying general
boundary conditions, II}
\RMIjournal{Comm. Pure Appl. Math.} 
\textbf{17} (1964), 35--92.

\bibitem{AAAHK} 
\RMIauthor{Alfonseca, M.A., Auscher, P., Axelsson, A., Hofmann, S., and Kim, S.}
\RMIpaper{Analyticity of layer potentials and $L^2$ solvability of boundary value problems for 
divergence form elliptic equations with complex $L^\infty$ coefficients}
\RMIjournal{Adv. Math.} 
\textbf{226} (2011), no. 5, 4533--4606.

\bibitem{AM} 
\RMIauthor{Alvarado, R. and Mitrea, M.}
\RMIbook{Hardy Spaces on Ahlfors-Regular Quasi-Metric Spaces. A Sharp Theory}
Lecture Notes in Mathematics, Vol. 2142, Springer, 2015.

\bibitem{ABR}
\RMIauthor{Axler, S., Bourdon, P., and Ramey, W.} 
\RMIbook{Harmonic Function Theory}
2nd edition, Graduate Texts in Mathematics, Vol. 137, Springer-Verlag, New York, 2001.

\bibitem{bennett-sharpley88} 
\RMIauthor{Bennett, C. and Sharpley, R.} 
\RMIbook{Interpolation of operators}
volume 129 of Pure and Applied Mathematics, Academic Press Inc., Boston, MA, 1988.

\bibitem{Beu} 
\RMIauthor{Beurling, A.} 
\RMIpaper{Construction and analysis of some convolution algebras}
\RMIjournal{Ann. Inst. Fourier (Grenoble)} 
\textbf{14} (1964), 1--32.

\bibitem{CL}
\RMIauthor{Chen, Y.Z. And Lau, K.S.} 
\RMIpaper{Some new classes of Hardy spaces}
\RMIjournal{J. Funct. Anal.} 
\textbf{84} (1989), 255--278.

\bibitem{CF} 
\RMIauthor{Chiarenza, F. and Frasca, M.} 
\RMIpaper{Morrey spaces and Hardy-Littlewood maximal function}
\RMIjournal{Rend. Mat. Appl.} 
\textbf{7} (1987), no. 3-4, 273--279.

\bibitem{CU-F} 
\RMIauthor{Cruz-Uribe, D. and Fiorenza, A.} 
\RMIbook{Variable Lebesgue Spaces: Foundations and Harmonic Analysis}
Birkhauser, Applied and Numerical Harmonic Analysis, 2013.

\bibitem{cruz-uribe-fiorenza-neugebauer03}
\RMIauthor{Cruz-Uribe, D., Fiorenza, A., and Neugebauer, C.J.} 
\RMIpaper{The maximal function on variable $L^p$ spaces}
\RMIjournal{Ann. Acad. Sci. Fenn. Math.} 
\textbf{28} (2003), 223--238.

\bibitem{CMP} 
\RMIauthor{Cruz-Uribe, D., Martell, J.M., and P\'erez, C.} 
\RMIbook{Weights, extrapolation and the theory of Rubio de Francia}
Operator Theory: Advances and Applications, Vol. 215, Birkh\"auser/Springer Basel AG, Basel, 2011.

\bibitem{curbera-garcia-cuerva-martell-perez06}
\RMIauthor{Curbera, G., Garc{\'i}a-Cuerva, J., Martell, J.M., and Per\'ez, C.} 
\RMIpaper{Extrapolation with weights, rearrangement-invariant function spaces, modular inequalities and applications to
singular integrals}
\RMIjournal{Adv. Math.} 
\textbf{203} (2006), 256--318.

\bibitem{diening05} 
\RMIauthor{Diening, L.} 
\RMIpaper{Maximal function on Musielak-Orlicz spaces and generalized Lebesgue spaces}
\RMIjournal{Bull. Sci. Math.}
\textbf{129} (2005), 657--700.

\bibitem{DHHR} 
\RMIauthor{Diening, L., Harjulehto,  P.,  H\"ast\"o, P., and R{\r{u}}{\v{z}}i{\v{c}}ka, M.} 
\RMIbook{Lebesgue and Sobolev spaces with variable exponents} 
Lecture Notes in Mathematics, Vol. 2017. Springer, Heidelberg, 2011.

\bibitem{GC} 
\RMIauthor{ Garc{\'\i}a-Cuerva, J.} 
\RMIpaper{Hardy spaces and Beurling algebras}
\RMIjournal{J. Lond. Math. Soc. (2)}
\textbf{39} (1989), 499--513.

\bibitem{GCRF85} 
\RMIauthor{Garc{\'\i}a-Cuerva, J.  and  Rubio de Francia, J.L.} 
\RMIbook{Weighted Norm Inequalities and Related Topics}
North Holland, Amsterdam, 1985.

\bibitem{He} 
\RMIauthor{Helms, L.L.} 
\RMIbook{Introduction to Potential Theory}
Wiley-Interscience, 1969.

\bibitem{HofMitMor} 
\RMIauthor{Hofmann, S., Mitrea, M., and Morris, A.} 
\RMIpaper{The method of layer potentials in $L^p$ and endpoint spaces for elliptic operators with $L^\infty$ coefficients}
\RMIjournal{Proc. Lond. Mat. Soc. (3)}
\textbf{111} (2015), no. 3, 681--716.


\bibitem{Ke94} 
\RMIauthor{Kenig, C.E.} 
\RMIbook{Harmonic Analysis Techniques for Second Order Elliptic  Boundary Value Problems}
CBMS Regional Conference Series in Mathematics, Vol. 83, 
Amer. Math. Soc., Providence, RI, 1994.

\bibitem{KMR2} 
\RMIauthor{Kozlov, V.A., Maz'ya, V.G., and Rossmann, J.} 
\RMIbook{Spectral Problems Associated with Corner Singularities of Solutions to Elliptic
Equations}
AMS, 2001.

\bibitem{LT} 
\RMIauthor{Lindenstrauss, J.  and Tzafriri, L.} 
\RMIbook{Classical Banach Spaces, I and II}
Springer-Verlag, Berlin Heidelberg, New York, 1977.

\bibitem{LionsMagenes} 
\RMIauthor{Lions, J.L. and Magenes, E.}
\RMIbook{Non-Homogeneous Boundary Value Problems  and Applications}
Die Grundlehren der mathematischen Wissenschaften 181,
Springer, Berlin, Heidelberg, 1972.

\bibitem{MaMiMiMi} 
\RMIauthor{Martell, J.M., Mitrea, D., Mitrea, I., and Mitrea, M.}
\RMIpaper{Poisson kernels and boundary problems for elliptic
systems in the upper-half space} 
Preprint, 2016.

\bibitem{MazShap} 
\RMIauthor{Maz'ya, V.G. and Shaposhnikova, T.O.} 
\RMIbook{Theory of Multipliers in Spaces of Differentiable Functions}
Monographs and Studies in Mathematics, Vol. 23, 
Pitman Advanced Publishing Program, Boston, MA, 1985.

\bibitem{DM} 
\RMIauthor{Mitrea, D.} 
\RMIbook{Distributions, Partial Differential Equations, and Harmonic Analysis}
Universitext, Springer, 2013.

\bibitem{Div-MMM} 
\RMIauthor{Mitrea, D., Mitrea, I. and Mitrea, M.} 
\RMIpaper{A sharp divergence theorem with nontangential pointwise traces and applications} 
Preprint, 2014.

\bibitem{MMMZ} 
\RMIauthor{Mitrea, D., Mitrea, I., Mitrea, M., and Ziad\'e, E.} 
\RMIpaper{Abstract capacitary estimates and the completeness and separability of certain classes of non-locally convex
topological vector spaces}
\RMIjournal{J. Funct. Anal.}
\textbf{262} (2012) 4766--4830.

\bibitem{IMM} 
\RMIauthor{Mitrea, I. and Mitrea, M.} 
\RMIbook{Multi-Layer Potentials and Boundary Problems for Higher-Order Elliptic Systems in Lipschitz Domains}
Lecture Notes in Mathematics, Vol. 2063, Springer, 2013.

\bibitem{nekvinda04} 
\RMIauthor{Nekvinda, A.} 
\RMIpaper{Hardy-Littlewood maximal operator on $L^{p(x)}(\mathbb R)$}
\RMIjournal{Math. Inequal. Appl.} 
\textbf{7} (2004), 255--265.

\bibitem{Shen} 
\RMIauthor{Shen, Z.} 
\RMIpaper{Boundary value problems in Morrey spaces for elliptic systems on Lipschitz domains}
\RMIjournal{Amer. J. Math.} 
\textbf{125} (2003), no. 5, 1079--1115.

\bibitem{Si88} 
\RMIauthor{Siegel, D.} 
\RMIpaper{The Dirichlet problem in a half-space and a new  Phragmen-Lindel\"of principle}, pp. 208-217 in 
\RMIbook{Maximum Principles and Eigenvalue  Problems in Partial Differential Equations}
P.W. Schaefer, ed., Longman, Harlow, 1988.

\bibitem{ST96} 
\RMIauthor{Siegel, D. and Talvila, E.O.} 
\RMIpaper{Uniqueness for the $n$-dimensional half space Dirichlet problem}
\RMIjournal{Pacific J. Math.} 
\textbf{175} (1996), no. 2, 571--587.

\bibitem{Sol} 
\RMIauthor{Solonnikov, V.A.} 
\RMIpaper{Bounds for the solutions of general boundary value problems for elliptic systems} 
\RMIjournal{Doklady Akad. Nauk. SSSR} 
\textbf{151} (1963), 783--785 (Russian). English translation in 
\RMIjournal{Soviet Math.} 
4 (1963), 1089--1091.

\bibitem{Sol1} 
\RMIauthor{Solonnikov, V.A.}  
\RMIpaper{General boundary value problems for systems elliptic in the sense of A. Douglis and L. Nirenberg. I}
(Russian) \RMIjournal{Izv. Akad. Nauk SSSR, Ser. Mat.} 
\textbf{28} (1964), 665--706.

\bibitem{Sol2} 
\RMIauthor{Solonnikov, V.A.} 
\RMIpaper{General boundary value problems for systems elliptic in the sense of A. Douglis and L. Nirenberg. II}, (Rusian)
\RMIjournal{Trudy Mat. Inst. Steklov}
\textbf{92} (1966), 233--297.

\bibitem{St70} 
\RMIauthor{Stein, E.M.} 
\RMIbook{Singular Integrals and Differentiability Properties of Functions}
Princeton Mathematical Series, No. 30,
Princeton University Press, Princeton, NJ, 1970.

\bibitem{Stein93} 
\RMIauthor{Stein, E.M.} 
\RMIbook{Harmonic Analysis: Real-Variable Methods,
Orthogonality, and Oscillatory Integrals}
Princeton Mathematical Series, Vol. 43, Monographs in Harmonic Analysis, III, Princeton University Press,
Princeton, NJ, 1993.

\bibitem{StWe71} 
\RMIauthor{Stein E.M., and Weiss, G.}
\RMIbook{Introduction to Fourier Analysis on Euclidean Spaces}
Princeton Mathematical Series, Princeton University Press, 
Princeton, NJ, 1971.

\bibitem{Taylor} 
\RMIauthor{Taylor, M.E.}
\RMIbook{Partial Differential Equations}
Second edition,  Applied Mathematical Sciences, Springer, New York, 2011.

\bibitem{Yo96} 
\RMIauthor{Yoshida, H.}
\RMIpaper{A type of uniqueness for the Dirichlet problem on a half-space with continuous data}
\RMIjournal{Pacific J. Math.} 
\textbf{172} (1996), no. 2, 591--909.

\bibitem{Wolf41} 
\RMIauthor{Wolf, F.} 
\RMIpaper{The Poisson integral. A study in the uniqueness of harmonic functions}
\RMIjournal{Acta Math.} 
\textbf{74} (1941), 65--100.

\bibitem{Zaa} 
\RMIauthor{Zaanen, A.C.}
\RMIbook{Integration}
North-Holland Publishing Company, Amsterdam, 1967.

\end{thebibliography}
\end{document}